\documentclass[fontsize=11pt,english,reqno]{amsart} 
\usepackage{ucs}
\usepackage{ae}
\usepackage{microtype}
\usepackage[utf8]{inputenc}
\usepackage[T1]{fontenc}
\usepackage[english]{babel}

\usepackage{amsmath}
\usepackage{amssymb}
\usepackage{amstext}
\usepackage{amsthm}
\usepackage{thmtools}

\usepackage{stmaryrd}

\usepackage{enumitem}

\usepackage{geometry}
\usepackage{xcolor}

\newcommand{\EDIT}[1]{{\color{black}#1}}
\newcommand{\EEDIT}[1]{{\color{black}#1}}

\usepackage{mathtools}
\usepackage{slashed}

\usepackage{mathscinet}
\usepackage[numbers]{natbib}

\usepackage{graphicx}
\usepackage{tikz-cd}
\usepackage{caption}

\usepackage{imakeidx}
\makeindex[intoc]

\usepackage{tabularx}

\usepackage[colorlinks=true,hyperindex, linkcolor=magenta, urlcolor=black, pagebackref=false, citecolor=cyan,pdfpagelabels]{hyperref}
\usepackage[capitalize]{cleveref}

\newif\ifcomment
\newcommand{\comment}[1]{\ifcomment#1\fi}

\newlength{\currentparindent}
\makeatletter
\newcommand{\@minipagerestore}{\setlength{\parindent}{\currentparindent}}
\makeatother

\newcommand{\nospacepunct}[1]{\makebox[0pt][l]{#1}}

\usepackage{relsize}
\usepackage[bbgreekl]{mathbbol}
\usepackage{amsfonts}

\DeclareSymbolFontAlphabet{\mathbb}{AMSb}
\DeclareSymbolFontAlphabet{\mathbbl}{bbold}
\newcommand{\Prism}{{\mathlarger{\mathbbl{\Delta}}}}
\newcommand{\prism}{\mathbbl{\Delta}}

\DeclareMathOperator{\Ker}{Ker}

\DeclareMathOperator{\id}{id}

\DeclareMathOperator{\Aut}{Aut}

\DeclareMathOperator{\Hom}{Hom}
\DeclareMathOperator{\Ext}{Ext}
\DeclareMathOperator{\Spf}{Spf}
\DeclareMathOperator{\Spec}{Spec}
\DeclareMathOperator{\Spa}{Spa}

\DeclareMathOperator{\Cone}{Cone}
\DeclareMathOperator{\Mod}{Mod}
\DeclareMathOperator{\Perf}{Perf}
\DeclareMathOperator{\gr}{gr}
\DeclareMathOperator{\Fil}{Fil}
\DeclareMathOperator{\Rep}{Rep}
\DeclareMathOperator{\Fun}{Fun}
\DeclareMathOperator{\MF}{MF}

\DeclareMathOperator{\Vect}{Vect}
\DeclareMathOperator{\Coh}{Coh}
\DeclareMathOperator{\Loc}{Loc}

\DeclareMathOperator{\fib}{fib}
\DeclareMathOperator{\cofib}{cofib}
\DeclareMathOperator{\RHom}{RHom}

\DeclareMathOperator{\Map}{Map}
\DeclareMathOperator{\Rees}{Rees}

\DeclareMathOperator{\FGauge}{F-Gauge}

\DeclareMathOperator{\Stab}{Stab}

\DeclareMathOperator{\Ind}{Ind}
\DeclareMathOperator{\Sym}{Sym}
\DeclareMathOperator{\Tot}{Tot}
\DeclareMathOperator{\Crys}{Crys}
\DeclareMathOperator{\Isoc}{Isoc}

\newcommand{\colim}{\operatornamewithlimits{colim}}

\newcommand{\sRHom}{\underline{\mathrm{RHom}}}

\newcommand{\sEnd}{\underline{\mathrm{End}}}
\newcommand{\Der}{\mathcal{D}\kern -.5pt er}

\newcommand{\tensorL}{\otimes^\mathbb{L}}

\newcommand\A{\mathbb{A}}
\newcommand\F{\mathbb{F}}
\newcommand\G{\mathbb{G}}
\newcommand\Q{\mathbb{Q}}
\newcommand\V{\mathbb{V}}
\newcommand\Z{\mathbb{Z}}

\newcommand\T{\mathbb{T}}

\newcommand\D{\mathcal{D}}
\newcommand\DF{\mathcal{DF}}

\let\O\cO

\let\H\calH

\let\T\calT

\newcommand\dHod{\slashed{D}}

\newcommand\crys{\mathrm{crys}}
\newcommand\dR{\mathrm{dR}}
\newcommand\N{\mathrm{N}}
\newcommand\Syn{\mathrm{Syn}}
\newcommand\et{\mathrm{\acute{e}t}}
\newcommand\proet{\mathrm{pro\acute{e}t}}
\newcommand\refl{\mathrm{refl}}
\newcommand\adm{\mathrm{adm}}
\newcommand\op{\mathrm{op}}

\let\inf\Ainf
\newcommand\HT{\mathrm{HT}}
\newcommand\lisse{\mathrm{lisse}}
\newcommand\FM{\mathrm{FM}}
\newcommand\red{\mathrm{red}}
\newcommand\Hod{\mathrm{Hod}}

\newcommand\can{\mathrm{can}}

\newcommand\typ{\mathrm{typ}}

\newcommand\nilp{\mathrm{nilp}}

\newcommand\an{\mathrm{an}}

\let\epsilon\varepsilon
\let\phi\varphi
\let\ol\overline
\let\ul\underline
\let\tensor\otimes

\let\cal\mathcal

\newtheorem{thm}{Theorem}[subsection]
\newtheorem{prop}[thm]{Proposition}
\newtheorem{lem}[thm]{Lemma}
\newtheorem{cor}[thm]{Corollary}

\theoremstyle{definition}
\newtheorem{defi}[thm]{Definition}

\newtheorem{rem}[thm]{Remark}

\newtheorem{exx}[thm]{Example}
\newenvironment{ex}
  {\pushQED{\qed}\exx}
  {\popQED\endexx}

\numberwithin{equation}{subsection}

\date{\today}

\makeatletter
\renewcommand{\address}[1]{\gdef\@address{#1}}
\renewcommand{\email}[1]{\gdef\@email{\url{#1}}}
\newcommand{\@endstuff}{\par\vspace{\baselineskip}\noindent\small
\begin{tabular}{@{}l}\scshape\@address\\\textit{E-mail address:} \@email\end{tabular}}
\AtEndDocument{\@endstuff}
\makeatother

\title[An étale-crystalline comparison with coefficients]{An arithmetic étale-crystalline comparison \\ with coefficients in crystalline local systems}
\author{Maximilian Hauck}
\address{Max-Planck-Institut f\"ur Mathematik, Vivatsgasse 7, 53111 Bonn, Germany}
\email{max.hauck01@gmail.com}

\begin{document}

\begin{abstract}
We use the stacky approach to $p$-adic cohomology theories recently developed by Drinfeld and Bhatt--Lurie to generalise a comparison theorem between the rational crystalline cohomology of the special fibre and the rational $p$-adic étale cohomology of the arithmetic generic fibre of any proper $p$-adic formal scheme $X$ due to Colmez--Nizio{\l} to the case of coefficients in an arbitrary crystalline local system on the generic fibre of $X$. In the process, we establish a version of the Beilinson fibre square of Antieau--Mathew--Morrow--Nikolaus with coefficients in the proper case and prove a comparison between syntomic cohomology and $p$-adic étale cohomology with coefficients in an arbitrary $F$-gauge. Our methods also yield a description of the isogeny category of perfect $F$-gauges on $\Z_p$.
\end{abstract}

\maketitle

\tableofcontents

\section{Introduction}

While the idea that one may compute the value of a cohomology theory attached to a scheme $X$ by instead computing the coherent cohomology of a suitably defined stack attached to $X$ goes back, in the case of de Rham cohomology, to work of Simpson in the 1990s, see \cite{Simpson} and \cite{Simpson2}, this approach has only recently found entrance into the field of $p$-adic Hodge theory and has been starting to be fully exploited in the course of the last few years with the formulation of prismatic cohomology in terms of stacks independently developed by Bhatt--Lurie and Drinfeld in \cite{APC}, \cite{PFS}, \cite{FGauges} and \cite{Prismatization}. Roughly speaking, \EDIT{similarly to how one can attach to any $p$-adic formal scheme $X$ its \emph{de Rham stack} $X^\dR$, which has the property that coherent cohomology of the structure sheaf $\O_{X^\dR}$ agrees with the ($p$-completed) de Rham cohomology of $X$ if $X$ is smooth,} they functorially attach a stack $X^\prism$ to any \EDIT{such} $X$ with the feature that coherent cohomology of the structure sheaf $\O_{X^\prism}$ agrees with the (absolute) prismatic cohomology of $X$ \EEDIT{if $X$ is $p$-quasisyntomic in the sense of \cite[Def. C.6]{APC}}; correspondingly, the stack $X^\prism$ is called the \emph{prismatisation} of $X$. Moreover, similarly to how the de Rham stack $X^\dR$ admits a filtered refinement $X^{\dR, +}$ over $\A^1/\G_m$ computing the Hodge filtration, they also introduce a filtered refinement $X^\N$ over $\A^1/\G_m$ of $X^\prism$ computing the (absolute) Nygaard filtration on prismatic cohomology. Finally, there is an even more refined variant $X^\Syn$ of $X^\N$ called the \emph{syntomification} of $X$, which arises by gluing together two copies of $X^\prism$ which sit inside $X^\N$ as open substacks.

The upshot of this picture is twofold: First, various statements about prismatic cohomology and related cohomology theories now admit a ``geometric'' formulation; \EDIT{for example}, the comparison between prismatic cohomology and de Rham cohomology from \cite[Thm. 5.4.2]{APC} can be reinterpreted as saying that, for any smooth $p$-adic formal scheme $X$, there is a functorial isomorphism
\begin{equation*}
(X_{p=0})^\prism\cong X^\dR\;.
\end{equation*}
Second, the stacky formulation immediately furnishes natural categories of coefficients for the respective cohomology theories: these should just be the categories of quasi-coherent complexes, or perhaps perfect complexes, on the corresponding stack. 

For us, the most important such category of will be the category of \emph{$F$-gauges} on $X$ defined by
\begin{equation*}
\FGauge_\prism(X)\coloneqq \D(X^\Syn)\;.
\end{equation*}
The fact that (perfect) $F$-gauges play a central role in the theory is due to the fact that they seem to be some sort of ``universal coefficients'' for $p$-adic cohomology theories, \EDIT{i.e.\ they appear to be good candidates for a theory of $p$-adic motives}: namely, they admit various realisation functors towards more classical notions of coefficients for $p$-adic cohomology theories, e.g.\ a \emph{de Rham realisation}
\begin{equation*}
T_\dR: \Vect(X^\Syn)\rightarrow \Vect(X^\dR)\cong \Vect^{\nabla}(X)
\end{equation*}
towards the category of vector bundles on $X$ equipped with a flat connection; a filtered refinement
\begin{equation*}
T_{\dR, +}: \Vect(X^\Syn)\rightarrow \Vect(X^{\dR, +})\;,
\end{equation*}
where the right-hand side may roughly be described as filtered vector bundles on $X$ equipped with a flat connection satisfying Griffiths transversality, \EEDIT{see \cref{rem:fildrstack-vect}}; an \emph{étale realisation} 
\begin{equation*}
T_\et: \Vect(X^\Syn)\rightarrow \Loc_{\Z_p}(X_\eta)
\end{equation*}
towards the category of pro-étale $\Z_p$-local systems on the generic fibre $X_\eta$ of $X$; \EDIT{and a \emph{crystalline realisation}
\begin{equation*}
T_\crys: \Vect(X^\Syn)\rightarrow \Vect((X_{p=0})^\Syn)\;,
\end{equation*}
which may be postcomposed with a natural functor
\begin{equation*}
\Vect((X_{p=0})^\Syn)\rightarrow \Isoc^\phi(X_{p=0}/\Z_p)
\end{equation*}
towards the category of \EDIT{$F$-isocrystals} on the special fibre $X_{p=0}$ of $X$.} The details of this picture as well as the necessary background on the stacks $X^\dR, X^\prism$, etc.\ will be reviewed in Section \cref{sect:stacks}.

In this paper, we make use of these features of the stacky approach to $p$-adic cohomology theories and, in particular, syntomic cohomology to prove the following comparison theorem between the crystalline cohomology of the special fibre of a smooth proper $p$-adic formal scheme $X$ and the étale cohomology of its generic fibre. For the relevant definitions about crystalline local systems, we refer to Section \ref{subsect:review-cryslocsys}.

\begin{thm}[\cref{thm:cryset-main}]
\label{thm:intro-cryset}
Let $X$ be a $p$-adic formal scheme \EEDIT{which is smooth and proper over $\Spf\Z_p$}. For any crystalline local system $L$ on $X_\eta$ with Hodge--Tate weights all at most $-i-1$ for some $i\geq 0$, let $\cal{E}$ be its associated $F$-isocrystal. Then there is a natural morphism
\begin{equation*}
R\Gamma_\crys(X_{p=0}, \cal{E})[\tfrac{1}{p}][-1]\rightarrow R\Gamma_\proet(X_\eta, L)[\tfrac{1}{p}]\;
\end{equation*}
which induces an isomorphism
\begin{equation*}
\tau^{\leq i} R\Gamma_\crys(X_{p=0}, \cal{E})[\tfrac{1}{p}][-1]\cong\tau^{\leq i} R\Gamma_\proet(X_\eta, L)[\tfrac{1}{p}]
\end{equation*}
and an injection on $H^{i+1}$.
\end{thm}

This generalises a previous result of Colmez--Nizio{\l} who obtained essentially the same statement in the special case of $L=\Z_p(n)$ being a Tate twist, see \cite[Cor. 1.4]{ColmezNiziol}. Our strategy to prove this will be to first construct an $F$-gauge corresponding to any crystalline \EDIT{local system} building on the works \cite{GuoReinecke} of Guo--Reinecke and \cite{GuoLi} of Guo--Li; more precisely, in Section \ref{subsect:locsystofgauge}, we prove that:

\begin{thm}[\cref{thm:cryset-locsysfgauges}]
\label{thm:intro-locsysfgauges}
Let $X$ be a smooth $p$-adic formal scheme. Then there is a fully faithful embedding
\begin{equation*}
\Loc^\crys_{\Z_p}(X_\eta)\hookrightarrow \Perf(X^\Syn)
\end{equation*}
which is a weak right inverse to étale realisation. The image of a crystalline local system $L$ under this embedding is called the \emph{associated $F$-gauge} $E$ of $L$ and has the property that $T_\dR(E)$ identifies with the associated $F$-isocrystal of $L$.
\end{thm}

Then the proof is divided into two parts: On the one hand, we show that, in the range which is relevant to us, étale cohomology with coefficients in a crystalline local system agrees with syntomic cohomology with coefficients in the associated $F$-gauge. More precisely, in Section \ref{sect:syntomicetale}, we obtain the following theorem:

\begin{thm}[\cref{thm:syntomicetale-mainfine}]
\label{thm:intro-syntomicetale}
Let $X$ be a smooth qcqs $p$-adic formal scheme. For any vector bundle $F$-gauge $E\in\Vect(X^\Syn)$ with Hodge--Tate weights all at most $-i-1$ for some $i\geq 0$, the natural morphism
\begin{equation*}
R\Gamma(X^\Syn, E)\rightarrow R\Gamma_\proet(X_\eta, T_\et(E))
\end{equation*}
induces an isomorphism
\begin{equation*}
\tau^{\leq i} R\Gamma(X^\Syn, E)\cong\tau^{\leq i} R\Gamma_\proet(X_\eta, T_\et(E))
\end{equation*}
and an injection on $H^{i+1}$.
\end{thm}

On the other hand, in order to make the connection between syntomic cohomology and crystalline cohomology, we use Section \ref{sect:beilfibsq} to prove the following version of the Beilinson fibre square from \cite{BeilFibSq} with coefficients, which recovers the previously known result \cite[Thm. 6.17]{BeilFibSq} of Antieau--Mathew--Morrow--Nikolaus in the case of smooth and proper $p$-adic formal schemes upon plugging in a Breuil--Kisin twist $\O\{n\}$ for the $F$-gauge $E$:

\begin{thm}[\cref{cor:beilfibsq-coeffs}]
\label{thm:intro-beilfibsq}
Let $X$ be a $p$-adic formal scheme which is smooth and proper over $\Spf\Z_p$. For any perfect $F$-gauge $E\in\Perf(X^\Syn)$, there is a natural pullback square
\begin{equation*}
\begin{tikzcd}
R\Gamma(X^\Syn, E)[\frac{1}{p}]\ar[r]\ar[d] & R\Gamma((X_{p=0})^\Syn, T_\crys(E))[\frac{1}{p}]\ar[d] \\
R\Gamma(X^{\dR, +}, T_{\dR, +}(E))[\frac{1}{p}]\ar[r] & R\Gamma(X^\dR, T_\dR(E))[\frac{1}{p}]
\end{tikzcd}
\end{equation*}
\EDIT{in the category $\D(\Q_p)$.}
\end{thm}

We note here that this really also gives a new proof of the special case $E=\O\{n\}$ treated in \cite[Thm. 6.17]{BeilFibSq}: Indeed, while our argument uses results from \cite{GuoReinecke}, which in turn are proved using the Beilinson fibre square in the form from \cite[Thm. 6.17]{BeilFibSq}, Du--Liu--Moon--Shimizu have shown in \cite{DuLiuMoonShimizu} that one can also obtain the same results via a different route which does not make use of the Beilinson fibre square, see also \cite[Rem. 1.9]{GuoReinecke}. In particular, in contrast to the proof of Antieau--Mathew--Morrow--Nikolaus, our proof does not make use of topological Hochschild homology and only relies on methods from the theory of prismatic cohomology. Let us also remark here that, \EDIT{during our work on the contents of this paper, we learnt that,} in currently still unpublished work, Lurie has recently found yet another proof of the Beilinson fibre square, which also makes use of the stacky formulation of syntomic cohomology, but still significantly differs from our proof. In particular, his proof should allow one to drop the properness hypothesis in the above result and one can then also remove the properness hypothesis from our main theorem.

Moreover, we also show that the pullback square from \cref{thm:intro-beilfibsq} actually comes from a commutative square of stacks 
\begin{equation}
\label{eq:intro-beilfibsq}
\begin{tikzcd}
\Z_p^\dR\ar[r]\ar[d] & \Z_p^{\dR, +}\ar[d] \\
\F_p^\Syn\ar[r] & \Z_p^\Syn\nospacepunct{\;,}
\end{tikzcd}
\end{equation}
which gives a new perspective on the Beilinson fibre square as a statement about the geometry of the stack $\Z_p^\Syn$: Indeed, one may reinterpret \cref{thm:intro-beilfibsq} as saying that, at least from the perspective of cohomology after inverting $p$, the stack $\Z_p^\Syn$ \EDIT{behaves like it} is glued together from the stacks $\F_p^\Syn$ and $\Z_p^{\dR, +}$. Indeed, we also prove that this interpretation is reflected on the level of derived categories:

\begin{thm}[\cref{thm:beilfibsq-categorical}]
\label{thm:intro-beilfibsqcat}
The functor
\begin{equation}
\label{eq:intro-beilfibsqcat}
\Perf(\Z_p^\Syn)[\tfrac{1}{p}]\rightarrow \DF(\Q_p)\times_{\D(\Q_p)} \D(\Mod^\phi(\Q_p))
\end{equation}
induced by the diagram (\ref{eq:intro-beilfibsq}) is a fully faithful embedding. Here, $\DF(\Q_p)$ denotes the filtered derived category of $\Q_p$-vector spaces and $\Mod^\phi(\Q_p)$ is the category of $\phi$-modules over $\Q_p$.
\end{thm}

\EDIT{In particular, note that the above is a categorical generalisation of \cref{thm:intro-beilfibsq}: indeed, for any $E\in\Perf(\Z_p^\Syn)$, one recovers the statement of \cref{thm:intro-beilfibsq} by computing $\RHom(\O, E)[\tfrac{1}{p}]$ via the right-hand side of (\ref{eq:intro-beilfibsqcat}).}

\subsection*{Notations and Conventions}

We freely make use of the language of $\infty$-categories in the style of Lurie, see \cite{HTT}, and of the theory of derived algebraic geometry as laid out in \cite{DAG}. In particular, we work derived throughout: e.g., all our pullbacks and pushforwards are in the derived sense, i.e.\ when we write $f_*$ for a map $f: \cal{X}\rightarrow\cal{Y}$ of schemes/formal schemes/stacks, we really mean the derived pushforward $Rf_*: \cal{D}(\cal{X})\rightarrow\cal{D}(\cal{Y})$.

All the stacks occurring in this paper are going to be in the fpqc topology on commutative rings and we point out that, for the purposes of this paper, it does not make a difference whether one works in the setting of stacks in groupoids or stacks in $\infty$-groupoids. By a quasi-coherent complex on a stack $\cal{X}$, we mean an object of the derived category $\D(\cal{X})$, which is defined via Kan extension from the affine case as in \cite[§3.2]{DAG}; the same applies to the full subcategories $\Vect(\cal{X})$ and $\Perf(\cal{X})$ of vector bundles and perfect complexes, respectively.

When we speak of ``completeness'' of a module or complex with respect to some ideal, we will usually mean derived completeness as defined, for example, in \cite[Tag 091N]{Stacks} and distinguish any other usage of the term ``complete'' by speaking about ``classical completeness''. \EEDIT{For a ring $A$ and an ideal $I$,} we will use $\widehat{\D}(A)$ to denote the category of derived $I$-complete complexes of $A$-modules; \EEDIT{here, we omit the ideal $I$ from the notation as it will generally be clear from the context.} \EEDIT{Moreover, in the aforementioned situation, we denote the derived $I$-completion of a complex $M$ of $A$-modules by $M_I^\wedge$.}

We frequently make use of the Rees equivalence between the category of quasi-coherent complexes on $\A^1/\G_m$ over a ring $R$ and the category $\DF(R)$ of filtered objects in $\D(R)$, see \cite[§2.2.1]{FGauges} or \cite{Moulinos} for an introduction to the Rees equivalence. However, beware that our sign convention slightly differs from the one in \cite{FGauges}: we use the stack $\A^1_-/\G_m$, where the subscript $(-)_-$ indicates that $\G_m$ actson $\A^1$ by placing the coordinate $t$ on $\A^1$ in grading degree $-1$, to realise the Rees equivalence for decreasing filtrations and the universal generalised Cartier divisor on $\A^1_-/\G_m$ is denoted $t: \O(1)\rightarrow\O$; this has the pleasant effect of removing the change of sign in the passage to the associated graded in \cite[Prop.\ 2.2.6.(3)]{FGauges}. Similarly, the notation $\A^1_+/\G_m$ indicates the quotient of $\A^1$ by the $\G_m$-action given by placing the coordinate on $\A^1$, which will be called $u$ this time, in grading degree $1$ and the universal section of $\A^1_+/\G_m$ is denoted $u: \O\rightarrow\O(1)$; in our sign convention, sheaves on $\A^1_+/\G_m$ correspond to increasing filtrations.

From time to time, we need base change statements for cartesian squares of the form
\begin{equation*}
\begin{tikzcd}
X^?\ar[r]\ar[d] & X^{??}\ar[d] \\
Y^?\ar[r] & Y^{??}
\end{tikzcd}
\end{equation*}
for $?, ??\in\{\dR, \prism, \N, \dots\}$ which are induced by a map $X\rightarrow Y$ of formal schemes. We will usually use these without further justification and refer to Appendix \ref{sect:basechange} for details regarding how to prove such results.

Throughout, $p$ is a fixed prime and $X$ denotes a bounded $p$-adic formal scheme. If $X=\Spf R$ is affine, we also use the notation $R^\dR, R^\prism, \dots$ to denote the stacks $(\Spf R)^\dR, (\Spf R)^\prism, \dots$. 

We also fix an algebraic closure $\ol{\Q}_p$ of $\Q_p$, \EEDIT{whose completion we will denote by $C$.}

\subsection*{Acknowledgements}

Most of the results in this paper first appeared in my master's thesis and I heartily want to thank my advisor Guido Bosco for his continued support, for many long and fruitful discussions, his constant willingness to answer all of my questions and for lots of helpful comments on an earlier version of the material presented here. I also thank Bhargav Bhatt for suggesting that one may prove \cref{prop:beilfibsq-perfzpsyncrys} via a cohomology computation. This paper was prepared during my time as a PhD student at the Max Planck Institute for Mathematics in Bonn and I wish to thank the the institute for its hospitality.

\section{Recollections on stacks and $p$-adic cohomology theories}
\label{sect:stacks}

We briefly remind the reader of the essential input from \cite{Prismatization}, \cite{APC}, \cite{PFS} and \cite{FGauges} needed for our purposes. For a more thorough introduction, we advise the reader to consult these sources directly. Throughout, the idea will be that, given a cohomology theory $R\Gamma_\typ(-)$ for $p$-adic formal schemes $X$, we will functorially attach a stack $X^\typ$ to $X$ with the property that, in good cases, coherent cohomology on the stack $X^\typ$ computes $R\Gamma_\typ(X)$, i.e.\
\begin{equation*}
R\Gamma(X^\typ, \O_{X^\typ})\cong R\Gamma_\typ(X)\;.
\end{equation*}

\subsection{Preliminaries on stacks}

Before we can start talking about cohomology theories, let us make a few preliminary remarks about stacks in general. 

\subsubsection{Formal stacks}
We will almost exclusively deal with \emph{($p$-adic) formal stacks}, i.e.\ stacks $\cal{X}$ such that $\cal{X}(S)$ is only nonempty if $S$ is $p$-nilpotent. Note that any adic ring $R$ with an ideal of definition $I\subseteq R$ containing $p$ defines a formal stack $\Spf R$ via $\Spf R\coloneqq \colim_n \Spec R/I^n$ and, if $R$ is classically $I$-complete, then $(\Spf R)(S)$ identifies with the groupoid of continuous morphisms from $R$ to the discrete ring $S$. Clearly, any stack $\cal{X}$ yields a formal stack $\cal{X}\times\Spf\Z_p$; as we will soon work with formal stacks exclusively, we will almost always suppress the product with $\Spf\Z_p$ from the notation and restrict to $p$-nilpotent rings $S$ as our test objects.

In this setting of formal stacks, we make the following slightly nonstandard definition:

\begin{defi}
Let $\cal{P}$ be a property of formal schemes which is fpqc-local. Then we say that a formal stack $\cal{X}$ has property $\cal{P}$ if there is an fpqc cover from an affine formal scheme $\Spf R\rightarrow\cal{X}$ such that $\Spf R$ has property $\cal{P}$.
\end{defi}

We remark that, for a ring $R$ with bounded $p^\infty$-torsion endowed with the $p$-adic topology, the category $\D(\Spf R)$ identifies with the category $\widehat{\D}(R)$ of derived $p$-complete objects in $\D(R)$, see \cite[Prop.\ A.11]{DeltaRings}. Moreover, with the following definition, we can also recover the usual $t$-structure on $\widehat{\D}(R)$ by \cite[Ex.\ A.18]{DeltaRings}:

\begin{defi}
Let $\cal{X}$ be a stack. An object $F\in\D(\cal{X})$ is called \emph{connective} if $f^*F\in\D(S)$ is connective with respect to the standard $t$-structure for all $f: \Spec S\rightarrow\cal{X}$. The full subcategory of $\D(\cal{X})$ spanned by connective objects is denoted $\D^{\leq 0}(\cal{X})$.
\end{defi}

\begin{rem}
Using \cite[§1]{Aisles}, one can check that $\D^{\leq 0}(\cal{X})$ indeed defines the connective part of a $t$-structure $(\D^{\leq 0}(\cal{X}), \D^{\geq 0}(\cal{X}))$ on the category $\D(\cal{X})$.
\end{rem}

The Rees equivalence also applies to our context of stacks over $\Spf\Z_p$ (instead of over $\Spec R$ for some ring $R$). Indeed, one can show:

\begin{lem}
For any stack $\cal{X}$, there are symmetric monoidal equivalences
\begin{equation*}
\begin{split}
\D(B\G_m\times\cal{X})&\cong\D_{\gr}(\cal{X})\coloneqq\Fun(\Z, \D(\cal{X})) \\
\D(\A^1_-/\G_m\times\cal{X})&\cong\DF(\cal{X})\coloneqq\Fun((\Z, \geq), \D(\cal{X}))\;.
\end{split}
\end{equation*}
\end{lem}
\begin{proof}
We only prove the first equivalence, the second one is similar. As any morphism $\Spec S\rightarrow B\G_m\times\cal{X}$ factors as $\Spec S\rightarrow B\G_m\times\Spec S\rightarrow B\G_m\times \cal{X}$, we obtain equivalences
\begin{equation*}
\begin{split}
\D(B\G_m\times\cal{X})&\cong\lim_{\Spec S\rightarrow\cal{X}} \D(B\G_m\times \Spec S)\cong \lim_{\Spec S\rightarrow\cal{X}} \Fun(\Z, \D(S)) \\
&\cong \Fun(\Z, \lim_{\Spec S\rightarrow\cal{X}} \D(S))\cong \Fun(\Z, \D(\cal{X})) 
\end{split}
\end{equation*}
using the Rees equivalence over the base $\Spec S$ and the universal property of the limit.
\end{proof}

In particular, we see that $\D(B\G_m\times\Spf\Z_p)$ and $\D(\A^1_-/\G_m\times\Spf\Z_p)$ identify with the categories $\widehat{\DF}(\Z_p)$ of filtered objects and $\widehat{\D}_{\gr}(\Z_p)$ of graded objects in $\widehat{\D}(\Z_p)$, respectively.

\subsubsection{Vector bundles and PD hulls}

For a projective $R$-module $E$ of finite rank, we adopt the convention that the geometric vector bundle $\V(E)\rightarrow\Spec R$ corresponding to $E$ is given by $\V(E)\coloneqq \Spec\Sym^\bullet_R E^\vee$. More generally, for a stack $\cal{X}$ and a vector bundle $E\in\Vect(\cal{X})$, we define $\V(E)$ by
\begin{equation*}
\V(E)(S)=\{\text{pairs $f: \Spec S\rightarrow\cal{X}, f^*E^\vee\rightarrow\O_{\Spec S})$}\}^{\cong}\;.
\end{equation*}

Particularly in the context of de Rham stacks, we will often consider the PD hull $\V(E)^\sharp$ of the zero section of $\V(E)$. On points, this is given by
\begin{equation*}
\begin{split}
\V(E)^\sharp(S)=\{&\text{tuples $(f: \Spec S\rightarrow\cal{X}, s_n: f^*E^{\tensor n\vee}\rightarrow\O_{\Spec S})$} \\
&\hspace{0.5cm}\text{such that $(s_n)_n$ is a system of divided powers for $s_1$}\}^{\cong}\;.
\end{split}
\end{equation*}
Here, the tuple $(s_n)_n$ being a system of divided powers for $s_1$ means that the relations
\begin{equation}
\label{eq:div-powers}
s_ms_n=\binom{m+n}{n}s_{m+n}
\end{equation}
are satisfied for all $m, n\geq 1$.

\begin{ex}
The stack $\G_a^\sharp$ classifies ring elements equipped with a system of divided powers and is given by $\G_a^\sharp=\Spec\Z[t, \tfrac{t^2}{2!}, \tfrac{t^3}{3!}, \dots]$.
\end{ex}

\comment{
\subsubsection{Cones}

We will often make use of the following construction, details of which can be found in \cite{RingGroupoid}:

\begin{defi}
Let $d: A\rightarrow B$ be a morphism of sheaves of abelian groups. We denote by $\Cone(d: A\rightarrow B)=\Cone(d)$ the stack obtained by sheafifying the assignment
\begin{equation*}
S\mapsto \Cone(A(S)\xrightarrow{d(S)} B(S))\;,
\end{equation*}
where the right-hand side denotes the quotient in the category of animated abelian groups. Note that $\Cone(d)$ is equipped with the structure of a $1$-truncated animated abelian group stack.
\end{defi}

In the cases relevant to us, the stack $\Cone(d)$ will even admit the structure of a $1$-truncated animated ring stack. This is due to the fact that the morphism $d$ will behave like the inclusion of an ideal of a ring in the following sense, see \cite[Ch.\ 3]{RingGroupoid}:

\begin{defi}
Let $C$ be a commutative ring scheme. A morphism $d: I\rightarrow C$ of $C$-module schemes is called a \emph{quasi-ideal} if, for any test ring $S$ and any $x, y\in I(S)$, we have $d(x)\cdot y=d(y)\cdot x$.
\end{defi}
}

\subsection{De Rham stacks}

Now we are in the position to define the de Rham stack. From now on, let $X$ be a bounded $p$-adic formal scheme; here, $X$ being bounded means that it locally has the form $\Spf R$ for $R$ having bounded $p^\infty$-torsion. We work with \EDIT{$p$-adic formal stacks}.

\begin{defi}
Consider the stack
\begin{equation*}
\G_a^\dR\coloneqq\Cone(\G_a^\sharp\xrightarrow{\can}\G_a)\;,
\end{equation*}
which admits the structure of a $1$-truncated animated $\G_a$-algebra stack. The \emph{de Rham stack} $X^\dR$ of $X$ is the stack defined by
\begin{equation*}
X^\dR(S)\coloneqq \Map(\Spec\G_a^\dR(S), X)\;,
\end{equation*}
where the mapping space is computed in derived algebraic geometry.
\end{defi}

\begin{ex}
\label{ex:drstack-perfd}
Let $R$ be a perfectoid ring and denote by $(A_\inf(R), I)=(\Prism_R, I)$ the initial object of the absolute prismatic site of $R$; i.e.\ $A_\inf(R)=W(R^\flat)$, where, as usual,
\begin{equation*}
R^\flat\coloneqq\lim_{x\mapsto x^p} R\cong\lim_{x\mapsto x^p} R/p
\end{equation*}
and $I$ is the kernel of the Fontaine map
\begin{equation*}
\theta: A_\inf(R)\rightarrow R\;, \hspace{0.5cm} \sum_k p^k[x_k]\mapsto \sum_k p^kx_k^\sharp\;;
\end{equation*}
here, for $x\in R^\flat=\lim_{x\mapsto x^p} R$, we denote by $x^\sharp$ the image of $x$ under the projection onto the first coordinate. \EDIT{As $I$ is principal, we may pick a generator $\xi\in A_\inf(R)$} and let
\begin{equation*}
A_\crys(R)\coloneqq A_\inf(R)[\tfrac{\xi^n}{n!}: n\geq 0]_{(p)}^\wedge
\end{equation*}
be the $p$-completed divided power envelope of $A_\inf(R)$ at $\xi$; more canonically, one may describe $A_\crys(R)$ as the $p$-completed divided power envelope of $A_\inf(R)$ at $I$. Let us emphasise that we are using the free divided power envelope here, i.e. $A_\inf(R)[\tfrac{\xi^n}{n!}: n\geq 0]$ denotes the quotient of the polynomial ring $A_\inf(R)[s_1, s_2, \dots]$ by the relations $s_1=\xi$ and (\ref{eq:div-powers}). This is not the same as the subring of $A_\inf(R)[\tfrac{1}{p}]$ generated by the elements $\tfrac{\xi^n}{n!}$ for $n\geq 0$: e.g., for $R=k$ a perfect field of characteristic field, we have
\begin{equation*}
A_\crys(k)\cong (W(k)[x_1, x_2, \dots]/(px_1-p^p, px_2-x_1^p, \dots))^\vee_{(p)}
\end{equation*}
even though $d=p$ already admits divided powers in $W(k)$!

We claim that
\begin{equation}
\label{eq:drstack-perfd}
R^\dR\cong\Spf A_\crys(R)\;,
\end{equation}
\EDIT{where $A_\crys(R)$ is endowed with the $(p, \xi)$-adic topology (note that this agrees with the $p$-adic topology due to the existence of divided powers for $\xi$).} Indeed, we may describe this isomorphism on $S$-valued points for $\G_a^\sharp$-acyclic $p$-nilpotent rings $S$ by \cref{lem:gasharp-acyclic} below. Then $\G_a^\dR(S)=\Cone(\G_a^\sharp(S)\rightarrow S)$ and an $S$-valued point of \EDIT{$R^\dR$} is the datum of a morphism $R\rightarrow\Cone(\G_a^\sharp(S)\rightarrow S)$. We will shortly show that this yields a contractibly unique lift
\begin{equation*}
\begin{tikzcd}
A_\inf(R)\ar[r, dotted]\ar[d] & S\ar[d] \\
R=A_\inf(R)/I\ar[r] & \Cone(\G_a^\sharp(S)\rightarrow S)\nospacepunct{\;.}
\end{tikzcd}
\end{equation*}
The commutativity of the diagram then provides a system of divided powers for the image of $I$ in $S$ and hence induces a unique map $A_\crys(R)\rightarrow S$; i.e.\ we obtain an $S$-valued point of $\Spf A_\crys(R)$. \EDIT{Finally, one easily sees that the construction is reversible: an $S$-valued point of $\Spf A_\crys(R)$ furnishes a composite map $A_\inf(R)\rightarrow A_\crys(R)\rightarrow S$ together with a system of divided powers for the image of $I$ in $S$ and hence a factorisation of the composite $A_\inf(R)\rightarrow S\rightarrow\Cone(\G_a^\sharp(S)\rightarrow S)$ through $R=A_\inf(R)/I$, as desired. Thus, (\ref{eq:drstack-perfd}) is proved.}

To perform the missing deformation-theoretic argument, first observe that the image of $\G_a^\sharp(S)\rightarrow S$ is locally nilpotent (indeed, if $n$ is large enough so that $p^n=0$ in $S$, then $x^{p^n}=(p^n)!\cdot\frac{x^{p^n}}{(p^n)!}=0$ for any $x\in S$ admitting divided powers), hence $S\rightarrow\Cone(\G_a^\sharp(S)\rightarrow S)$ is surjective on $\pi_0$ with locally nilpotent kernel. As the cotangent complex $L\Omega_{R^\flat/\F_p}$ vanishes by virtue of $R^\flat$ being perfect (see e.g.\ \cite[Lem.\ 3.5]{PrismaticBhatt}), derived deformation theory\footnote{More precisely, we are using the following assertion: Let $A$ be an animated ring and $B$ an $A$-algebra such that $L\Omega_{B/A}$ vanishes. For any map of animated $A$-algebras $C'\rightarrow C$ which is surjective on $\pi_0$ with locally nilpotent kernel, any $A$-algebra map $B\rightarrow C$ lifts uniquely to an $A$-algebra map $B\rightarrow C'$. This roughly follows by combining a variant of the proof of \cite[Prop.\ 11.2.1.2]{SAG} with a noetherian approximation argument as in the proof of \cite[Cor.\ 5.2.10]{FlatPurity}.} shows that the composition 
\begin{equation*}
R^\flat=\lim_{x\mapsto x^p} R/p\rightarrow R/p\rightarrow \Cone(\G_a^\sharp(S)\rightarrow S)/^\mathbb{L} p\;,
\end{equation*}
where the first map is projection onto the first coordinate, lifts uniquely to a map $R^\flat\rightarrow S/^\mathbb{L} p$ and composition with projection onto $\pi_0$ yields a map $R^\flat\rightarrow S/p$ (conversely, $R^\flat\rightarrow S/^\mathbb{L} p$ can uniquely be recovered from this map as $S/^\mathbb{L} p\rightarrow S/p$ is an isomorphism on $\pi_0$). Since also $L\Omega_{W(R^\flat)/p^n/\Z/p^n}$ vanishes for any $n\geq 1$ (indeed, this may be checked after reduction mod $p$ \EDIT{by derived Nakayama}, but $L\Omega_{W(R^\flat)/p^n/\Z/p^n}\tensorL_{\Z/p^n} \F_p\cong L\Omega_{R^\flat/\F_p}=0$), the same argument shows that this map lifts uniquely to maps $A_\inf(R)=W(R^\flat)\rightarrow S/p^n$. As $S$ is $p$-nilpotent, taking $n$ sufficiently large, we obtain the desired map $A_\inf(R)\rightarrow S/p^n=S$.
\end{ex}

The computation above made use of the following statement:

\begin{lem}
\label{lem:gasharp-acyclic}
There is a basis of the fpqc topology consisting of $\G_a^\sharp$-acyclic rings (i.e.\ rings for which all $\G_a^\sharp$-torsors are trivial).
\end{lem}
\begin{proof}
As $\G_a^\sharp$ is countably presented, this follows from a variant of the proof of \cite[Lem.\ 5.3.1]{FlatPurity}, which we briefly \EDIT{outline here}: Namely, for any ring $S$, fix a set $\cal{S}$ of representatives \EEDIT{of} isomorphism classes of countably presented faithfully flat $S$-algebras and put 
\begin{equation*}
\EEDIT{S_1\coloneqq \bigotimes_{T\in\cal{S}} T\coloneqq \colim_{\cal{S}'\subseteq\cal{S}\text{ finite}} \bigotimes_{T\in\cal{S}'} T\;.}
\end{equation*}
Iterating this procedure, we obtain a tower of \EEDIT{faithfully flat} $S$-algebras $S_1\rightarrow S_2\rightarrow\dots$ and we put $S_\omega\coloneqq\colim_n S_n$. Continuing in this manner, we obtain $S$-algebras $S_\alpha$ for each countable ordinal $\alpha$ by transfinite induction and finally put $S_{\omega_1}\coloneqq\colim_\alpha S_\alpha$. 

\EEDIT{Now observe that the map $S\rightarrow S_{\omega_1}$ is a flat cover: Indeed, each $S\rightarrow \bigotimes_{T\in\cal{S}'} T$ for $\cal{S}'\subseteq\cal{S}$ is a flat cover since flat covers are stable under base change and composition; this then implies that $S_1$ is flat over $S$ as flatness is stable under filtered colimits. Moreover, note that any transition map in the colimit defining $S_1$ is also faithfully flat; in particular, all transition maps are injective. This means that, for any nonzero $S$-module $M$, we have 
\begin{equation*}
M\tensor_S S_1=\colim_{\cal{S}'\subseteq\cal{S}\text{ finite}} M\tensor_S \bigotimes_{T\in\cal{S}'} T\;,
\end{equation*}
where each term of the colimit is nonzero and each transition map is injective, hence $M\tensor_S S_1$ is nonzero itself and, consequently, $S_1$ is a faithfully flat $S$-algebra. Using similar arguments, one can now indeed deduce that $S\rightarrow S_{\omega_1}$ is a flat cover by transfinite induction.}

\EEDIT{Now note that} \EDIT{the ring $S_{\omega_1}$ admits no nonsplit countably presented flat cover $S_{\omega_1}\rightarrow R$: indeed, as $R$ is a countably presented $S_{\omega_1}$-algebra, it descends to a countably presented $S_\alpha$-algebra $\widetilde{R}$ for some countable ordinal $\alpha$ and $S_\alpha\rightarrow\widetilde{R}$ is faithfully flat since this can be checked after base change to $S_{\omega_1}$. However, by construction of $S_{\alpha+1}$, this implies that there is a morphism $\widetilde{R}\rightarrow S_{\alpha+1}\rightarrow S_{\omega_1}$ of $S_\alpha$-algebras, which yields the desired splitting of $S_{\omega_1}\rightarrow R$. In particular, we conclude that any $\G_a^\sharp$-torsor over $S_{\omega_1}$ is split, i.e.\ trivial.}
\end{proof}

In good cases, coherent cohomology on the de Rham stack of $X$ actually computes \EDIT{the} ($p$-completed) de Rham cohomology of $X$:

\begin{thm}
\label{thm:drstack-comparison}
Let $X$ be a \EEDIT{smooth qcqs} $p$-adic formal scheme and write $\pi_{X^\dR}: X^\dR\rightarrow \Spf\Z_p$. Then $\H_\dR(X)\coloneqq \pi_{X^\dR, *}\O_{X^\dR}$ identifies with $R\Gamma_\dR(X)$.
\end{thm}
\begin{proof}
This is easily deduced from \cite[Thm.\ 2.5.6]{FGauges}.
\end{proof}

Motivated by this result, we make the following definition:

\begin{defi}
Let $X$ be a \EEDIT{smooth qcqs} $p$-adic formal scheme. For a quasi-coherent complex $E\in\D(X^\dR)$, we define the \emph{de Rham cohomology} of $X$ with coefficients in $E$ as
\begin{equation*}
\EDIT{
R\Gamma_\dR(X, E)\coloneqq R\Gamma(X^\dR, E)=\pi_{X^\dR, *}(E)\;.
}
\end{equation*}
\end{defi}

\begin{rem}
\label{rem:drstack-vect}
In fact, the above notion of coefficients for de Rham cohomology recovers a more classical such notion: Namely, vector bundles on $X^\dR$ are the same as vector bundles on $X$ equipped with a flat connection \EDIT{having locally nilpotent $p$-curvature after reduction mod $p$ (see \cite[§5]{Katz} for a definition of the notion of $p$-curvature)} -- this can be proved using \cite[Rem.\ 2.5.7]{FGauges} and a similar calculation as in \cite[Lem.\ 3.4.2]{NygaardHodge}.
\end{rem}

To additionally incorporate the Hodge-filtration, we have to consider the following deformation of the stack $X^\dR$ over $\A^1_-/\G_m$:

\begin{defi}
Over $\A^1_-/\G_m$, the canonical map \EEDIT{$\V(\O(1))^\sharp\rightarrow\G_a^\sharp\rightarrow\G_a$} defines a 1-truncated animated $\G_a$-algebra stack
\begin{equation*}
\EEDIT{
\G_a^{\dR, +}\coloneqq\Cone(\V(\O(1))^\sharp\xrightarrow{\can}\G_a)
}
\end{equation*}
over $\A^1_-/\G_m$. The \emph{Hodge-filtered de Rham stack} $X^{\dR, +}$ of $X$ is the stack $\pi_{X^{\dR, +}}: X^{\dR, +}\rightarrow\A^1_-/\G_m$ defined by
\begin{equation*}
X^{\dR, +}(\Spec S\rightarrow\A^1_-/\G_m)\coloneqq \Map(\Spec\G_a^{\dR, +}(S), X)\;,
\end{equation*}
where the mapping space is computed in derived algebraic geometry.
\end{defi}

\begin{rem}
\label{rem:fildrstack-unfiltereddr}
Observe that the preimage of $\G_m/\G_m\subseteq\A_1/\G_m$ under $\pi_{X^{\dR, +}}$ recovers the de Rham stack $X^\dR$. The preimage of $B\G_m\subseteq \A^1_-/\G_m$ is called the \emph{Hodge stack} of $X$ and denoted $X^\Hod$.
\end{rem}

\begin{ex}
\label{ex:fildrstack-perfd}
Generalising \cref{ex:drstack-perfd}, we claim that, for a perfectoid ring $R$, there is an isomorphism of stacks
\begin{equation}
\label{eq:fildrstack-perfd}
R^{\dR, +}\cong\Spf(A_\inf(R)[\tfrac{u^n}{n!}, t: n\geq 1]_{(p)}^\wedge/(ut-\xi))/\G_m
\end{equation}
over $\A^1_-/\G_m$, \EDIT{where $A=A_\inf(R)[\tfrac{u^n}{n!}, t: n\geq 1]_{(p)}^\wedge/(ut-\xi)$ is equipped with the $(p, \xi)$-adic topology (note that this agrees with the $p$-adic topology due to the existence of divided powers for $u$ and hence also for $\xi=ut$) and, as usual, $t$ has degree \EEDIT{$-1$} while $u$ has degree \EEDIT{$1$}. We note here that while $A$ is clearly derived $p$-complete, it is also classically $p$-complete since it is $p$-adically separated.} 

To prove (\ref{eq:fildrstack-perfd}), as before, it suffices to describe this isomorphism on $S$-valued points for $\G_a^\sharp$-acyclic $p$-nilpotent rings $S$ equipped with a morphism $\Spec S\rightarrow\A^1_-/\G_m$ by \cref{lem:gasharp-acyclic}. Then $R^{\dR, +}(S)$ is the groupoid of maps of animated rings \EEDIT{$R\rightarrow\Cone(\V(\O(1))^\sharp(S)\rightarrow S)$} by $\G_a^\sharp$-acyclicity of $S$ (any \EEDIT{$\V(\O(1))^\sharp$-torsor} over $S$ is also a $\G_a^\sharp$-torsor) and if the morphism $\Spec S\rightarrow\A^1_-/\G_m$ is classified by the generalised Cartier divisor $t: L\rightarrow S$, then \EEDIT{$\V(\O(1))^\sharp(S)$} identifies with the set of global sections of $L$ equipped with a system of divided powers. Now given any such map \EEDIT{$R\rightarrow\Cone(\V(\O(1))^\sharp(S)\rightarrow S)$}, we obtain a unique lift
\begin{equation*}
\EEDIT{
\begin{tikzcd}[ampersand replacement=\&]
A_\inf(R)\ar[r, dotted]\ar[d] \& S\ar[d] \\
R=A_\inf(R)/(\xi)\ar[r] \& \Cone(\V(\O(1))^\sharp(S)\rightarrow S)\nospacepunct{\;.}
\end{tikzcd}
}
\end{equation*}
as in \cref{ex:drstack-perfd} and then the commutativity of the diagram provides a unique factorisation $S\xrightarrow{u} L\xrightarrow{t} S$ of the multiplication-by-$\xi$-map together with a system of divided powers for $u$ (i.e.\ maps $\frac{u^n}{n!}: S\rightarrow L^{\tensor n}$). \EEDIT{Together with the map $A_\inf(R)\rightarrow S$ we obtained,} this is precisely the datum of an $S$-valued point of $\Spf A/\G_m$ and the construction is clearly reversible, so we are done.
\end{ex}

As expected, coherent cohomology on the Hodge-filtered de Rham stack of $X$ computes \EDIT{the} Hodge-filtered de Rham cohomology of $X$ in good cases:

\begin{thm}
\label{thm:fildrstack-comparison}
Let $X$ be a \EEDIT{smooth qcqs} $p$-adic formal scheme and consider its Hodge-filtered de Rham stack $\pi_{X^{\dR, +}}: X^{\dR, +}\rightarrow\A^1_-/\G_m$. Then $\H_{\dR, +}(X)\coloneqq \pi_{X^{\dR, +}, *}\O_{X^{\dR, +}}$ identifies with $\Fil^\bullet_{\Hod} R\Gamma_\dR(X)$ in $\widehat{\DF}(\Z_p)$ under the Rees equivalence.
\end{thm}
\begin{proof}
This is \cite[Thm.\ 2.5.6]{FGauges}.
\end{proof}

Note that the above statement also implies that the pushforward of $\O_{X^\Hod}$ to $\Z_p^\Hod=B\G_m$ identifies as a graded object with \EDIT{the} Hodge cohomology of $X$. Again, we are led to the following definition:

\begin{defi}
Let $X$ be a \EEDIT{smooth qcqs} $p$-adic formal scheme. For a quasi-coherent complex $E\in\D(X^{\dR, +})$, we define the \emph{Hodge-filtered de Rham cohomology} of $X$ with coefficients in $E$ as
\begin{equation*}
\EDIT{\Fil^\bullet_\Hod R\Gamma_\dR(X, E)\coloneqq \pi_{X^{\dR, +}, *}(E)\;.}
\end{equation*}
\end{defi}

\begin{rem}
\label{rem:fildrstack-vect}
\EDIT{
As in the case of the de Rham stack, one can show that the above notion of coefficients for Hodge-filtered de Rham cohomology agrees with a more classical one, \EEDIT{see \cite[Rem.\ 2.5.8]{FGauges}}: Namely, a vector bundle on $X^{\dR, +}$ is the same as a vector bundle $E$ on $X$ equipped with a decreasing filtration $\Fil^\bullet E$ by subbundles and a flat connection $\nabla: E\rightarrow E\tensor\Omega_{X/\Z_p}^1$ which has nilpotent $p$-curvature mod $p$ and satisfies Griffiths transversality with respect to the filtration $\Fil^\bullet E$, i.e.\ we have $\nabla(\Fil^i E)\subseteq \Fil^{i-1} E\tensor\Omega_{X/\Z_p}^1$ for all $i\in\Z$.
}
\end{rem}

\subsection{Filtered prismatisation}
\label{subsect:filprism}

Having seen stacky formulations of de Rham cohomology and the Hodge filtration, we now describe a similar story for prismatic cohomology and the Nygaard filtration. For the Nygaard-filtered prismatisation, we have chosen to make use of the alternative definition recently given by Gardner--Madapusi in \cite[§6.2]{GardnerMadapusi} as this will significantly simplify some of the arguments in the upcoming sections.

\begin{defi}
\label{defi:prismatisation-cwdiv}
For a $p$-nilpotent ring $S$, a \emph{Cartier--Witt divisor} on $S$ is a generalised Cartier divisor $\alpha: I\rightarrow W(S)$ on $W(S)$ satisfying the following two conditions:
\begin{enumerate}[label=(\roman*)]
\item The \EDIT{ideal generated by the} image of the map $I\xrightarrow{\alpha} W(S)\rightarrow S$ \EDIT{is nilpotent}.
\item The image of the map $I\xrightarrow{\alpha} W(S)\xrightarrow{\delta} W(S)$ generates the unit ideal.
\end{enumerate}
Here, $\delta: W(S)\rightarrow W(S)$ is the usual $\delta$-structure on $W(S)$.
\end{defi}

\begin{defi}
The \emph{prismatisation} $X^\prism$ is the stack over $\Spf\Z_p$ given by assigning to a $p$-nilpotent ring $S$ the groupoid of pairs 
\begin{equation*}
(I\xrightarrow{\alpha} W(S), \Spec \cofib(I\xrightarrow{\alpha} W(S))\rightarrow X)\;,
\end{equation*}
where the first entry is a Cartier--Witt divisor on $S$ and the second entry is a morphism of derived formal schemes.
\end{defi}

\begin{ex}
For any Cartier--Witt divisor $I\xrightarrow{\alpha} W(S)$, its Frobenius pullback $F^*I\xrightarrow{F^*\alpha} W(S)$ is again a Cartier--Witt divisor, see \cite[Rem. 5.1.10]{FGauges}, and this induces an endomorphism $F_X: X^\prism\rightarrow X^\prism$ called the \emph{Frobenius} on $X^\prism$. For $X=\Spf\Z_p$, we just write $F$ in place of $F_X$; in that case, using the presentation of $\Z_p^\prism$ given in \cite[Prop.\ 3.2.3]{APC} and the fact that the Witt vector Frobenius is faithfully flat, see \cite[Prop.\ 3.4.7]{APC}, one can show that $F: \Z_p^\prism\rightarrow\Z_p^\prism$ is an fpqc cover.
\end{ex}

\begin{ex}
\label{ex:prismatisation-cwdivtodiv}
If $I\xrightarrow{\alpha} W(S)$ is a Cartier--Witt divisor, then $I\tensor_{W(S)} S\rightarrow S$ is a generalised Cartier divisor on $S$ whose image lands inside the nilradical of $S$. Thus, we obtain a map $\mu_X: X^\prism\rightarrow\widehat{\A}^1_-/\G_m$; if $X=\Spf\Z_p$, we just write $\mu$ in place of $\mu_X$.
\end{ex}

\begin{ex}
\label{ex:prismatisation-qrsp}
By \cite[Thm.\ 5.5.7]{FGauges}, the prismatisation of a quasiregular semiperfectoid ring $R$ is given by $R^\prism\cong\Spf\Prism_R$, where $(\Prism_R, I)$ is the initial object of the absolute prismatic site of $R$ and $\Prism_R$ is equipped with the $(p, I)$-adic topology. In particular, if $R$ is perfectoid, we have $R^\prism\cong \Spf A_\inf(R)$.
\end{ex}

As before, in good situations, coherent cohomology of the structure sheaf on $X^\prism$ agrees with \EDIT{the} prismatic cohomology of $X$ computed via the absolute prismatic site:

\begin{thm}
\label{thm:prismatisation-comparison}
Let $X$ be a bounded $p$-adic formal scheme and \EEDIT{furthermore} assume that $X$ is $p$-quasisyntomic and qcqs. Then there is a natural isomorphism
\begin{equation*}
R\Gamma(X^\prism, \O_{X^\prism})\cong R\Gamma_\prism(X)\;.
\end{equation*}
\end{thm}
\begin{proof}
Combine \cite[Cor.\ 8.17]{PFS} with \cite[Thm.\ 4.4.30]{APC}.
\end{proof}

\begin{ex}
\label{ex:prismatisation-bktwist}
There is a certain line bundle $\O_{\Z_p^\prism}\{1\}\in\D(\Z_p^\prism)$ called the \emph{Breuil--Kisin twist} which will play a key role in the sequel and whose pullback to $X^\prism$ we will denote by $\O_{X^\prism}\{1\}$. Roughly, one can obtain $\O_{\Z_p^\prism}\{1\}$ as follows: The category $\D(\Z_p^\prism)$ identifies with the category of prismatic crystals by \cite[Prop.\ 3.3.5]{APC} and, under this equivalence, the sheaf $\O_{\Z_p^\prism}\{1\}$ corresponds to the prismatic crystal assigning to every prism $(A, I)$ the Breuil--Kisin twist $A\{1\}$ from \cite[Def.\ 2.5.2]{APC}, which can be thought of heuristically as the infinite tensor product $\bigotimes_{n\geq 0} (\phi^n)^*I$, where $\phi$ is the Frobenius lift associated to the prism $(A, I)$. Most importantly for what follows, we have $\O_{\Z_p^\prism}\{1\}\tensor F^*\O_{\Z_p^\prism}\{-1\}\cong \cal{I}$, where $\cal{I}\coloneqq \mu^*\O(-1)$, see \cref{ex:prismatisation-cwdivtodiv}, and, in particular,
\begin{equation*}
\O_{\Z_p^\prism}\{p-1\}/p\cong\cal{I}^{-1}/p\;. \qedhere
\end{equation*}
\end{ex}

As before, this leads to a notion of prismatic cohomology with coefficients:

\begin{defi}
Let $X$ be a bounded $p$-adic formal scheme which is $p$-quasisyntomic and qcqs. For a quasi-coherent complex $E\in\D(X^\prism)$, we define the \emph{prismatic cohomology} of $X$ with coefficients in $E$ as
\begin{equation*}
R\Gamma_\prism(X, E)\coloneqq R\Gamma(X^\prism, E)\;.
\end{equation*}
\end{defi}

We now turn to defining the filtered refinement of $X^\prism$ computing the Nygaard filtration. For this, let $W$ be the ring scheme of $p$-typical Witt vectors, i.e.\ $W\cong\prod_{n\in\mathbb{N}} \Spec\Z[t_1, t_2, \dots]$ as schemes \EDIT{so that the functor of points of $W$ is given by $S\mapsto W(S)$ for any $p$-nilpotent ring $S$ via the Witt components}. As usual, we denote the Frobenius and the Verschiebung by $F$ and $V$, respectively. Recall from \cite[Cor.\ 2.6.8]{FGauges} that there is an isomorphism $F_*W/^\mathbb{L} p\cong\G_a^\dR$ of $W$-module schemes. Thus, we obtain a map
\begin{equation*}
\widetilde{\mu}: \Z_p^\prism\rightarrow (\A^1_-/\G_m)^\dR
\end{equation*}
by sending a Cartier--Witt divisor $I\xrightarrow{\alpha} W(S)$ to the generalised Cartier divisor $F_*I\tensor_{F_*W(S)} \G_a^\dR(S)\rightarrow\G_a^\dR(S)$ on the animated ring $\G_a^\dR(S)$.

\begin{defi}
The stack $\Z_p^\N$ is defined by the pullback diagram
\begin{equation*}
\begin{tikzcd}
\Z_p^\N\ar[r, "\pi"]\ar[d, "{(t, u)}", swap] & \Z_p^\prism\ar[d, "\widetilde{\mu}"] \\
\A^1_-/\G_m\times (\A^1_+/\G_m)^\dR\ar[r, "\text{mult}"] & (\A^1_-/\G_m)^\dR\nospacepunct{\;,}
\end{tikzcd}
\end{equation*}
where the bottom map is given by sending a pair $(L\rightarrow S, \G_a^\dR(S)\rightarrow L')$ to the generalised Cartier divisor $L\tensor_S {L'}^{-1}\rightarrow \G_a^\dR(S)$ on $\G_a^\dR(S)$.
The map $t: \Z_p^\N\rightarrow\A^1_-/\G_m$ is called the \emph{Rees map} while $\pi: \Z_p^\N\rightarrow\Z_p^\prism$ is called the \emph{structure map}.
\end{defi}

\begin{ex}
\label{ex:filprism-jdrjht}
The inclusion $\G_m/\G_m\subseteq\A^1_-/\G_m$ induces an open immersion
\begin{equation*}
j_\dR: \Z_p^\prism\hookrightarrow\Z_p^\N\;.
\end{equation*}
Moreover, by the pullback diagram
\begin{equation*}
\begin{tikzcd}
\Z_p^\prism\ar[r, "F"]\ar[d, "\mu", swap] & \Z_p^\prism\ar[d, "\widetilde{\mu}"] \\
\A^1_-/\G_m\ar[r] & (\A^1_-/\G_m)^\dR\nospacepunct{\;,}
\end{tikzcd}
\end{equation*}
which in turn comes from the pullback square
\begin{equation*}
\begin{tikzcd}
W\ar[r, "F"]\ar[d] & F_*W\ar[d] \\
\G_a\ar[r] & \G_a^\dR
\end{tikzcd}
\end{equation*}
(note that both rows have fibre $\G_a^\sharp$ by \cite[Rem.\ 2.6.2]{FGauges}), the inclusion $(\G_m/\G_m)^\dR\hookrightarrow (\A^1/\G_m)^\dR$ induces an open immersion
\begin{equation*}
j_\HT: \Z_p^\prism\hookrightarrow\Z_p^\N
\end{equation*}
whose composition with the structure map identifies with the Frobenius endomorphism $F: \Z_p^\prism\rightarrow\Z_p^\prism$ of $\Z_p^\prism$ defined in \cite[Constr.\ 3.6.1]{APC}.
\end{ex}

\begin{ex}
The two maps
\begin{equation*}
\A^1_-/\G_m\rightarrow \A^1_-/\G_m\times (\A^1_+/\G_m)^\dR\;, \hspace{0.3cm} (L\rightarrow S)\mapsto (L\rightarrow S, \G_a^\dR(S)\xrightarrow{0} L\tensor_S \G_a^\dR(S))
\end{equation*}
and $\A^1_-/\G_m\rightarrow\Spf\Z_p\rightarrow\Z_p^\prism$, where the latter map sends any $p$-nilpotent ring $S$ to the Cartier--Witt divisor $W(S)\xrightarrow{p} W(S)$, induce a map
\begin{equation*}
i_{\dR, +}: \A^1_-/\G_m\rightarrow \Z_p^\N
\end{equation*}
called the \emph{Hodge-filtered de Rham map}. Restricting $i_{\dR, +}$ to $\Spf\Z_p=\G_m/\G_m\subseteq\A^1_-/\G_m$ and $B\G_m\subseteq\A^1_-/\G_m$ yields maps $i_\dR: \Spf\Z_p\rightarrow\Z_p^\N$ and $i_\Hod: B\G_m\rightarrow\Z_p^\N$ called the \emph{de Rham map} and the \emph{Hodge map}, respectively. We warn the reader that, despite the notation, the maps \EEDIT{$i_{\dR, +}, i_\dR$ and $i_\Hod$ are \emph{not} closed immersions.}
\end{ex}

To construct $X^\N$ in general, we use the same procedure as in the case of de Rham cohomology and first construct $\G_a^\N$. For this, consider an $S$-point of $\Z_p^\N$; this is the data of (dual) generalised Cartier divisors $L\rightarrow S, \G_a^\dR(S)\rightarrow L'$ and a Cartier--Witt divisor $I\rightarrow W(S)$ together with an isomorphism 
\begin{equation*}
(F_*I\tensor_{F_*W(S)} \G_a^\dR(S)\rightarrow \G_a^\dR(S))\cong (L\tensor_S {L'}^{-1}\rightarrow \G_a^\dR(S))
\end{equation*}
of generalised Cartier divisors on $\G_a^\dR(S)$. In particular, we obtain a commutative diagram
\begin{equation*}
\begin{tikzcd}
F_*I\tensor_{F_*W(S)} \G_a^\dR(S)\ar[r]\ar[d, equals] & \G_a^\dR(S)\ar[d, equals] \\
L\tensor_S {L'}^{-1}\tensor_S \G_a^\dR(S)\ar[d]\ar[r] & \G_a^\dR(S)\ar[d, equals] \\
L\tensor_S \G_a^\dR(S)\ar[r] & \G_a^\dR(S)\nospacepunct{\;,}
\end{tikzcd}
\end{equation*}
where the bottom left arrow is induced by ${L'}^{-1}\rightarrow \G_a^\dR(S)$, and taking horizontal cofibres yields a map
\begin{equation*}
F_*W(S)/^\mathbb{L} F_*I\rightarrow \G_a^\dR(S)/^\mathbb{L} (F_*I\tensor_{F_*W(S)} \G_a^\dR(S))\rightarrow \G_a^\dR(S)/^\mathbb{L} (L\tensor_S \G_a^\dR(S))\;.
\end{equation*}

\begin{defi}
Let $(L\rightarrow S, \G_a^\dR(S)\rightarrow L', I\rightarrow W(S))$ be an $S$-valued point of $\Z_p^\N$ as above. The ring stack $\G_a^\N\rightarrow\Z_p^\N$ is defined by the pullback diagram
\begin{equation*}
\begin{tikzcd}
\G_a^\N(\Spec S\rightarrow\Z_p^\N)\ar[r]\ar[d] & F_*W(S)/^\mathbb{L} F_*I\ar[d] \\
S/^\mathbb{L} L\ar[r] & \G_a^\dR(S)/^\mathbb{L} (L\tensor_S \G_a^\dR(S))\nospacepunct{\;.}
\end{tikzcd}
\end{equation*}
The \emph{Nygaard-filtered prismatisation} $X^\N$ is the stack over $\Z_p^\N$ given by
\begin{equation*}
X^\N(\Spec S\rightarrow\Z_p^\N)\coloneqq \Map(\Spec\G_a^\N(\Spec S\rightarrow\Z_p^\N), X)\;,
\end{equation*}
where the mapping space is computed in derived algebraic geometry.
\end{defi}

All the maps that we have defined above induce corresponding maps for $X^\N$ in place of $\Z_p^\N$ etc. The situation is summarised by the following commutative diagram:
\begin{equation}
\label{eq:filprism-maps}
\EDIT{
\begin{tikzcd}[ampersand replacement=\&]
\&\& X^\prism \ar[drr, equals] \&\& \\
X^\prism\ar[rr, "j_\HT"] \ar[dd] \ar[urr, "F_X"] \&\& X^\N \ar[dd, "t_X" {yshift=10pt}] \ar[u, "\pi_X", swap] \&\& X^\prism \ar[ll, "j_\dR", swap] \ar[dd] \\
\& X^{\dR, +}\ar[ur, "i_{\dR, +}" {yshift=-2pt}]\ar[dr] \&\& X^\dR\ar[ur, "i_\dR" {yshift=-2pt}] \ar[dr]\ar[ll] \&  \\
\widehat{\A}^1_-/\G_m \ar[rr] \&\& \A^1_-/\G_m \&\& \G_m/\G_m \ar[ll]
\end{tikzcd}
}
\end{equation}

\begin{ex}
\label{ex:filprism-qrsp}
By \cite[Cor.\ 5.5.11]{FGauges}, the Nygaard-filtered prismatisation of a quasiregular semiperfectoid ring $R$ is given by 
\begin{equation*}
R^\N\cong \Spf(\Rees(\Fil^\bullet_\N\Prism_R))/\G_m\;,
\end{equation*}
where $(\Prism_R, I)$ is the initial prism of the absolute prismatic site of $R$ as above. In particular, if $R$ is perfectoid, using the notation from \cref{ex:drstack-perfd}, we have
\begin{equation*}
R^\N\cong \Spf(A_\inf(R)\langle u, t\rangle/(ut-\phi^{-1}(\xi)))/\G_m\;,
\end{equation*}
where $\phi$ is the Frobenius on $A_\inf(R)$ and $u$ and $t$ have degree $1$ and $-1$, respectively; we stress that the computation from \cite[Ex.\ 5.5.6]{FGauges} shows that the notations $u$ and $t$ for the variables here is compatible with the notation $(t, u)$ for the map $\Z_p^\N\rightarrow \A^1_-/\G_m\times (\A^1_+/\G_m)^\dR$.
\end{ex}

\begin{rem}
\label{rem:filprism-regular}
\EDIT{
There is an fpqc cover $\Spf\Z_p\langle u, t\rangle\rightarrow\Z_p^\N$ and hence the stack $\Z_p^\N$ is noetherian and regular. This can be proved by reduction to the perfectoid case treated in the previous example, see \cite[Ex.\ 5.5.20]{FGauges} for details.
}
\end{rem}

As one should have expected by now, in good situations, coherent cohomology on the \EEDIT{Nygaard-filtered prismatisation} of $X$ computes \EDIT{the} Nygaard-filtered prismatic cohomology of $X$ (here, we use the definition of the Nygaard filtration on absolute prismatic cohomology from \cite[Def. 5.5.3]{APC}):

\begin{thm}
\label{thm:filprism-comparison}
Let $X$ be a bounded $p$-adic formal scheme. Assume that $X$ is $p$-quasisyntomic and qcqs. Then $t_{X, *}\O_{X^\N}$ identifies with $\Fil^\bullet_\N R\Gamma_\prism(X)$ in $\widehat{\DF}(\Z_p)$ under the Rees equivalence.
\end{thm}
\begin{proof}
This follows from \cite[Cor.\ 5.5.11, Rem.\ 5.5.18]{FGauges} and \cite[Cor.\ 5.5.21]{APC} via quasisyntomic descent.
\end{proof}

\begin{ex}
\label{ex:filprism-bktwist}
The line bundle
\begin{equation*}
\EEDIT{
\O_{\Z_p^\N}\{1\}\coloneqq \pi^*\O_{\Z_p^\prism}\{1\}\tensor t^*\O(1)\in \D(\Z_p^\N)
}
\end{equation*}
is called the \emph{Breuil--Kisin twist}; via pullback to $X^\N$, we obtain line bundles $\O_{X^\N}\{1\}\in\D(X^\N)$.
\end{ex}

\begin{defi}
Let $X$ be a bounded $p$-adic formal scheme which is $p$-quasisyntomic and qcqs. For any $E\in\D(X^\N)$, we define the \emph{Nygaard-filtered prismatic cohomology} of $X$ with coefficients in $E$ as
\begin{equation*}
\EDIT{\Fil^\bullet_\N R\Gamma_\prism(X, E)\coloneqq t_{X, *}(E)\;.}
\end{equation*}
\end{defi}

\EDIT{
Recalling that the stack $\Z_p^\N$ is noetherian by virtue of the cover $\Spf\Z_p\langle u, t\rangle\rightarrow \Z_p^\N$ from \cref{rem:filprism-regular}, we can also introduce a reasonable notion of coherent sheaves on $\Z_p^\N$:

\begin{defi}
A quasi-coherent complex on $\Z_p^\N$ is called \emph{coherent} if its pullback to $\Spf\Z_p\langle u, t\rangle$ along the map from \cref{rem:filprism-regular} is a finitely generated $\Z_p\langle u, t\rangle$-module. The full subcategory of $\Perf(\Z_p^\N)$ spanned by coherent sheaves is denoted $\Coh(\Z_p^\N)$.
\end{defi}

\begin{rem}
\label{rem:filprism-cohheartperf}
Since $\Z_p\langle u, t\rangle$ \EEDIT{is regular}, one may alternatively describe $\Coh(\Z_p^\N)$ as the heart of a canonical $t$-structure on $\Perf(\Z_p^\N)$ induced by the $t$-structure on $\Perf(\Z_p\langle u, t\rangle)$, see \cite[Rem.\ 5.5.19]{FGauges}.
\end{rem}
}

We conclude this section by giving an alternative description of the Nygaard-filtered prismatisation via descent which is sometimes helpful. For this, recall from \cref{ex:filprism-qrsp} that we can explicitly describe the Nygaard-filtered prismatisation of a quasiregular semiperfectoid ring. Subsequently, one can use the fact that quasiregular semiperfectoid rings form a basis of the quasisyntomic site introduced by Bhatt--Morrow--Scholze in \cite[Def. 4.10]{THHandPAdicHodgeTheory} to obtain a description of $X^\N$ for any $p$-adic formal scheme $X$ which is $p$-quasisyntomic and qcqs. For this, one needs to use the following property of the Nygaard-filtered prismatisation:

\begin{lem}
\label{lem:filprism-quasisyntomiccover}
Assume that $f: X\rightarrow Y$ is a quasi-syntomic cover of $p$-adic formal schemes which are $p$-quasisyntomic and qcqs. Then the induced map $X^\N\rightarrow Y^\N$ is a flat cover.
\end{lem}
\begin{proof}
See \cite[Cor.\ 6.12.8]{GardnerMadapusi}.
\end{proof}

\subsection{Syntomification}
\label{sect:syntomification}

\begin{defi}
\label{defi:syntomification-def}
The \emph{syntomification} $X^\Syn$ of $X$ is defined as the pushout 
\begin{equation*}
\begin{tikzcd}[column sep=large]
X^\prism\sqcup X^\prism\ar[r, "j_\HT\sqcup j_\dR"]\ar[d] & X^\N\ar[d, "j_\N"] \\
X^\prism\ar[r, "j_\prism"] & X^\Syn\;.
\end{tikzcd}
\end{equation*}
\end{defi}

\begin{defi}
The category $\D(X^\Syn)$ is called the category of \emph{$F$-gauges} on $X$ and denoted $\FGauge_\prism(X)$.
\end{defi}

Observe that there is an equaliser diagram
\begin{equation}
\label{eq:syntomification-equaliser}
\begin{tikzcd}
\D(X^\Syn)\ar[r, "j_\N^*"] & \D(X^\N)\ar[r,shift left=.75ex,"j_\HT^*"]
  \ar[r,shift right=.75ex,swap,"j_\dR^*"] & \D(X^\prism)\nospacepunct{\;.}
\end{tikzcd}
\end{equation}
In particular, for any $E\in\D(X^\Syn)$, there is a fibre sequence
\begin{equation}
\label{eq:syntomification-fibreseq}
\begin{tikzcd}[column sep=large]
R\Gamma(X^\Syn, E)\ar[r] & R\Gamma(X^\N, j_\N^*E)\ar[r, "j_\HT^*-j_\dR^*"] & R\Gamma(X	^\prism, j_\prism^*E)\nospacepunct{\;.}
\end{tikzcd}
\end{equation}

\begin{ex}
\label{ex:syntomification-bktwists}
One can check that the Breuil--Kisin twist $\O_{X^\N}\{1\}$ descends to a line bundle $\O_{X^\Syn}\{1\}$. We also call $\O_{X^\Syn}\{1\}$ the \emph{Breuil--Kisin twist} and denote its $n$-th tensor power by $\O_{X^\Syn}\{n\}$.
\end{ex}

\begin{ex}
\label{ex:syntomification-fgaugesperfd}
Let $R$ be a perfectoid ring. Due to the isomorphism
\begin{equation*}
R^\N\cong\Spf(A_\inf(R)\langle u, t\rangle/(ut-\phi^{-1}(\xi)))/\G_m
\end{equation*}
from \cref{ex:filprism-qrsp}, a sheaf $E$ on $R^\N$ carries the same data as a diagram
\begin{equation*}
\begin{tikzcd}
\dots \ar[r,shift left=.5ex,"t"]
  & \ar[l,shift left=.5ex, "u"] M^{i+1} \ar[r,shift left=.5ex,"t"] & \ar[l,shift left=.5ex, "u"] M^i \ar[r,shift left=.5ex,"t"] & \ar[l,shift left=.5ex, "u"] M^{i-1} \ar[r,shift left=.5ex,"t"] & \ar[l,shift left=.5ex, "u"] \dots
\end{tikzcd}
\end{equation*}
of $(p, \xi)$-complete $A_\inf(R)$-complexes such that $ut=tu=\phi^{-1}(\xi)$. Writing $M^{-\infty}=(\colim_i M^{-i})_{(p, \xi)}^\wedge$ and $M^\infty=(\colim_i M^i)_{(p, \xi)}^\wedge$ for the $(p, \xi)$-completed colimits along the $t$- and $u$-maps, respectively, the pullback $j_\dR^*E$ identifies with $M^{-\infty}$ while $j_\HT^*E$ identifies with $\phi^*M^\infty$. Thus, specifying a descent datum of $E$ to $R^\Syn$ amounts to specifying an isomorphism $\tau: \phi^*M^\infty\cong M^{-\infty}$. In this situation, observing that $R\Gamma(R^\N, j_\N^* E)=M^0$, we see that the fibre sequence (\ref{eq:syntomification-fibreseq}) yields 
\begin{equation*}
R\Gamma(R^\Syn, E)=\fib(M^0\xrightarrow{t^\infty-\tau u^\infty} M^{-\infty})\;.
\end{equation*}
Here, $t^\infty: M^0\rightarrow M^{-\infty}$ and $u^\infty: M^0\rightarrow M^\infty$ denote the maps induced by the $t$- and $u$-maps, respectively, and to define the composition $\tau u^\infty$, we use the identification $M^\infty\cong\phi^*M^\infty$.
\end{ex}

\begin{defi}
\label{defi:nygaardhodge-htweights}
Let $X$ be a bounded $p$-adic formal scheme and $E\in\D(X^\Syn)$ an $F$-gauge on $X$. Then the pullback of $E$ along
\begin{equation*}
\EEDIT{
\begin{tikzcd}[ampersand replacement=\&]
X\times B\G_m\ar[r] \& X^{\Hod}\ar[r, "i_\Hod"] \& X^\N\ar[r, "j_\N"] \& X^\Syn\;,
\end{tikzcd}
}
\end{equation*}
where the first map is induced by the canonical map \EEDIT{$\G_a\rightarrow B\V(\O(1))^\sharp\oplus \G_a=\G_a^\Hod$ of stacks over $B\G_m$,} identifies with a graded quasi-coherent complex $M^\bullet$ on $X$ and the set of integers $i$ such that \EEDIT{the $i$-th graded piece $M^i$} is nonzero is called the set of \emph{Hodge--Tate weights} of $E$.
\end{defi}

\begin{ex}
The $F$-gauge $\O_{X^\Syn}\{i\}$ only has a single Hodge--Tate weight, which equals $-i$.
\end{ex}

Similar to what we have seen previously, the syntomification and the category of $F$-gauges provide a sensible notion of syntomic cohomology with coefficients:

\begin{defi}
Let $X$ be a bounded $p$-adic formal scheme which is $p$-quasisyntomic and qcqs. For an $F$-gauge $E$ on $X$, we define the \emph{syntomic cohomology} of $X$ with coefficients in $E$ as
\begin{equation*}
R\Gamma_\Syn(X, E)\coloneqq R\Gamma(X^\Syn, E)\;.
\end{equation*}
If $E=\O_{X^\Syn}\{n\}$ is a Breuil-Kisin twist, we also write $R\Gamma_\Syn(X, \Z_p(n))$ instead of $R\Gamma_\Syn(X, E)$ and we use $R\Gamma_\Syn(X, \Q_p(n))$ to denote $R\Gamma_\Syn(X, E)[\tfrac{1}{p}]$.
\end{defi}

In some sense, the category of $F$-gauges on $X$ should capture a universal notion of coefficients for various $p$-adic cohomology theories. Evidence for this is given by the existence of certain \emph{realisation functors} from $F$-gauges to coefficients for various other cohomology theories.

\begin{ex}
\label{ex:syntomification-tcrys}
By functoriality of the syntomification applied to $X_{p=0}\rightarrow X$, we obtain a pullback functor 
\begin{equation*}
T_\crys: \FGauge_\prism(X)\rightarrow\FGauge_\prism(X_{p=0})\;,
\end{equation*}
which we call the \emph{crystalline realisation}.
\end{ex}

\begin{ex}
\label{ex:syntomification-tdr+}
The Hodge-filtered de Rham map $i_{\dR, +}: X^{\dR, +}\rightarrow X^\N$ composed with the map $j_\N: X^\N\rightarrow X^\Syn$ induces a pullback functor
\begin{equation*}
T_{\dR, +}: \FGauge_\prism(X)\rightarrow \D(X^{\dR, +})\;,
\end{equation*}
which we call the \emph{Hodge-filtered de Rham realisation}. Pulling back further to $X^\dR$, we obtain a functor
\begin{equation*}
\EDIT{
T_\dR: \FGauge_\prism(X)\rightarrow \D(X^\dR)\;,
}
\end{equation*}
which we call the \emph{de Rham realisation}.
\end{ex}

There is also an \emph{étale realisation}
\begin{equation*}
T_\et: \Perf(X^\Syn)\rightarrow\D^b_\lisse(X_\eta, \Z_p)\;.
\end{equation*}
Here, $X_\eta$ denotes the generic fibre of $X$ regarded as an adic space over $\Q_p$ and $\D^b_\lisse(X_\eta, \Z_p)$ denotes the full subcategory of $\D(X_{\eta, \proet}, \Z_p)$ spanned by derived $p$-complete locally bounded objects whose mod $p$ reduction has locally constant cohomology sheaves with finitely generated stalks.

Using the arc-topology introduced in \cite{arcTopology}, it suffices to describe $T_\et$ in the case where $X=\Spf R$ is perfectoid by descent, see \cite[Constr.\ 6.3.2]{FGauges}. Here, in the notation of \cref{ex:syntomification-fgaugesperfd}, we see that $u^\infty$ and $t^\infty$ induce isomorphisms
\begin{equation*}
M^\infty[\tfrac{1}{\phi^{-1}(\xi)}]\xleftarrow{u^\infty} M^0[\tfrac{1}{\phi^{-1}(\xi)}]\xrightarrow{t^\infty} M^{-\infty}[\tfrac{1}{\phi^{-1}(\xi)}]\;.
\end{equation*}
Hence, $\tau$ induces an isomorphism
\begin{equation*}
\phi^*M^{-\infty}[\tfrac{1}{\xi}]\cong \phi^*M^\infty[\tfrac{1}{\xi}]\cong M^{-\infty}[\tfrac{1}{\xi}]
\end{equation*}
endowing $M^{-\infty}$ with the structure of a prismatic $F$-crystal on $R$. Thus, we obtain a natural map
\begin{equation}
\label{eq:syntomification-fgaugetolaurentfcrystal}
\Perf(R^\Syn)\rightarrow \Perf^\phi(A_\inf(R)[\tfrac{1}{\xi}]^\wedge_{(p)})\;,
\end{equation}
\EDIT{where the right-hand side denotes the category of perfect complexes over $A_\inf(R)[\tfrac{1}{\xi}]^\wedge_{(p)}$ equipped with an automorphism linear over the Frobenius of $A_\inf(R)[\tfrac{1}{\xi}]^\wedge_{(p)}$.} By \cite[Ex.\ 3.5]{FCrystals}, the latter category identifies with $\D^b_\lisse(\Spa(R[\tfrac{1}{p}], R), \Z_p)$ via the construction
\begin{equation}
\label{eq:syntomification-laurentfcrystaltolocsys}
M\mapsto (M\tensorL_{A_\inf(R)} W(R[\tfrac{1}{p}]^\flat))^{\phi=1}\;,
\end{equation}
where $\phi$ acts diagonally \EEDIT{and the $\phi$-invariants are taken in the derived sense, i.e.\ the right-hand side denotes the fibre of the map}
\begin{equation*}
\EEDIT{
\phi-\id: M\tensorL_{A_\inf(R)} W(R[\tfrac{1}{p}]^\flat)\rightarrow M\tensorL_{A_\inf(R)} W(R[\tfrac{1}{p}]^\flat)\;.
}
\end{equation*}

\begin{rem}
\label{rem:syntomification-etalerealisationglobal}
It is also possible to describe the étale realisation globally (i.e.\ without using descent): Namely, pullback along either one of the maps $X^\prism\rightarrow X^\N\rightarrow X^\Syn$ given by $j_\N\circ j_\HT$ or $j_\N\circ j_\dR$, respectively, induces a functor
\begin{equation}
\label{eq:syntomification-etalerealisationglobal}
\Perf(X^\Syn)\rightarrow \Perf^\phi(X_\prism)\;,
\end{equation}
where the right-hand side denotes the category of prismatic $F$-crystals in perfect complexes on $X$ as defined in \cite[Rem.\ 4.2]{FCrystals} and the Frobenius structure comes from the gluing description of $X^\Syn$ using the fact that $j_\HT$ and $j_\dR$ differ by a Frobenius twist, see \cite[Rem.\ 6.3.4]{FGauges} for details. Then the étale realisation is given by the composition
\begin{equation*}
\Perf(X^\Syn)\rightarrow \Perf^\phi(X_\prism)\rightarrow \Perf^\phi(X_\prism, \O_\prism[1/\cal{I}_\prism]_{(p)}^\wedge)\cong \D^b_\lisse(X_\eta, \Z_p)\;;
\end{equation*}
here, the last equivalence is \cite[Cor.\ 3.7]{FCrystals}.
\end{rem}

\begin{ex}
One can show that the étale realisation carries the Breuil--Kisin twists $\O_{X^\Syn}\{n\}$ to the usual Tate twist $\Z_p(n)$, see \cite[Ex.\ 4.9]{FCrystals}.
\end{ex}

In fact, in the case $X=\Spf\Z_p$, it turns out that the étale realisation defines an equivalence between a certain full subcategory of $\Perf(\Z_p^\Syn)$ and the category $\Rep_{\Z_p}^\crys(G_{\Q_p})$ of crystalline $G_{\Q_p}$-representations in finite free $\Z_p$-modules, \EDIT{as we will now recall; here, $G_{\Q_p}$ denotes the absolute Galois group of $\Q_p$.}

\begin{prop}
\label{prop:syntomification-refltofcrystal}
Let $R$ be a perfectoid valuation ring. Then the following categories are equivalent:
\begin{enumerate}[label=(\roman*)]
\item The category $\Vect^\phi(A_\inf(R))$ of prismatic $F$-crystals in vector bundles on $R$.
\item The category $\Coh^\refl(R^\Syn)$ of reflexive $F$-gauges on $R$ as defined in \cite[Def.\ 6.6.4]{FGauges}
\end{enumerate}
\end{prop}
\begin{proof}
See \cite[Cor.\ 6.6.5]{FGauges}.
\end{proof}

\begin{rem}
\label{rem:syntomification-descriptionrefltofcrystal}
The functor $\Vect^\phi(A_\inf(R))\rightarrow\Coh^\refl(R^\Syn)$ can actually be described as follows: To a prismatic $F$-crystal $M$ in vector bundles on $R$ with $\tau: \phi^*M[\tfrac{1}{\xi}]\cong M[\tfrac{1}{\xi}]$, we associate the $\phi^{-1}(\xi)$-adically filtered $A_\inf(R)$-module $\Fil^\bullet M$ given by $\Fil^i M=\phi_*(\xi^i M\cap \phi^*M)$, where the intersection takes place inside $\phi^*M[\tfrac{1}{\xi}]\cong M[\tfrac{1}{\xi}]$ via $\tau$; by \cref{ex:syntomification-fgaugesperfd}, this data will determine an $F$-gauge on $R$.
\end{rem}

\begin{thm}
\label{thm:syntomification-reflcrys}
There is an equivalence of categories
\begin{equation*}
\Coh^\refl(\Z_p^\Syn)\cong \Vect^\phi((\Z_p)_\prism)\cong\Rep_{\Z_p}^\crys(G_{\Q_p})
\end{equation*}
induced by étale realisation. (Here, an $F$-gauge on $\Z_p$ is called \emph{reflexive} if its pullback to $\O_C^\Syn$ is reflexive.)
\end{thm}
\begin{proof}
See \cite[Thm.\ 6.6.13]{FGauges}.
\end{proof}

Actually, it turns out that this allows one to show that the étale realisation of all coherent sheaves on $\Z_p^\Syn$ is a crystalline $G_{\Q_p}$-representations as soon as one inverts $p$ and that, moreover, in this setting, cohomology of a coherent sheaf agrees with the crystalline part of the Galois cohomology of its étale realisation:

\begin{prop}
\label{prop:syntomification-cohcrystalline}
For any $M\in\Coh(\Z_p^\Syn)$, the $G_{\Q_p}$-representation $T_\et(M)[\tfrac{1}{p}]$ is crystalline and the map
\begin{equation*}
R\Gamma(\Z_p^\Syn, M)[\tfrac{1}{p}]\rightarrow R\Gamma(G_{\Q_p}, T_\et(M)[\tfrac{1}{p}])
\end{equation*}
induced by étale realisation has the following properties:
\begin{enumerate}[label=(\roman*)]
\item It induces an isomorphism on $H^0$.

\item It induces an injective map on $H^1$ with image the subspace of $H^1(G_{\Q_p}, T_\et(M)[\tfrac{1}{p}])$ spanned by crystalline extensions of $\Q_p$ by $T_\et(M)[\tfrac{1}{p}]$.

\item All $H^i(\Z_p^\Syn, M)[\tfrac{1}{p}]$ for $i\geq 2$ vanish.
\end{enumerate}
\end{prop}
\begin{proof}
See \cite[Prop.\ 6.7.3]{FGauges}.
\end{proof}

We now focus on studying $\Z_p^\Syn$ further and will hence omit the subscript $\Z_p^\Syn$ from the notation for the Breuil--Kisin twists for the purpose of simplification. Namely, we will describe a special locus $\Z_{p, \red}^\Syn$ on $\Z_p^\Syn$ called the \emph{reduced locus} and this will require us to construct a certain section $v_1\in H^0(\Z_p^\Syn, \O\{p-1\}/p)$. There are various ways to do this, e.g.\ via descent from the quasiregular semiperfectoid case, using a moduli-theoretic description or via the prismatic logarithm, see \cite[Constr.\ 6.2.1]{FGauges}, \cite[§5.10.8]{Prismatization} and \cite[Constr.\ 2.7]{EtaleTateTwists}. Here, however, we outline a construction of $v_1$ that seems to be entirely new and that is based on Gardner--Madapusi's definition of $\Z_p^\N$ that we are working with in this paper.

For this, note that, while the sheaf $\O(-1)\in \D(B\G_m)$ does not descend to $B\G_m^\dR$, in characteristic $p$, its $p$-th tensor power $\O(-p)\in\D((B\G_m)_{p=0})$ does descend to $(B\G_m^\dR)_{p=0}$. This is due to the fact that changing a section $u\in \G_m(S)$ by a section $(v_1, v_2, \dots)\in\G_m^\sharp(S)$ for a characteristic $p$ ring $S$ does not change the $p$-th power of $u$ due to
\begin{equation*}
(uv_1)^p=u^p(1+(v_1-1))^p=u^p(1+(v_1-1)^p)=u^p\;;
\end{equation*}
this is because elements admitting divided powers have vanishing $p$-th power in characteristic $p$. By essentially the same argument, one sees that $u^p: \O\rightarrow\O(p)$ descends from $(\A^1_+/\G_m)_{p=0}$ to $(\A^1_+/\G_m)^\dR_{p=0}$ as well.

\begin{lem}
\label{lem:syntomification-v1lem}
On $(\Z_p^\N)_{p=0}$, there is an isomorphism $t^*\O(-1)\tensor u^*\O(p)\cong \O\{p-1\}/p$.
\end{lem}
\begin{proof}
We work over $\F_p$ throughout and thus mostly suppress the subscript $(-)_{p=0}$ from the notation. Recalling that $\O_{\Z_p^\prism}\{p-1\}/p\cong \cal{I}^{-1}/p$ from \cref{ex:prismatisation-bktwist}, all we have to show is that
\begin{equation*}
u^*\O(p)\cong \pi^*\cal{I}^{-1}/p\tensor t^*\O(p)\;.
\end{equation*}
To this end, consider the pullback diagram
\begin{equation*}
\begin{tikzcd}
\Z_p^\N\ar[d, "{(t, u)}", swap]\ar[r, "\pi"] & \Z_p^\prism\ar[d, "\widetilde{\mu}"] \\
\A^1_-/\G_m\times (\A^1_+/\G_m)^\dR\ar[r, "\mathrm{mult}"] & (\A^1_-/\G_m)^\dR
\end{tikzcd}
\end{equation*}
defining $\Z_p^\N$ and note that $\mathrm{mult}^*\O(-p)\cong \O(-p)\boxtimes\O(p)$ and that the further pullback of this along $(t, u)$ is given by $t^*\O(-p)\tensor u^*\O(p)$. Thus, considering the pullback of $\O(-p)$ from $(\A^1_-/\G_m)^\dR$ to $\Z_p^\N$ along the other path in the diagram, it suffices to prove that the pullback of $\O(-p)$ along $\widetilde{\mu}: \Z_p^\prism\rightarrow (\A^1_-/\G_m)^\dR$ is given by $\cal{I}^{-1}/p$.

We first claim that there is a map $\widetilde{\mu}^*\O(p)\rightarrow\cal{I}/p$. For this, it suffices to show that the closed substack of $(\Z_p^\prism)_{p=0}$ cut out by the ideal sheaf $\cal{I}/p$ lands inside the closed substack of $(\A^1_-/\G_m)^\dR$ cut out by $t^p: \O(p)\rightarrow\O$, which is a well-defined section by the discussion above. To see this, we invoke the fact that the former is the (mod $p$ reduction of the) Hodge--Tate stack of $\Z_p$ and thus isomorphic to $B\G_m^\sharp$, the classifying stack of the PD hull of $\G_m$ at $1$, with the quotient map $\Spec\F_p\rightarrow B\G_m^\sharp\subseteq (\Z_p^\prism)_{p=0}$ classifying the Cartier--Witt divisor $p: W(\F_p)\rightarrow W(\F_p)$, see \cite[§3.4]{APC}. As we may work fpqc locally, all we have to check is that this particular Cartier--Witt divisor lands in the vanishing locus of $t^p$ under $\widetilde{\mu}$, but this is true by the moduli description of the map $\widetilde{\mu}$.

To check that the map $\widetilde{\mu}^*\O(p)\rightarrow\cal{I}/p$ constructed above is an isomorphism, we may now pull back along the fpqc cover $F: \Z_p^\prism\rightarrow\Z_p^\prism$. However, using the commutative diagram (it is in fact a pullback square)
\begin{equation*}
\begin{tikzcd}
\Z_p^\prism\ar[r, "F"]\ar[d, "\mu", swap] & \Z_p^\prism\ar[d, "\widetilde{\mu}"] \\
\A^1_-/\G_m\ar[r] & (\A^1_-/\G_m)^\dR
\end{tikzcd}
\end{equation*}
from \cref{ex:filprism-jdrjht}, we may compute $F^*\widetilde{\mu}^*\O(p)$ alternatively as the pullback of $\O(p)$ along $\mu: \Z_p^\prism\rightarrow\A^1_-/\G_m$, which comes out to be $\cal{I}^p/p$. Moreover, as $F: \Z_p^\prism\rightarrow\Z_p^\prism$ lifts the Frobenius on $(\Z_p^\prism)_{p=0}$ by \cite[Rem.\ 3.6.2]{APC}, the pullback of $\cal{I}/p$ along $F$ is $\cal{I}^p/p$ as well and one can check that, under these identifications, the map constructed is just the identity on $\cal{I}^p/p$.
\end{proof}

\begin{defi}
Under the isomorphism $\O\{p-1\}/p\cong t^*\O(-1)\tensor u^*\O(p)$ from the lemma above, we define the section
\begin{equation*}
v_1\in H^0(\Z_p^\N, \O\{p-1\}/p)
\end{equation*}
as the one identifying with the pullback of $u^pt$.
\end{defi}

Tracing through the identifications, to see that $v_1$ actually descends to $\Z_p^\Syn$, we have to identify the pullback of the section $u^p: \O(-p)\rightarrow\O$ along $\Z_p^\prism\rightarrow (\A^1_+/\G_m)^\dR$ with the pullback of the section $t: \O(1)\rightarrow\O$ along $\mu$. However, this actually already follows from the proof above.

\begin{defi}
The \emph{reduced locus} $\Z_{p, \red}^\Syn$ of $\Z_p^\Syn$ is the vanishing locus of $(p, v_1)$. Its pullback to $\Z_p^\N$ is called the \emph{reduced locus} of $\Z_p^\N$ and denoted $\Z_{p, \red}^\N$.
\end{defi}

It turns out that there is a very concrete way to describe $\Z_{p, \red}^\Syn$ and cohomology on it. To explain this, we need to consider several components of $\Z_{p, \red}^\Syn$. Note that we are now working over $\F_p$.

\begin{enumerate}[label=(\roman*)]
\item The \emph{conjugate-filtered Hodge--Tate component} $\Z_{p, \HT, c}^\N$ is the closed substack of $(\Z_p^\N)_{p=0}$ cut out by the equation $t=0$ and one can show that $\Z_{p, \HT, c}^\N\cong \G_{a, +}/(\G_a^\sharp\rtimes\G_m)$, see \cite[Prop.\ 5.3.7]{FGauges}. Thus, the stack $\Z_{p, \HT, c}^\N$ admits the following stratification:
\begin{enumerate}[label=(\arabic*)]
\item The substack $\Z_{p, \HT}^\N\coloneqq B\Stab_{\G_a^\sharp\rtimes\G_m}(1\in\G_a)\cong B\G_m^\sharp$ is open and one can show that it identifies with $j_\HT(\Z_p^\prism)\cap \Z_{p, \red}^\N$.

\item The substack $\Z_{p, \Hod}^\N\coloneqq B\Stab_{\G_a^\sharp\rtimes\G_m}(0\in\G_a)\cong B(F_*\G_a^\sharp\rtimes\G_m)$ is closed and identifies with the reduced complement of $\Z_{p, \HT}^\N$.
\end{enumerate}

\item The \emph{Hodge-filtered de Rham component} $\Z_{p, \dR, +}^\N$ is the closed substack of $(\Z_p^\N)_{p=0}$ cut out by the equation $u^p=0$. As $i_{\dR, +}: \A^1_-/\G_m\rightarrow \Z_{p, \dR, +}^\N$ is an fpqc cover splitting the restriction $t_{\dR, +}: \Z_{p, \dR, +}^\N\rightarrow\A^1_-/\G_m$ of the Rees map, we have $\Z_{p, \dR, +}^\N\cong B_{\A^1_-/\G_m}\cal{G}$, where $\cal{G}$ denotes the group scheme $\Aut(i_{\dR, +})$. One can explicitly compute $\cal{G}$, see \cite[Ch.\ 7]{Prismatization}; for us, however, it is only important that $\Z_{p, \dR, +}^\N$ admits the following stratification:
\begin{enumerate}[label=(\arabic*)]
\item The preimage of $\Spec\F_p\subseteq\A^1_-/\G_m$ under $t_{\dR, +}$ identifies with the open substack $j_\dR(\Z_p^\prism)\cap \Z_{p, \red}^\N$ and \EDIT{is} denoted $\Z_{p, \dR}^\N$. Up to a Frobenius twist, which is trivial here, it is isomorphic to $\Z_{p, \HT}^\N$.

\item The preimage of $B\G_m\subseteq\A^1_-/\G_m$ under $t_{\dR, +}$ is isomorphic to $\Z_{p, \Hod}^\N$.
\end{enumerate}
\end{enumerate}

To obtain $\Z_{p, \red}^\N$ from these components, one glues $\Z_{p, \HT, c}^\N$ and $\Z_{p, \dR, +}^\N$ along $\Z_{p, \Hod}^\N$. Further gluing $\Z_{p, \HT}^\N$ and $\Z_{p, \dR}^\N$ along a Frobenius twist, which is trivial in our situation, then yields $\Z_{p, \red}^\Syn$. Thus, to study cohomology on $\Z_{p, \red}^\Syn$, we need to study cohomology on the loci described above, starting with $\Z_{p, \dR}^\N$ and $\Z_{p, \Hod}^\N$:

\begin{enumerate}[label=(\arabic*)]
\item For $\Z_{p, \dR}^\N\cong B\G_m^\sharp$, one can show that
\begin{equation*}
\D(\Z_{p, \dR}^\N)\cong \D_{(\Theta^p-\Theta)-\nilp}(\F_p[\Theta])\;,
\end{equation*}
where the latter category consists of $\F_p$-complexes $V$ equipped an operator $\Theta$ such that the action of $\Theta^p-\Theta$ on cohomology is locally nilpotent, see \cite[Thm.\ 3.5.8]{APC}. Under this equivalence, the Breuil--Kisin twist $\O\{n\}$ corresponds to $(\F_p, \Theta=n)$ and if $E\in\D(\Z_{p, \dR}^\N)$ corresponds to $(V, \Theta)$, then
\begin{equation*}
R\Gamma(\Z_{p, \dR}^\N, E)=\fib(V\xrightarrow{\Theta} V)\;,
\end{equation*}
see \cite[Ex.\ 3.5.6, Prop.\ 3.5.11]{APC}.

\item For $\Z_{p, \Hod}^\N\cong B(F_*\G_a^\sharp\rtimes\G_m)$, one can show that
\begin{equation*}
\D(\Z_{p, \Hod}^\N)\cong \D_{\gr, \Theta-\nilp}(\F_p[\Theta])\;,
\end{equation*}
where $\Theta$ has degree $-p$, i.e.\ the latter category consists of $\Z$-indexed collections of $\F_p$-complexes $V_i$ equipped with operators $\Theta_i: V_i\rightarrow V_{i-p}$ such that the action of $\Theta=\bigoplus_i \Theta_i$ on cohomology is locally nilpotent. Under this equivalence, the Breuil--Kisin twist $\O\{n\}$ corresponds to the vector space $\F_p$ sitting in grading degree $-n$ and if $E\in\D(\Z_{p, \Hod}^\N)$ corresponds to \EEDIT{$(\{V_i\}, \{\Theta_i\})$}, then
\begin{equation*}
R\Gamma(\Z_{p, \Hod}^\N, E)=\fib(V_0\xrightarrow{\Theta} V_{-p})\;.
\end{equation*}
\end{enumerate}

On $\Z_{p, \dR, +}^\N$ and $\Z_{p, \HT, c}^\N$, we obtain the following descriptions of quasi-coherent complexes:

\begin{enumerate}[label=(\roman*)]
\item For $\Z_{p, \HT, c}^\N\cong \G_a/(\G_a^\sharp\rtimes\G_m)$, one can show that 
\begin{equation*}
\D(\Z_{p, \HT, c}^\N)\cong\D_{\gr, D-\nilp}(\cal{A}_1)\;,
\end{equation*}
where $\cal{A}_1\coloneqq\F_p\{x, D\}/(Dx-xD-1)$ with $D$ having degree \EEDIT{$-1$} and $x$ having degree \EEDIT{$1$}; here, the notation $\F_p\{x, D\}$ refers to the free (non-commutative!) $\F_p$-algebra on the variables $x$ and $D$. Alternatively, via the Rees equivalence, we may describe any $E\in\D(\Z_{p, \HT, c}^\N)$ as an increasingly filtered $\F_p$-complex $\Fil_\bullet V$ together with operators $D: \Fil_\bullet V\rightarrow \Fil_{\bullet-1} V$ which are locally nilpotent in cohomology such that $Dx=xD+1$, where $x: \Fil_\bullet V\rightarrow\Fil_{\bullet+1} V$ denotes the transition maps. \EEDIT{We warn the reader that, in particular, this relation means that the operators $xD: \Fil_\bullet V\rightarrow\Fil_\bullet V$ are \emph{not} compatible with the filtration; however, since we are working over $\F_p$ and hence $xD-p=xD$, the operator $xD$ \emph{is} compatible with the coarser filtration
\begin{equation*}
\Fil_{p\bullet} V = (\dots\rightarrow \Fil_{-p} V\rightarrow \Fil_0 V\rightarrow\Fil_p V\rightarrow\dots)\;.
\end{equation*}
Under the above description of $\D(\Z_{p, \HT, c}^\N)$, the Breuil--Kisin twist $\O\{n\}$ corresponds to the filtered complex $\Fil_\bullet V$ with $\Fil_\bullet V=\F_p$ if $\bullet\geq -n$ and zero else with the transition maps between nonzero terms being the identity; moreover, the operator $D: \Fil_i V\rightarrow\Fil_{i-1} V$ is given by multiplication by $i$ for all $i\in\Z$.}

If $E$ corresponds to a graded $\cal{A}_1$-module $M$, then restriction to $\Z_{p, \dR}^\N$ corresponds to passing to $(M[1/x^p]_{\deg=0}, \Theta=xD)$ and one may view this as passing from \EEDIT{$\Fil_{p\bullet} V$} to the underlying non-filtered complex and taking $\Theta=xD$. Restriction to $\Z_{p, \Hod}^\N$, however, corresponds to passing to the associated graded $\gr_\bullet V$ of $\Fil_\bullet V$ and taking $\Theta=D^p$. Finally, we have
\begin{equation*}
R\Gamma(\Z_{p, \HT, c}^\N, E)=\fib(\Fil_0 V\xrightarrow{D} \Fil_{-1} V)\;.
\end{equation*}

\item For $\Z_{p, \dR, +}^\N\cong B_{\A^1_-/\G_m}\cal{G}$, one can show that any $E\in\D(\Z_{p, \dR, +}^\N)$ corresponds to a \EEDIT{decreasingly} filtered complex $\Fil^\bullet V\in\DF(\F_p)$ together with an operator $\Theta: \Fil^\bullet V\rightarrow\Fil^{\bullet-p} V$ such that the action of $\Theta^p-t^{p^2-p}\Theta$ on $\Rees(\Fil^\bullet V)$ is locally nilpotent in cohomology, see \cite[Prop.\ 6.5.6]{FGauges}. Under this correspondence, the Breuil--Kisin twist $\O\{n\}$ corresponds to the filtered complex $\Fil^\bullet V$ with $\Fil^\bullet V=\F_p$ if $\bullet\leq -n$ and zero else with the transition maps between nonzero terms being the identity and $\Theta=n$. For any $E\in\D(\Z_{p, \dR, +}^\N)$ corresponding to $(\Fil^\bullet V, \Theta)$, restriction to $\Z_{p, \dR}^\N$ corresponds to passage to the underlying non-filtered complex and restriction to $\Z_{p, \Hod}^\N$ corresponds to passage to the associated graded; moreover, we have
\begin{equation*}
R\Gamma(\Z_{p, \dR, +}^\N, E)=\fib(\Fil^0V\xrightarrow{\Theta} \Fil^{-p}V)\;.
\end{equation*}
\end{enumerate}

For $E\in\D(\Z_{p, \red}^\Syn)$, let us introduce the notation $F_\dR(E)$ for $R\Gamma(\Z_{p, \dR}^\N, E|_{\Z_{p, \dR}^\N})$ and define $F_\Hod(E), F_{\HT, c}(E)$ and $F_{\dR, +}(E)$ analogously. Then we have natural restriction maps
\begin{equation*}
\begin{split}
a_E: F_{\dR, +}(E)&\rightarrow F_\dR(E)\oplus F_\Hod(E)\;, \\
b_E: F_{\HT, c}(E)&\rightarrow F_\dR(E)\oplus F_\Hod(E)
\end{split}
\end{equation*}
\EEDIT{and} the gluing description of $\Z_{p, \red}^\Syn$ shows that
\begin{equation}
\label{eq:syntomification-cohomologyreduced}
R\Gamma(\Z_{p, \red}^\Syn, E)=\fib(F_{\dR, +}(E)\oplus F_{\HT, c}(E)\xrightarrow{a_E-b_E} F_\dR(E)\oplus F_\Hod(E))\;.
\end{equation}

As we now have a good handle on cohomology on $\Z_{p, \red}^\Syn$, we want to make use of this to understand cohomology on $\Z_p^\Syn$.

\begin{defi}
\label{defi:syntomification-fil}
Let \EEDIT{$E\in\D((\Z_p^\Syn)_{p=0})$}. The filtration
\begin{equation*}
\begin{tikzcd}[column sep=scriptsize]
\dots\ar[r, "v_1"] & R\Gamma(\Z_p^\Syn, E\{-(p-1)\})\ar[r, "v_1"] & R\Gamma(\Z_p^\Syn, E)\ar[r, "v_1"] & R\Gamma(\Z_p^\Syn, E\{p-1\})\ar[r, "v_1"] & \dots  
\end{tikzcd}
\end{equation*}
is called the \emph{syntomic filtration} and denoted $\Fil_\bullet^\Syn R\Gamma(\Z_p^\Syn, E)[\frac{1}{v_1}]$; here, $E\{n\}$ denotes the tensor product $E\tensor\O\{n\}$ for any $n$. We denote the underlying unfiltered object by $R\Gamma(\Z_p^\Syn, E)[\frac{1}{v_1}]$.
\end{defi}

\EDIT{As one can show that $v_1$ is topologically nilpotent, the syntomic filtration is complete, see \cref{prop:syntomicetale-filtrationetale}. Moreover, notice} that, for any \EEDIT{$E\in\D((\Z_p^\Syn)_{p=0})$}, there is a canonical isomorphism
\begin{equation}
\label{eq:syntomification-grsyn}
\gr^\Syn_\bullet R\Gamma(\Z_p^\Syn, E)[\tfrac{1}{v_1}]\cong R\Gamma(\Z_{p, \red}^\Syn, E/v_1\{\bullet(p-1)\})\;,
\end{equation}
where $E/v_1\coloneqq\cofib(E\{-(p-1)\}\xrightarrow{v_1} E)$. Moreover, \EEDIT{if $E\in\Perf((\Z_p^\Syn)_{p=0})$}, one can check that there is a natural isomorphism
\begin{equation}
\label{eq:syntomification-synet}
R\Gamma(\Z_p^\Syn, E)[\tfrac{1}{v_1}]\cong R\Gamma(G_{\Q_p}, T_\et(E))\;,
\end{equation}
see \cite[Eq. (6.4.1)]{FGauges}.

\section{The stacky Beilinson fibre square}
\label{sect:beilfibsq}

\EDIT{Recall the following theorem of Antieau--Mathew--Morrow--Nikolaus relating syntomic cohomology to Hodge-filtered de Rham cohomology:}

\begin{thm}
\label{thm:beilfibsq-motivation}
Let $X$ be a \EEDIT{smooth qcqs} $p$-adic formal scheme. For each $i\geq 0$, there is a functorial pullback square
\begin{equation*}
\begin{tikzcd}
R\Gamma_\Syn(X, \Z_p(i))[\frac{1}{p}]\ar[r]\ar[d] & R\Gamma_\Syn(X_{p=0}, \Z_p(i))[\frac{1}{p}]\ar[d] \\
\Fil^i_\Hod R\Gamma_\dR(X)[\frac{1}{p}]\ar[r] & R\Gamma_\dR(X)[\frac{1}{p}]\nospacepunct{\;.}
\end{tikzcd}
\end{equation*}
\end{thm}
\begin{proof}
\EDIT{The affine version is \cite[Thm.\ 6.17]{BeilFibSq} and we deduce the general version using Zariski descent and \cref{lem:syntomicetale-colimtot}.}
\end{proof}

\EDIT{We want to prove a stacky version of this result and thereby also generalise it to allow arbitrary coefficients, at least in the case where $X$ is proper.} Namely, the main theorem we are going to prove is the following:

\begin{thm}
\label{thm:beilfibsq-main}
There is a commutative square of stacks
\begin{equation*}
\begin{tikzcd}
\Z_p^\dR\ar[r]\ar[d] & \Z_p^{\dR, +}\ar[d, "i_{\dR, +}"] \\
\F_p^\Syn\ar[r] & \Z_p^\Syn
\end{tikzcd}
\end{equation*}
which is an almost pushout up to $p$-isogeny. \EDIT{More precisely,} for any $E\in\Perf(\Z_p^\Syn)$, it induces a pullback diagram
\begin{equation*}
\begin{tikzcd}
R\Gamma(\Z_p^\Syn, E)[\frac{1}{p}]\ar[r]\ar[d] & R\Gamma(\F_p^\Syn, T_\crys(E))[\frac{1}{p}]\ar[d] \\
R\Gamma(\Z_p^{\dR, +}, T_{\dR, +}(E))[\frac{1}{p}]\ar[r] & R\Gamma(\Z_p^\dR, T_\dR(E))[\frac{1}{p}]\nospacepunct{\;.}
\end{tikzcd}
\end{equation*}
\end{thm}

This allows us to generalise \cref{thm:beilfibsq-motivation} to accommodate $F$-gauge coefficients in the case where $X$ is proper in the following way:

\begin{cor}
\label{cor:beilfibsq-coeffs}
Let $X$ be a $p$-adic formal scheme which is smooth and proper over $\Spf\Z_p$. For any perfect $F$-gauge $E\in\Perf(X^\Syn)$, there is a natural pullback square
\begin{equation*}
\begin{tikzcd}
R\Gamma_\Syn(X, E)[\frac{1}{p}]\ar[r]\ar[d] & R\Gamma_\Syn(X_{p=0}, T_\crys(E))[\frac{1}{p}]\ar[d] \\
\EDIT{\Fil^0_\Hod R\Gamma_\dR(X, T_{\dR, +}(E))[\frac{1}{p}]}\ar[r] & R\Gamma_\dR(X, T_\dR(E))[\frac{1}{p}]\nospacepunct{\;.}
\end{tikzcd}
\end{equation*}
\end{cor}
\begin{proof}
Using \cref{prop:finiteness-main}, this follows immediately from \cref{thm:beilfibsq-main}.
\end{proof}

In particular, by the comparison from \cref{thm:fildrstack-comparison}, the above recovers the result of \cref{thm:beilfibsq-motivation} \EDIT{in the proper case} by putting $E=\O\{i\}$.

\subsection{Construction of the square}
\label{subsect:beilfibsq-construction}

We start by constructing the stacky Beilinson fibre square. Namely, using the open immersion $\Z_p^\dR\rightarrow\Z_p^{\dR, +}$ from \cref{rem:fildrstack-unfiltereddr} and the isomorphism $\F_p^\prism\cong\Spf\Z_p\cong\Z_p^\dR$ from \cref{ex:prismatisation-qrsp}, we obtain a square
\begin{equation}
\label{eq:beilfibsq-constructsquare}
\begin{tikzcd}
\F_p^\prism\cong\Z_p^\dR\ar[d, "j_\prism", swap]\ar[r] & \Z_p^{\dR, +}\ar[d, "i_{\dR, +}"] \\
\F_p^\Syn\ar[r] & \Z_p^\Syn\nospacepunct{\;.}
\end{tikzcd}
\end{equation}

\begin{prop}
\label{prop:beilfibsq-square}
The diagram (\ref{eq:beilfibsq-constructsquare}) is commutative.
\end{prop}
\begin{proof}
Clearly, it suffices to show that
\begin{equation*}
\begin{tikzcd}
\F_p^\prism\cong\Z_p^\dR\ar[d, "j_\dR", swap]\ar[r] & \Z_p^{\dR, +}\ar[d, "i_{\dR, +}"] \\
\F_p^\N\ar[r] & \Z_p^\N\nospacepunct{\;.}
\end{tikzcd}
\end{equation*}
commutes and this equates to showing that the composite map $\F_p^\prism\xrightarrow{j_\dR}\F_p^\N\rightarrow\Z_p^\N$ identifies with the de Rham map $i_\dR: \Z_p^\dR\rightarrow\Z_p^\N$. By naturality of the map $j_\dR$, the diagram
\begin{equation*}
\begin{tikzcd}
\F_p^\prism\ar[r]\ar[d, "j_\dR", swap] & \Z_p^\prism\ar[d, "j_\dR"] \\
\F_p^\N\ar[r] & \Z_p^\N
\end{tikzcd}
\end{equation*}
commutes and thus we may consider the composition $\F_p^\prism\rightarrow\Z_p^\prism\xrightarrow{j_\dR}\Z_p^\N$ instead. However, under the identification $\F_p^\prism\cong\Spf\Z_p$ from \cref{ex:prismatisation-qrsp}, the map $\F_p^\prism\rightarrow\Z_p^\prism$ identifies with the map $\Spf\Z_p\rightarrow\Z_p^\prism$ sending any $p$-nilpotent ring $S$ to the Cartier--Witt divisor $W(S)\xrightarrow{p} W(S)$ by the proof of \cite[Prop.\ 5.4.2]{FGauges} and this shows that the composition $\F_p^\prism\rightarrow\Z_p^\prism\xrightarrow{j_\dR}\Z_p^\N$ identifies with $i_\dR: \Spf\Z_p\cong\Z_p^\dR\rightarrow\Z_p^\N$, as claimed.
\end{proof}

\subsection{A Beilinson-type fibre square for crystalline representations}
\label{subsect:beilfibsq-crys}

Our strategy to prove \cref{thm:beilfibsq-main} will be to make use of \cref{prop:syntomification-cohcrystalline}. Thus, we will start by quickly reviewing some of the theory of crystalline representations. For more details on this topic, we refer to \cite{Fontaine}.

Recall that we use $C$ to denote a \EDIT{fixed} completed algebraic closure of $\Q_p$ and recall the rings
\begin{equation*}
\EEDIT{
A_\inf\coloneqq A_\inf(\O_C)=W(\O_C^\flat)\;,\hspace{0.5cm} A_\crys\coloneqq A_\crys(\O_C)=A_\inf[\tfrac{\xi^n}{n!}: n\geq 0]_{(p)}^\wedge
}
\end{equation*}
from \cref{ex:drstack-perfd}, where \EEDIT{$\xi$} \EDIT{is a generator of} the kernel of the Fontaine map $\theta: A_\inf\rightarrow\O_C$. Choosing once and for all a compatible system $\epsilon=(1, \zeta_p, \zeta_{p^2}, \dots)\in\O_C^\flat$ of $p$-power roots of unity, one can actually take
\begin{equation*}
\EEDIT{
\xi=\frac{[\epsilon]-1}{[\epsilon^{1/p}]-1}=1+[\epsilon^{1/p}]+\dots+[\epsilon^{(p-1)/p}]\;,
}
\end{equation*}
where $[\hspace{2pt}\cdot\hspace{2pt}]: \O_C^\flat\rightarrow W(\O_C^\flat)$ is the Teichmüller representative. \EDIT{One can show that} $A_\crys$ contains \EDIT{the element}
\begin{equation}
\label{eq:beilfibsq-t}
t=\log[\epsilon]=-\sum_{k=1}^\infty \frac{(1-[\epsilon])^k}{k}\;,
\end{equation}
\EDIT{see \cite[Prop.\ 7.6]{Fontaine},} and \EDIT{thus} we define rings
\begin{equation*}
B_\crys^+\coloneqq A_\crys[\tfrac{1}{p}]\;,\hspace{0.5cm} B_\crys\coloneqq B_\crys^+[\tfrac{1}{t}]\;;
\end{equation*}
note that one can actually show that $B_\crys=A_\crys[\frac{1}{t}]$. The ring $B_\crys$ is equipped with a natural $G_{\Q_p}$-action and, perhaps most importantly, the Frobenius $\phi$ of $A_\inf$ can be continued to an endomorphism of $B_\crys$. We moreover define a ring
\begin{equation*}
B_\dR^+\coloneqq A_\inf[\tfrac{1}{p}]_{(\xi)}^\wedge\;,
\end{equation*}
which is a complete discrete valuation ring with uniformiser $\xi$ and hence its fraction field $B_\dR\coloneqq B_\dR[\tfrac{1}{\xi}]$ is equipped with a natural descending \EEDIT{$\xi$-adic} filtration
\begin{equation*}
\Fil^i B_\dR\coloneqq \xi^iB_\dR^+\;.
\end{equation*}
Note that $B_\crys^+\subseteq B_\dR^+$ and $B_\crys\subseteq B_\dR$.

\begin{defi}
A finite-dimensional $G_{\Q_p}$-representation $V$ over $\Q_p$ is called \emph{crystalline} if there is a $G_{\Q_p}$-equivariant isomorphism
\begin{equation*}
V\tensor_{\Q_p} B_\crys\cong B_\crys^{\dim V}\;.
\end{equation*}
We use $\Rep_{\Q_p}^\crys(G_{\Q_p})$ to denote the category of crystalline $G_{\Q_p}$-representations.
\end{defi}

For a crystalline representation $V$, we write
\begin{equation*}
D_\crys(V)\coloneqq (V\tensor_{\Q_p} B_\crys)^{G_{\Q_p}}\;,
\end{equation*}
where $G_{\Q_p}$ acts diagonally. \EDIT{From} $B_\dR^{G_{\Q_p}}=B_\crys^{G_{\Q_p}}=\Q_p$, \EDIT{see \cite[Prop.\ 6.26]{Fontaine}, we conclude that} $D_\crys(V)$ is a $\Q_p$-vector space of dimension $\dim V$. By virtue of the Frobenius on $B_\crys$, it comes equipped with a natural Frobenius action $\phi$ and due to 
\begin{equation*}
D_\crys(V)=(V\tensor_{\Q_p} B_\dR)^{G_{\Q_p}}\;,
\end{equation*}
\EDIT{which again comes from $B_\dR^{G_{\Q_p}}=B_\crys^{G_{\Q_p}}$ and $V$ being crystalline}, it also inherits a natural filtration from $B_\dR$. Thus, $D_\crys$ defines a functor
\begin{equation*}
D_\crys: \Rep_{\Q_p}^\crys(G_{\Q_p})\rightarrow\MF^\phi(\Q_p)
\end{equation*}
from crystalline representations into the category of filtered $\phi$-modules over $\Q_p$, which is defined as follows:

\begin{defi}
A \emph{$\phi$-module} over $\Q_p$ is a $\Q_p$-vector space $M$ together with the data of an isomorphism $\phi^*M\cong M$, where $\phi$ denotes the Frobenius on $\Q_p$, and we denote the category of $\phi$-modules over $\Q_p$ by $\Mod^\phi(\Q_p)$. A \emph{filtered $\phi$-module} over $\Q_p$ is a filtered $\Q_p$-vector space $D\in\MF(\Q_p)\coloneqq\Fun((\Z, \geq), \Mod(\Q_p))$ equipped with the data of a $\phi$-module structure on $\ul{D}=\colim_i D(i)$.
\end{defi}

Note that our definition of a filtered $\phi$-module differs slightly from the usual one in that we also allow non-honest filtrations (i.e.\ filtrations with non-injective transition maps) and that we furthermore do not require the filtration to be \EDIT{complete} as one would usually do. This has the advantage that the category $\MF^\phi(\Q_p)$ is clearly abelian (while this is not the case with the usual definition) and it is also more convenient for our goal of proving \cref{thm:beilfibsq-main}.

It is a theorem of Fontaine that the functor $D_\crys$ defines an equivalence from $\Rep_{\Q_p}^\crys(G_{\Q_p})$ onto an abelian subcategory of $\MF^\phi(\Q_p)$ denoted $\MF^\phi_\adm(\Q_p)$, \EDIT{see \EEDIT{\cite[§5.5.2.(iii)]{Semistable}},} and a filtered $\phi$-module in the essential image of $D_\crys$ is called \emph{admissible}.

Note that we have exact forgetful functors
\begin{equation*}
\begin{split}
T_{\dR, +}: \MF^\phi(\Q_p)&\rightarrow \MF(\Q_p)\;, \\
T_\crys: \MF^\phi(\Q_p)&\rightarrow\Mod^\phi(\Q_p)\;, \\
T_\dR: \MF^\phi(\Q_p)&\rightarrow\Mod(\Q_p)
\end{split}
\end{equation*}
defined by forgetting the $\phi$-structure, the filtration or both, respectively. Note that we use the same notation as for the realisation functors for $F$-gauges, but it will always be clear from the context which functor we mean. Our main goal in this section is to prove the following statement:

\begin{prop}
\label{prop:beilfibsq-crys}
Let $D\in\D^b(\MF^\phi_\adm(\Q_p))$ be a bounded complex of admissible filtered $\phi$-modules. Then the forgetful functors above induce a pullback square
\begin{equation*}
\begin{tikzcd}
\RHom_{\MF^\phi_\adm(\Q_p)}(\Q_p, D)\ar[r]\ar[d] & \RHom_{\Mod^\phi(\Q_p)}(\Q_p, T_\crys(D))\ar[d] \\
\RHom_{\MF(\Q_p)}(\Q_p, T_{\dR, +}(D))\ar[r] & \RHom_{\Q_p}(\Q_p, T_\dR(D))\nospacepunct{\;.}
\end{tikzcd}
\end{equation*}
\end{prop}

Note that, as the category $\MF^\phi_\adm(\Q_p)$ does not have enough injectives \EDIT{(in fact, it does not have any nonzero injective object by \cite[Rem.\ 1.9]{Bannai})}, the notation $\RHom_{\MF^\phi_\adm(\Q_p)}(\Q_p, D)$ does not make sense a priori. However, whenever we have an abelian category $\cal{A}$, we can embed it into its Ind-category $\Ind(\cal{A})$ and this will have enough injectives. By \cite[Thm.\ 3.5]{Oort}, we have
\begin{equation*}
\Ext^i_{\cal{A}}(M, N)\cong H^i(\RHom_{\Ind(\cal{A})}(M, N))\;,
\end{equation*}
where we use Yoneda Ext on the left-hand side, and hence we may define 
\begin{equation*}
\RHom_{\cal{A}}(M, N)\coloneqq\RHom_{\Ind(\cal{A})}(M, N)
\end{equation*}
as this recovers the usual definition in the case where $\cal{A}$ has enough injectives.

\begin{proof}[Proof of \cref{prop:beilfibsq-crys}]
First observe that we can readily calculate
\begin{equation*}
\begin{split}
\RHom_{\Q_p}(\Q_p, T_\dR(D))&=\ul{D}\;, \\
\RHom_{\MF(\Q_p)}(\Q_p, T_{\dR, +}(D))&=\Fil^0\ul{D}\;, \\
\RHom_{\Mod^\phi(\Q_p)}(\Q_p, T_\crys(D))&=\ul{D}^{\phi=1}\;,
\end{split}
\end{equation*}
\EDIT{where 
\begin{equation*}
\ul{D}^{\phi=1}\coloneqq\fib(\ul{D}\xrightarrow{\phi-1}\ul{D})
\end{equation*}
denotes the derived $\phi$-invariants of $\ul{D}$,} and hence we only need to show that
\begin{equation*}
\RHom_{\MF^\phi_\adm(\Q_p)}(\Q_p, D)\cong \fib(\ul{D}^{\phi=1}\rightarrow\ul{D}/\Fil^0\ul{D})\;,
\end{equation*}
where the quotient $\ul{D}/\Fil^0\ul{D}$ is to be taken in the derived sense, i.e.\ 
\begin{equation*}
\ul{D}/\Fil^0\ul{D}\coloneqq\cofib(\Fil^0\ul{D}\rightarrow\ul{D})\;.
\end{equation*}
However, this follows from \cite[Cor.\ 2.4.4]{CrystallineExt}.
\end{proof}

\begin{rem}
\label{rem:beilfibsq-rhommfphionetrunc}
The above proof also shows that, if $D$ is an admissible filtered $\phi$-module, then the complex $\RHom_{\MF^\phi_\adm(\Q_p)}(\Q_p, D)$ is concentrated in degrees $0$ and $1$: \EDIT{Indeed, in this case the map $\Fil^0\ul{D}\rightarrow\ul{D}$ is injective and hence $\ul{D}/\Fil^0\ul{D}$ is concentrated in degree zero.}
\end{rem}

\subsection{{É}tale realisation and the Beilinson fibre square}

\label{sect:etrealisationandbeilfibsq}

Recall from \cref{prop:syntomification-cohcrystalline} that, for any $E\in\Coh(\Z_p^\Syn)$, the Galois representation $T_\et(E)[\tfrac{1}{p}]$ is crystalline. Hence, we can associate to $E$ a filtered $\phi$-module $D_\crys(T_\et(E)[\tfrac{1}{p}])$ and, in this section, we will show how this filtered $\phi$-module can also be obtained in a geometric manner.

To explain this, recall from \cref{ex:syntomification-fgaugesperfd} that an $F$-gauge $E$ on $\F_p$ corresponds to a diagram 
\begin{equation}
\label{eq:beilfibsq-fgaugefp}
\begin{tikzcd}
\dots \ar[r,shift left=.5ex,"t"]
  & \ar[l,shift left=.5ex, "u"] M^{i+1} \ar[r,shift left=.5ex,"t"] & \ar[l,shift left=.5ex, "u"] M^i \ar[r,shift left=.5ex,"t"] & \ar[l,shift left=.5ex, "u"] M^{i-1} \ar[r,shift left=.5ex,"t"] & \ar[l,shift left=.5ex, "u"] \dots
\end{tikzcd}
\end{equation}
of $p$-complete $\Z_p$-complexes such that $ut=tu=p$ together with an isomorphism $\tau: \phi^*M^\infty\cong M^{-\infty}$. Base changing everything to $\Q_p$, we see that all maps in the diagram above become isomorphisms and the only remaining data is $M^{-\infty}$ together with the isomorphism $\phi^*M^{-\infty}[\frac{1}{p}]\cong M^{-\infty}[\frac{1}{p}]$ induced by $\tau$. Sending $E$ to $M^{-\infty}[\frac{1}{p}]$ equipped with its Frobenius structure determines a functor
\begin{equation}
\label{eq:beilfibsq-dfpsynisog}
(-)[\tfrac{1}{p}]: \D(\F_p^\Syn)\rightarrow\Mod_{\Q_p[x, x^{-1}]}(\D(\Q_p))\cong\D(\Mod^\phi(\Q_p))\;,
\end{equation}
where the last equivalence is due to \cite[Thm.\ 7.1.3.1]{HA}, and we remind the reader that \EDIT{the restriction of this functor to $\Perf(\F_p^\Syn)$ exactly recovers} the functor from (\ref{eq:syntomification-fgaugetolaurentfcrystal}). \EDIT{Finally, recalling that the pullback $j_\prism^*E\in\D(\F_p^\prism)\cong\widehat{\D}(\Z_p)$, where we have used \cref{ex:prismatisation-qrsp}, identifies with $M^{-\infty}$, we see that the functor (\ref{eq:beilfibsq-dfpsynisog}) yields a commutative diagram
\begin{equation}
\label{eq:beilfibsq-crysfrob}
\begin{tikzcd}[ampersand replacement=\&, column sep=large]
\D(\F_p^\Syn)\ar[r, "j_\prism^*"]\ar[dr, "{(-)[\tfrac{1}{p}]}", swap] \& \D(\F_p^\prism)\cong \widehat{\D}(\Z_p)\ar[r, "{(-)[\tfrac{1}{p}]}"] \& \D(\Q_p) \\
\&  \D(\Mod^\phi(\Q_p))\ar[ur] \& \nospacepunct{\;,}
\end{tikzcd}
\end{equation}
where $\D(\Mod^\phi(\Q_p))\rightarrow \D(\Q_p)$ is the forgetful functor.}

\begin{lem}
\label{lem:beilfibsq-rationalcohfpsyn}
For any $E\in\D(\F_p^\Syn)$, the functor $(-)[\frac{1}{p}]$ induces an isomorphism
\begin{equation*}
R\Gamma(\F_p^\Syn, E)[\tfrac{1}{p}]\cong \RHom_{\Mod^\phi(\Q_p)}(\Q_p, E[\tfrac{1}{p}])\;.
\end{equation*}
\end{lem}
\begin{proof}
Recall from \cref{ex:syntomification-fgaugesperfd} that
\begin{equation*}
R\Gamma(\F_p^\Syn, E)=\fib(M^0\xrightarrow{t^\infty-\tau u^\infty} M^{-\infty})
\end{equation*}
in the above notation. However, as $t^\infty$ and $u^\infty$ become isomorphisms upon inverting $p$, this yields
\begin{equation*}
R\Gamma(\F_p^\Syn, E)[\tfrac{1}{p}]\cong\fib(M^{-\infty}[\tfrac{1}{p}]\xrightarrow{1-\phi} M^{-\infty}[\tfrac{1}{p}])\cong\RHom_{\Mod^\phi(\Q_p)}(\Q_p, M^{-\infty}[\tfrac{1}{p}])\;,
\end{equation*}
so we are done.
\end{proof}

\begin{lem}
\label{lem:beilfibsq-etalephimod}
For any $E\in\Coh(\Z_p^\Syn)$, there is a natural isomorphism
\begin{equation*}
T_\crys(E)[\tfrac{1}{p}]\cong T_\crys(D_\crys(T_\et(E)[\tfrac{1}{p}]))\;.\footnote{\EDIT{Beware that, on the right-hand side, we use the functor $T_\crys: \MF^\phi(\Q_p)\rightarrow\Mod^\phi(\Q_p)$, while $T_\crys(E)$ on the left-hand side denotes the pullback of $E$ to $\F_p^\Syn$ as defined in \cref{ex:syntomification-tcrys}.}}
\end{equation*}
\end{lem}
\begin{proof}
Since any coherent $F$-gauge on $\Z_p$ is isomorphic to a reflexive $F$-gauge up to $p$-isogeny by \cite[Prop.\ 6.7.1]{FGauges}, we may assume that $E$ is reflexive. Let $\cal{E}\in\Vect^\phi((\Z_p)_\prism)$ denote the prismatic $F$-crystal in vector bundles corresponding to $E$ by \cref{thm:syntomification-reflcrys}. \EEDIT{Identifying $T_\crys(E)$ with a diagram (\ref{eq:beilfibsq-fgaugefp}) and writing $M^{-\infty}\coloneqq(\colim_i M^{-i})_{(p)}^\wedge$ as usual, the module} $M^{-\infty}$ identifies with the pullback of $\cal{E}$ to the absolute prismatic site of $\F_p$, i.e.\ $T_\crys(E)[\tfrac{1}{p}]\cong\cal{E}(\Z_p)[\frac{1}{p}]$, so we have to show
\begin{equation}
\label{eq:beilfibsq-etalephimod}
\cal{E}(\Z_p)[\tfrac{1}{p}]\cong (T_\et(E)\tensor_{\Z_p} B_\crys)^{G_{\Q_p}}
\end{equation}
as $\phi$-modules. Recalling that
\begin{equation*}
T_\et(E)=(\cal{E}(A_\inf)\tensor_{A_\inf} W(C^\flat))^{\phi=1}
\end{equation*}
\EEDIT{from (\ref{eq:syntomification-laurentfcrystaltolocsys}),} we see that \cite[Lem.\ 4.26]{BMS} implies that there is a natural isomorphism
\begin{equation}
\label{eq:beilfibsq-tetcrystal}
T_\et(E)\tensor_{\Z_p} A_\inf[\tfrac{1}{\mu}]\cong \cal{E}(A_\inf)[\tfrac{1}{\mu}]\;,
\end{equation}
where $\mu\coloneqq [\epsilon^{1/p}]-1$; note that our definition of $\mu$ differs from the one in loc.\ cit.\ by a Frobenius twist since not $\cal{E}(A_\inf)$, but $\phi^*\cal{E}(A_\inf)$ is a Breuil--Kisin--Fargues module in the sense of \cite[Def. 4.22]{BMS}. \EEDIT{Using the maps $\O_C\rightarrow \O_C/p\leftarrow\F_p$ and the fact that $(A_\crys, (p))$ is the initial object of the absolute prismatic site of $\O_C/p$, see \cite[Not. 5.1]{FCrystals}, we obtain natural maps of prisms 
\begin{equation*}
(A_\inf, (\xi))\rightarrow (A_\crys, (p))\leftarrow (\Z_p, (p))
\end{equation*}
and hence $\cal{E}$ being a prismatic crystal yields isomorphisms}
\begin{equation}
\label{eq:beilfibsq-fcrystalbasechange}
\cal{E}(A_\inf)\tensor_{A_\inf} A_\crys\cong \cal{E}(A_\crys)\cong \cal{E}(\Z_p)\tensor_{\Z_p} A_\crys\;.
\end{equation}
\EEDIT{Finally, as} $\mu$ is invertible in $B_\crys$ (note that $t\in \xi\mu B_\crys^+$), we overall obtain
\begin{equation*}
\EEDIT{
T_\et(E)\tensor_{\Z_p} B_\crys\cong \cal{E}(A_\inf)\tensor_{A_\inf} B_\crys\cong \cal{E}(\Z_p)\tensor_{\Z_p} B_\crys\;.
}
\end{equation*}
As this isomorphism is both $\phi$- and $G_{\Q_p}$-equivariant, taking $G_{\Q_p}$-fixed points yields (\ref{eq:beilfibsq-etalephimod}).
\end{proof}

Observe that there is also a functor
\begin{equation}
\label{eq:beilfibsq-da1gmisog}
(-)[\tfrac{1}{p}]: \D(\Z_p^{\dR, +})\cong\D(\A^1_-/\G_m)\cong\widehat{\DF}(\Z_p)\rightarrow\DF(\Q_p)
\end{equation}
\EDIT{given by sending a $p$-complete filtered complex $M^\bullet$ over $\Z_p$ to the filtered complex $M^\bullet[\tfrac{1}{p}]$ over $\Q_p$ and that this yields an analogous result to \cref{lem:beilfibsq-rationalcohfpsyn}. Namely, we have the following:}

\begin{lem}
\label{lem:beilfibsq-rationalcoha1gm}
\EDIT{For any $E\in\D(\Z_p^{\dR, +})$, the functor $(-)[\tfrac{1}{p}]$ induces an isomorphism}
\begin{equation*}
\EDIT{R\Gamma(\Z_p^{\dR, +}, E)[\tfrac{1}{p}]\cong \RHom_{\MF(\Q_p)}(\Q_p, E[\tfrac{1}{p}])\;.}
\end{equation*}
\end{lem}
\begin{proof}
\EDIT{Identifying $E$ with a $p$-complete filtered complex $M^\bullet$ via the Rees equivalence, this follows by observing that}
\begin{equation*}
\EDIT{
R\Gamma(\Z_p^{\dR, +}, E)[\tfrac{1}{p}]\cong M^0[\tfrac{1}{p}]\cong\RHom_{\MF(\Q_p)}(\Q_p, E[\tfrac{1}{p}])\;.}
\qedhere
\end{equation*}
\end{proof}

Furthermore, we also get a statement analogous to \cref{lem:beilfibsq-etalephimod}:

\begin{lem}
\label{lem:beilfibsq-etalefiltered}
For any $E\in\Coh(\Z_p^\Syn)$, there is a natural isomorphism
\begin{equation*}
T_{\dR, +}(E)[\tfrac{1}{p}]\cong T_{\dR, +}(D_\crys(T_\et(E)[\tfrac{1}{p}]))\;.\footnote{\EDIT{Similarly to \cref{lem:beilfibsq-etalephimod}, beware that, on the right-hand side, we use the functor $T_{\dR, +}: \MF^\phi(\Q_p)\rightarrow\MF(\Q_p)$, while $T_{\dR, +}(E)$ on the left-hand side denotes the pullback of $E$ to $\A^1_-/\G_m$ as defined in \cref{ex:syntomification-tdr+}.}}
\end{equation*}
\end{lem}
\begin{proof}
\EDIT{As in the proof of of \cref{lem:beilfibsq-etalephimod}}, we may assume that $E$ is reflexive and let $\cal{E}\in\Vect^\phi((\Z_p)_\prism)$ denote the corresponding prismatic $F$-crystal in vector bundles. Recall that the filtration on
\begin{equation*}
D_\crys(T_\et(E)\EDIT{[\tfrac{1}{p}]})=(T_\et(E)\tensor_{\Z_p} B_\crys)^{G_{\Q_p}}=(T_\et(E)\tensor_{\Z_p} B_\dR)^{G_{\Q_p}}
\end{equation*}
is given by 
\begin{equation*}
\Fil^i D_\crys(T_\et(E))=(T_\et(E)\tensor_{\Z_p} \xi^i B_\dR^+)^{G_{\Q_p}}\;.
\end{equation*}
Furthermore, note that the filtered object $T_{\dR, +}(E)$ can be defined by descent from $T_{\dR, +}(E_{\O_C})$, where $E_{\O_C}$ denotes the pullback of $E$ to $\O_C^\Syn$. Now observe that, identifying $E_{\O_C}$ with a graded $A_\inf[u, t]/(ut-\phi^{-1}(\xi))$-module $M$ using \cref{ex:syntomification-fgaugesperfd}, the calculation from \cref{ex:fildrstack-perfd} implies that $T_{\dR, +}(E_{\O_C})$ corresponds to the graded $A_\inf[\frac{u^n}{n!}, t: n\geq 1]/(ut-\xi)$-module $\phi^*M\tensor_{A_\inf} A_\crys$. Hence, by \cref{prop:syntomification-refltofcrystal} and \cref{rem:syntomification-descriptionrefltofcrystal}, we conclude that
\begin{equation*}
T_{\dR, +}(E)[\tfrac{1}{p}]=(\phi^*\cal{E}(A_\inf)\tensor_{A_\inf} B_\crys^+)^{G_{\Q_p}}\;,
\end{equation*}
where \EDIT{the filtration on $\phi^*\cal{E}(A_\inf)\tensor_{A_\inf} B_\crys^+$ is the tensor product filtration and, in turn, the filtration on $\phi^*\cal{E}(A_\inf)$ is given by}
\begin{equation*}
\Fil^i\phi^*\cal{E}(A_\inf)=\xi^i\cal{E}(A_\inf)\cap \phi^*\cal{E}(A_\inf)\;;
\end{equation*}
\EDIT{here,} the intersection takes place inside $\phi^*\cal{E}(A_\inf)[\frac{1}{\xi}]\cong \cal{E}(A_\inf)[\frac{1}{\xi}]$ via the Frobenius structure on $\cal{E}$. Using (\ref{eq:beilfibsq-fcrystalbasechange}), however, we see that
\begin{equation*}
\begin{split}
(\phi^*\cal{E}(A_\inf)\tensor_{A_\inf} B_\crys^+)^{G_{\Q_p}}&\cong (\phi^*\cal{E}(\Z_p)\tensor_{\Z_p} B_\crys^+)^{G_{\Q_p}}\cong \phi^*\cal{E}(\Z_p)[\tfrac{1}{p}] \\
&\cong (\phi^*\cal{E}(\Z_p)\tensor_{\Z_p} B_\dR)^{G_{\Q_p}}\cong (\phi^*\cal{E}(A_\inf)\tensor_{A_\inf} B_\dR)^{G_{\Q_p}}
\end{split}
\end{equation*}
and hence we also have
\begin{equation*}
T_{\dR, +}(E)[\tfrac{1}{p}]=(\phi^*\cal{E}(A_\inf)\tensor_{A_\inf} B_\dR)^{G_{\Q_p}}\;.
\end{equation*}
We conclude that it suffices to show that the isomorphism
\begin{equation}
\label{eq:beilfibsq-tetcrystaldr}
T_\et(E)\tensor_{\Z_p} B_\dR\cong \cal{E}(A_\inf)\tensor_{A_\inf} B_\dR\cong\phi^*\cal{E}(A_\inf)\tensor_{A_\inf} B_\dR
\end{equation}
obtained from (\ref{eq:beilfibsq-tetcrystal}) using $\phi^*\cal{E}(A_\inf)[\frac{1}{\xi}]\cong \cal{E}(A_\inf)[\frac{1}{\xi}]$ is compatible with the filtrations, \EDIT{where again the right-hand side is equipped with the tensor product filtration.} As both sides are $\xi$-adically filtered, \EDIT{i.e.\ their filtrations have the property that $\Fil^i=\xi^i\Fil^0$ for all $i\in\Z$}, it suffices to check this in filtration degree $0$. \EDIT{In other words,} we have to show that the isomorphism above restricts to an isomorphism
\begin{equation*}
T_\et(E)\tensor_{\Z_p} B_\dR^+\cong \bigcup_{i\in\Z} (\xi^i\cal{E}(A_\inf)\cap \phi^*\cal{E}(A_\inf))\tensor_{A_\inf} \xi^{-i} B_\dR^+=\cal{E}(A_\inf)\tensor_{A_\inf} B_\dR^+\;,
\end{equation*}
where the right-hand side is a submodule of $\phi^*\cal{E}(A_\inf)\tensor_{A_\inf} B_\dR$ via the isomorphism $\phi^*\cal{E}(A_\inf)[\frac{1}{\xi}]\cong \cal{E}(A_\inf)[\frac{1}{\xi}]$. However, this again follows from (\ref{eq:beilfibsq-tetcrystal}): As \EDIT{$\mu=[\epsilon^{1/p}]-1$} maps to the nonzero element $\zeta_p-1$ in the residue field of $B_\dR^+$, it is already invertible in $B_\dR^+$ and hence the isomorphism (\ref{eq:beilfibsq-tetcrystaldr}) restricts to an isomorphism
\begin{equation*}
T_\et(E)\tensor_{\Z_p} B_\dR^+\cong \cal{E}(A_\inf)\tensor_{A_\inf} B_\dR^+\;,
\end{equation*}
as desired.
\end{proof}

\subsection{Proof of the main theorem}
\label{subsect:beilfibsq-main}

We are now ready to put everything together and prove \cref{thm:beilfibsq-main}. We begin with a general lemma about $t$-structures on stable $\infty$-categories, which is probably well-known, see also \cite[Lem.\ 3.2]{Bridgeland}. \EDIT{For this, first recall the following definition:}

\begin{defi}
\label{defi:beilfibsq-tstructbounded}
\EDIT{
Let $\cal{C}$ be a stable $\infty$-category equipped with a $t$-structure $(\cal{C}^{\geq 0}, \cal{C}^{\leq 0})$ and put
\begin{equation*}
\EEDIT{
\cal{C}^-\coloneqq \bigcup_{n\in\Z} \cal{C}^{\leq n}\;, \hspace{0.5cm} \cal{C}^+\coloneqq \bigcup_{n\in\Z} \cal{C}^{\geq -n}\;, \hspace{0.5cm} \cal{C}^b\coloneqq \cal{C}^-\cap\cal{C}^+\;.
}
\end{equation*}
Then the $t$-structure $(\cal{C}^{\geq 0}, \cal{C}^{\leq 0})$ is called \emph{bounded} if $\cal{C}^b=\cal{C}$. It is called \emph{non-degenerate} if}
\begin{equation*}
\EDIT{
\bigcap_{n\in\Z} \cal{C}^{\leq n}=\bigcap_{n\in\Z} \cal{C}^{\geq n}=0\;.
}
\end{equation*}
\end{defi}

\begin{rem}
\label{rem:beilfibsq-nondegeneratetstruct}
\EDIT{
Recall that a $t$-structure $(\cal{C}^{\geq 0}, \cal{C}^{\leq 0})$ on a stable $\infty$-category $\cal{C}$ induces canoncial truncation functors
\begin{equation*}
\tau^{\leq n}: \cal{C}\rightarrow\cal{C}^{\leq n}\;,\hspace{0.5cm} \tau^{\geq n}: \cal{C}\rightarrow\cal{C}^{\geq n}
\end{equation*}
for any $n\in\Z$ which are right-adjoint, respectively left-adjoint, to the inclusions $\cal{C}^{\leq n}\rightarrow\cal{C}$, respectively $\cal{C}^{\geq n}\rightarrow\cal{C}$. Moreover, we obtain cohomology functors
\begin{equation*}
H^n(-)=(\tau^{\geq n}\circ\tau^{\leq n})(-)[n]: \cal{C}\rightarrow \cal{C}^{\geq 0}\cap\cal{C}^{\leq 0}=\cal{C}^\heartsuit\;.
\end{equation*}
If the $t$-structure is non-degenerate, this implies that an object $C\in\cal{C}$ with $H^n(C)=0$ for all $n\in\Z$ must vanish. Moreover, it also implies that $C\in\cal{C}^{\geq 0}$ if and only if $H^n(C)=0$ for all $n<0$ and, similarly, $C\in\cal{C}^{\leq 0}$ if and only if $H^n(C)=0$ for all $n>0$, see \cite[Rem.\ 1.5]{tStructures}.
}
\end{rem}

\begin{lem}
\label{lem:beilfibsq-boundedtstructure}
Let $\cal{C}$ be a stable $\infty$-category equipped with a bounded non-degenerate $t$-structure. Then $\cal{C}$ is generated under shifts and \EDIT{fibres} by its heart $\cal{C}^\heartsuit$.
\end{lem}
\begin{proof}
\EDIT{By boundedness of the $t$-structure,} any $E\in\cal{C}$ has only finitely many non-vanishing cohomology groups and we use induction on the number of such nonzero $H^i(E)$. If all $H^i(E)$ vanish, then $E$ is itself zero as the $t$-structure is non-degenerate. Otherwise, let $k$ be maximal such that $H^k(E)\neq 0$ and put
\begin{equation*}
E'\coloneqq\fib(E\rightarrow\tau^{\leq k-1}E)\;.
\end{equation*}
As $H^i(\tau^{\leq k-1} E)=H^i(E)$ except for $i=k$, in which case $H^k(\tau^{\leq k-1}E)=0$, the long exact cohomology sequence shows that $H^i(E')=0$ for all $i\neq k$ and hence $E'[k]$ lies in $\cal{C}^\heartsuit$ by non-degeneracy of the $t$-structure. Thus, we are done by induction.
\end{proof}

\begin{proof}[Proof of \cref{thm:beilfibsq-main}]
The relevant square was already constructed in \cref{prop:beilfibsq-square}, so it remains to show that it is an almost pushout up to $p$-isogeny. \EDIT{First note that the $t$-structure on $\Perf(\Z_p^\Syn)$ from \cref{rem:filprism-cohheartperf} satisfies the hypotheses of \cref{lem:beilfibsq-boundedtstructure}: Indeed, this follows from the fact that this $t$-structure is induced from the one on $\Perf(\Z_p\langle u, t\rangle)$, which is clearly bounded and non-degenerate. As all terms of the square we claim to be a pullback commute with \EEDIT{shifts and fibres}, we may thus assume that $E$ is coherent by \cref{lem:beilfibsq-boundedtstructure}.} In this case, \EEDIT{recalling from \cref{prop:syntomification-cohcrystalline} that $T_\et(E)$ is crystalline and} writing $D=D_\crys(T_\et(E)[\frac{1}{p}])$, we know by \cref{prop:beilfibsq-crys} that there is a pullback square
\begin{equation*}
\begin{tikzcd}
\RHom_{\MF^\phi_\adm(\Q_p)}(\Q_p, D)\ar[r]\ar[d] & \RHom_{\Mod^\phi(\Q_p)}(\Q_p, T_\crys(D))\ar[d] \\
\RHom_{\MF(\Q_p)}(\Q_p, T_{\dR, +}(D))\ar[r] & \RHom_{\Q_p}(\Q_p, T_\dR(D))\nospacepunct{\;.}
\end{tikzcd}
\end{equation*}
\EEDIT{As} \cref{lem:beilfibsq-rationalcohfpsyn}, \cref{lem:beilfibsq-etalephimod}, \EEDIT{\cref{lem:beilfibsq-rationalcoha1gm}} and \cref{lem:beilfibsq-etalefiltered} allow us to rewrite this as
\begin{equation*}
\begin{tikzcd}
\RHom_{\MF^\phi_\adm(\Q_p)}(\Q_p, D)\ar[r]\ar[d] & R\Gamma(\F_p^\Syn, T_\crys(E))[\tfrac{1}{p}]\ar[d] \\
R\Gamma(\Z_p^{\dR, +}, T_{\dR, +}(E))[\tfrac{1}{p}]\ar[r] & R\Gamma(\Z_p^\dR, T_\dR(E))[\tfrac{1}{p}]\;,
\end{tikzcd}
\end{equation*}
it remains to show that $T_\et$ induces a quasi-isomorphism
\begin{equation*}
R\Gamma(\Z_p^\Syn, E)[\tfrac{1}{p}]\cong\RHom_{\MF^\phi_\adm(\Q_p)}(\Q_p, D)\;.
\end{equation*}
For this, first note that the embedding $\MF^\phi_\adm(\Q_p)\cong\Rep_{\Q_p}^\crys(G_{\Q_p})\hookrightarrow\Rep_{\Q_p}(G_{\Q_p})$ induces a morphism
\begin{equation*}
\RHom_{\MF^\phi_\adm(\Q_p)}(\Q_p, D)\rightarrow R\Gamma(G_{\Q_p}, T_\et(E)[\tfrac{1}{p}])
\end{equation*}
and let $Q$ denote its cofibre. We have to show that the composite
\begin{equation}
\label{eq:beilfibsq-compositezero}
R\Gamma(\Z_p^\Syn, E)[\tfrac{1}{p}]\rightarrow R\Gamma(G_{\Q_p}, T_\et(E)[\tfrac{1}{p}])\rightarrow Q\;,
\end{equation}
where the first map is induced by $T_\et$, is zero. From the long exact sequence
\begin{equation*}
\begin{tikzcd}
0\ar[r] & \Hom_{\MF^\phi_\adm(\Q_p)}(\Q_p, D)\ar[r, "\cong"] & \Hom_{G_{\Q_p}}(\Q_p, T_\et(E)[\tfrac{1}{p}])\ar[r] & H^0(Q) \ar[overlay, dll, out=-10, in=170] \\
& \Ext^1_{\MF^\phi_\adm(\Q_p)}(\Q_p, D)  \ar[r, hookrightarrow] & \Ext^1_{G_{\Q_p}}(\Q_p, T_\et(E)[\tfrac{1}{p}])\ar[r] & H^1(Q) \ar[r] & 0\nospacepunct{\;,}
\end{tikzcd}
\end{equation*}
where the last zero is due to the fact that $\RHom_{\MF^\phi_\adm(\Q_p)}(\Q_p, D)$ is concentrated in degrees $0$ and $1$ \EDIT{by \cref{rem:beilfibsq-rhommfphionetrunc},} we see that $H^0(Q)=0$ and that $H^1(Q)$ is the cokernel of the injection $\Ext^1_{\MF^\phi_\adm(\Q_p)}(\Q_p, D)\hookrightarrow \Ext^1_{G_{\Q_p}}(\Q_p, T_\et(E)[\tfrac{1}{p}])$. However, by \cref{prop:syntomification-cohcrystalline}, the complex $R\Gamma(\Z_p^\Syn, E)[\tfrac{1}{p}]$ is concentrated in degrees $0$ and $1$ and the map 
\begin{equation*}
H^1(\Z_p^\Syn, E)[\tfrac{1}{p}]\rightarrow\Ext^1_{G_{\Q_p}}(\Q_p, T_\et(E))
\end{equation*}
factors through $\Ext^1_{\MF^\phi_\adm(\Q_p)}(\Q_p, D)$, hence the map in (\ref{eq:beilfibsq-compositezero}) induces the zero map on all cohomology groups and thus is itself zero. This means that we obtain a unique factorisation
\begin{equation*}
R\Gamma(\Z_p^\Syn, E)[\tfrac{1}{p}]\rightarrow \RHom_{\MF^\phi_\adm(\Q_p)}(\Q_p, D)\rightarrow R\Gamma(G_{\Q_p}, T_\et(E)[\tfrac{1}{p}])
\end{equation*}
and, again by \cref{prop:syntomification-cohcrystalline}, the first map is a quasi-isomorphism, as desired.
\end{proof}

\subsection{Comparison with Fontaine--Messing syntomic cohomology}

As a corollary of \cref{thm:beilfibsq-main}, we can prove a comparison between syntomic cohomology in our sense and the syntomic cohomology introduced by Fontaine--Messing in \cite{FontaineMessing}.

\begin{defi}
Let $X$ be a \EEDIT{smooth qcqs} $p$-adic formal scheme. For $0\leq i\leq p-2$, we define the integral \emph{Fontaine--Messing syntomic cohomology} of $X$ in weight $i$ as
\begin{equation*}
R\Gamma_{\Syn, \FM}(X, \Z_p(i))\coloneqq \fib(\Fil^i_\Hod R\Gamma_\dR(X)\xrightarrow{1-\phi/p^i} R\Gamma_\dR(X))
\end{equation*}
and there is also a rational version defined for all $i\geq 0$ by 
\begin{equation*}
R\Gamma_{\Syn, \FM}(X, \Q_p(i))\coloneqq \fib(\Fil^i_\Hod R\Gamma_\dR(X)[\tfrac{1}{p}]\xrightarrow{\phi-p^i} R\Gamma_\dR(X)[\tfrac{1}{p}])\;.
\end{equation*}
Here, $\phi$ denotes the crystalline Frobenius on $R\Gamma_\dR(X)$.
\end{defi}

Note that the above definition indeed makes sense: By quasisyntomic descent and \cite[Prop.\ 6.8]{BeilFibSq}, the Frobenius on $\Fil^i_\Hod R\Gamma_\dR(X)$ is divisible by $p^i$ for $0\leq i\leq p-1$. In loc.\ cit., Antieau--Mathew--Morrow--Nikolaus prove the following comparison between syntomic cohomology and Fontaine--Messing syntomic cohomology:

\begin{thm}
\label{thm:beilfibsq-fmratmotivation}
Let $X$ be a \EEDIT{smooth qcqs} $p$-adic formal scheme. For $i\geq 0$, there is a natural isomorphism
\begin{equation*}
R\Gamma_{\Syn, \FM}(X, \Q_p(i))\cong R\Gamma_\Syn(X, \Q_p(i))\;.
\end{equation*}
\end{thm}
\begin{proof}
The affine version is \cite[Thm.\ 6.22]{BeilFibSq} \EDIT{and we deduce the general version using Zariski descent and \cref{lem:syntomicetale-colimtot}.}
\end{proof}

\EDIT{To generalise this result to arbitrary coefficients, at least in the case where $X$ is proper, we first observe that de Rham cohomology with $F$-gauge coefficients is also equipped with a crystalline Frobenius, just as in the case with trivial coefficients:}

\begin{rem}
\label{rem:beilfibsq-crysfrob}
\EDIT{
Let $X$ be a bounded $p$-adic formal scheme and $E\in\D(X^\Syn)$ an $F$-gauge on $X$. Then the rational de Rham cohomology $R\Gamma_\dR(X, T_\dR(E))[\tfrac{1}{p}]$ is equipped with an automorphism $\tau$ linear over the Frobenius of $\Z_p$ in the following way: By the stacky version of the crystalline comparison, i.e.\ by $\phi^*\F_p^\prism\cong \Z_p^\dR$, see \cite[Rem.\ 2.5.12, Constr.\ 3.1.1]{FGauges}, we have
\begin{equation*}
\EEDIT{
R\Gamma_\dR(X, T_\dR(E))[\tfrac{1}{p}]\cong \phi^*R\Gamma_\prism(X_{p=0}, j_\prism^*T_\crys(E))[\tfrac{1}{p}]=\phi^*j_\prism^*\pi_{(X_{p=0})^\Syn, *}T_\crys(E)[\tfrac{1}{p}]\;.
}
\end{equation*}
However, by (\ref{eq:beilfibsq-crysfrob}), the complex \EEDIT{$j_\prism^*\pi_{(X_{p=0})^\Syn, *}T_\crys(E)[\tfrac{1}{p}]$} is naturally equipped with a Frobenius automorphism since it is pulled back from $\F_p^\Syn$ along $j_\prism$. We call the corresponding automorphism $\tau$ of $R\Gamma_\dR(X, T_\dR(E))[\tfrac{1}{p}]$ the \emph{crystalline Frobenius}.
}
\end{rem}

\begin{thm}
\label{thm:beilfibsq-fmrat}
Let $E\in\Perf(\Z_p^\Syn)$. Then the Hodge-filtered de Rham map induces a fibre sequence
\begin{equation*}
R\Gamma(\Z_p^\Syn, E)[\tfrac{1}{p}]\rightarrow R\Gamma(\Z_p^{\dR, +}, T_{\dR, +}(E))[\tfrac{1}{p}]\xrightarrow{1-\tau} R\Gamma(\Z_p^\dR, T_\dR(E))[\tfrac{1}{p}]\;,
\end{equation*}
where $\tau$ denotes \EDIT{the crystalline Frobenius from \cref{rem:beilfibsq-crysfrob}.}
\end{thm}
\begin{proof}
Recall from \cref{lem:beilfibsq-rationalcohfpsyn} that there is a fibre sequence
\begin{equation*}
R\Gamma(\F_p^\Syn, T_\crys(E))[\tfrac{1}{p}]\rightarrow R\Gamma(\F_p^\prism, j_\prism^*T_\crys(E))[\tfrac{1}{p}]\xrightarrow{1-\tau} R\Gamma(\F_p^\prism, j_\prism^*T_\crys(E))[\tfrac{1}{p}]\;.
\end{equation*}
\EDIT{Again using the stacky version $\phi^*\F_p^\prism\cong \Z_p^\dR$ of the crystalline comparison as in \cref{rem:beilfibsq-crysfrob}}, the fibre sequence above becomes
\begin{equation*}
R\Gamma(\F_p^\Syn, T_\crys(E))[\tfrac{1}{p}]\rightarrow R\Gamma(\Z_p^\dR, T_\dR(E))[\tfrac{1}{p}]\xrightarrow{1-\tau} R\Gamma(\Z_p^\dR, T_\dR(E))[\tfrac{1}{p}]
\end{equation*}
and hence \cref{thm:beilfibsq-main} implies that it suffices to show that
\begin{equation*}
\begin{tikzcd}
\fib(R\Gamma(\Z_p^{\dR, +}, T_{\dR, +}(E))[\frac{1}{p}]\xrightarrow{1-\tau}R\Gamma(\Z_p^\dR, T_\dR(E))[\frac{1}{p}])\ar[r]\ar[d] & R\Gamma(\Z_p^\dR, T_\dR(E))[\frac{1}{p}]^{\tau=1}\ar[d] \\
R\Gamma(\Z_p^{\dR, +}, T_{\dR, +}(E))[\frac{1}{p}]\ar[r] & R\Gamma(\Z_p^\dR, T_\dR(E))[\frac{1}{p}]
\end{tikzcd}
\end{equation*}
is a pullback diagram. For this, let $M$ denote the term in the upper left corner and write
\begin{equation*}
N\coloneqq R\Gamma(\Z_p^\dR, T_\dR(E))[\tfrac{1}{p}]/R\Gamma(\Z_p^{\dR, +}, T_{\dR, +}(E))[\tfrac{1}{p}]\;.
\end{equation*}
\EDIT{Since we already know that} all rows and columns except the leftmost column in the commutative diagram
\begin{equation*}
\begin{tikzcd}
M\ar[r]\ar[d] & R\Gamma(\Z_p^{\dR, +}, T_{\dR, +}(E))[\frac{1}{p}]\ar[r, "1-\tau"]\ar[d] & R\Gamma(\Z_p^\dR, T_\dR(E))[\frac{1}{p}]\ar[d] \\
R\Gamma(\Z_p^\dR, T_\dR(E))[\frac{1}{p}]^{\tau=1}\ar[r]\ar[d] & R\Gamma(\Z_p^\dR, T_\dR(E))[\frac{1}{p}]\ar[r, "1-\tau"]\ar[d] & R\Gamma(\Z_p^\dR, T_\dR(E))[\frac{1}{p}]\ar[d] \\
N\ar[r] & N\ar[r] & 0
\end{tikzcd}
\end{equation*}
are fibre sequences, we conclude that also the left column is a fibre sequence and this finishes the proof.
\end{proof}

\begin{cor}
\EDIT{
Let $X$ be a $p$-adic formal scheme which is smooth and proper over $\Spf\Z_p$. For any perfect $F$-gauge $E\in\Perf(X^\Syn)$, there is a natural fibre sequence
\begin{equation*}
R\Gamma_\Syn(X, E)[\tfrac{1}{p}]\rightarrow \Fil^0_\Hod R\Gamma_\dR(X, T_{\dR, +}(E))[\tfrac{1}{p}]\xrightarrow{1-\tau} R\Gamma_\dR(X, T_\dR(E))[\tfrac{1}{p}]\;,
\end{equation*}
where $\tau$ again denotes the crystalline Frobenius from \cref{rem:beilfibsq-crysfrob}.
}
\end{cor}
\begin{proof}
\EDIT{Using \cref{prop:finiteness-main}, this follows immediately from \cref{thm:beilfibsq-fmrat}.}
\end{proof}

\EDIT{In particular, by the comparison from \cref{thm:fildrstack-comparison}, the above recovers the result of \cref{thm:beilfibsq-fmratmotivation} in the proper case by putting $E=\O\{i\}$.}

\begin{rem}
In fact, the result of \cref{thm:beilfibsq-fmratmotivation} holds already integrally for $i\leq p-2$. This integral version can also be obtained via $F$-gauges as shown in \cite[Rem.\ 5.21]{FontaineLaffaille} and \cite[Rem.\ 6.5.15]{FGauges}, hence \cref{thm:beilfibsq-fmrat} admits an integral analogue for perfect $F$-gauges $E$ with Hodge--Tate weights all at least $-(p-2)$.
\end{rem}

\subsection{The isogeny category of perfect $F$-gauges on $\Z_p$}

Taking the commutative square from \cref{prop:beilfibsq-square} and passing to categories of perfect complexes, we obtain a commutative diagram
\begin{equation*}
\begin{tikzcd}
\Perf(\Z_p^\Syn)\ar[r, "T_\crys"]\ar[d, "T_{\dR, +}", swap] & \Perf(\F_p^\Syn)\ar[d] \\
\Perf(\Z_p^{\dR, +})\ar[r] & \Perf(\Z_p^\dR)
\end{tikzcd}
\end{equation*}
of stable $\infty$-categories and hence a functor
\begin{equation*}
\Perf(\Z_p^\Syn)\rightarrow \Perf(\Z_p^{\dR, +})\times_{\Perf(\Z_p^\dR)} \Perf(\F_p^\Syn)\;.
\end{equation*}
Passing to isogeny categories and composing with the functors \EDIT{from (\ref{eq:beilfibsq-dfpsynisog}) and (\ref{eq:beilfibsq-da1gmisog}),} we obtain a functor
\begin{equation*}
\Perf(\Z_p^\Syn)[\tfrac{1}{p}]\rightarrow \DF(\Q_p)\times_{\D(\Q_p)} \D(\Mod^\phi(\Q_p))\;.
\end{equation*}
The category on the right-hand side (or, more precisely, a small variant thereof) has been called the \emph{category of $p$-adic Hodge complexes} in \cite{Bannai}. In this section, we will show that the Beilinson fibre square admits the following categorical extension:

\begin{thm}
\label{thm:beilfibsq-categorical}
The functor
\begin{equation*}
\Perf(\Z_p^\Syn)[\tfrac{1}{p}]\rightarrow \DF(\Q_p)\times_{\D(\Q_p)} \D(\Mod^\phi(\Q_p))
\end{equation*}
described above is a fully faithful embedding.
\end{thm}

On our way of proving the above result, \EDIT{we will also obtain the following statement, which we learnt from Bhargav Bhatt in private communication, see also \cite[Rem.\ 6.7.5]{FGauges}:}

\begin{prop}
\label{prop:beilfibsq-perfzpsyncrys}
The functor $T_\et$ induces an equivalence of categories
\begin{equation*}
\Perf(\Z_p^\Syn)[\tfrac{1}{p}]\cong \D^b(\Rep^\crys_{\Q_p}(G_{\Q_p}))
\end{equation*}
between the isogeny category of perfect complexes on $\Z_p^\Syn$ and the bounded derived category of crystalline $\Q_p$-representations of $G_{\Q_p}$.
\end{prop}
\begin{proof}
To prove full faithfulness, it suffices to show that, for any $E_1, E_2\in\Perf(\Z_p^\Syn)$, we have
\begin{equation*}
\RHom(E_1, E_2)[\tfrac{1}{p}]\cong \RHom_{\MF_\adm^\phi(\Q_p)}(D_1, D_2)\;,
\end{equation*}
where $D_i\coloneqq D_\crys(T_\et(E_i)[\tfrac{1}{p}])$ for $i=1, 2$. \EEDIT{For $E_1=\O$, we already know this: indeed, as all functors in sight commute with shifts and fibres, we may assume that $E_2$ is a coherent sheaf by an application of \cref{lem:beilfibsq-boundedtstructure} and then the claim follows from the proof of \cref{thm:beilfibsq-main}. However, the special case $E_1=\O$ now implies the general case: since} also $\underline{\RHom}(E_1, E_2)$ is a perfect complex, the tensor-hom adjunction yields
\begin{equation*}
\begin{split}
\RHom(E_1, E_2)[\tfrac{1}{p}]&\cong \RHom(\O, \sRHom(E_1, E_2))[\tfrac{1}{p}] \\
&\cong \RHom_{\MF^\phi_\adm(\Q_p)}(\Q_p, D_\crys(T_\et(\sRHom(E_1, E_2))[\tfrac{1}{p}])) \\
&\cong \RHom_{\MF^\phi_\adm(\Q_p)}(\Q_p, \sRHom_{\MF^\phi_\adm(\Q_p)}(D_1, D_2)) \\
&\cong \RHom_{\MF^\phi_\adm(\Q_p)}(D_1, D_2)\;,
\end{split}
\end{equation*}
as desired. Note that we have used that
\begin{equation*}
D_\crys(T_\et(\sRHom(E_1, E_2))[\tfrac{1}{p}])\cong \sRHom_{\MF^\phi_\adm(\Q_p)}(D_1, D_2)\;;
\end{equation*}
\EDIT{this} follows from the fact that $D_\crys$ is \EDIT{a symmetric monoidal equivalence} while the étale realisation is the composition of a pullback along an open immersion with a symmetric monoidal equivalence, \EDIT{both of which preserve internal hom} \EEDIT{between perfect complexes}. Finally, \EDIT{by \cref{lem:beilfibsq-boundedtstructure},} essential surjectivity is immediate from \cref{thm:syntomification-reflcrys} and the fact that the étale realisation preserves fibres \EDIT{(since it commutes with shifts and colimits)}.
\end{proof}

From here, it is quite straightforward to prove \cref{thm:beilfibsq-categorical}.

\begin{proof}[Proof of \cref{thm:beilfibsq-categorical}]
\EDIT{Since $\DF(\Q_p)\cong \D(\MF(\Q_p))$}, all we have is to show is that, for any $E_1, E_2\in\Perf(\Z_p^\Syn)$, the given functor induces a pullback diagram
\begin{equation*}
\begin{tikzcd}
\RHom(E_1, E_2)[\tfrac{1}{p}]\ar[r]\ar[d] & \RHom_{\Mod^\phi(\Q_p)}(T_\crys(E_1)[\tfrac{1}{p}], T_\crys(E_2)[\tfrac{1}{p}])\ar[d] \\
\RHom_{\MF(\Q_p)}(T_{\dR, +}(E_1)[\tfrac{1}{p}], T_{\dR, +}(E_2)[\tfrac{1}{p}])\ar[r] & \RHom_{\Q_p}(T_\dR(E_1)[\tfrac{1}{p}], T_\dR(E_2)[\tfrac{1}{p}])\nospacepunct{\;.}
\end{tikzcd}
\end{equation*}
However, \EEDIT{using \cref{prop:beilfibsq-perfzpsyncrys} and} writing $D_i=D_\crys(T_\et(E_i)[\tfrac{1}{p}])$ for $i=1, 2$ again, it follows from \cref{lem:beilfibsq-etalephimod} and \cref{lem:beilfibsq-etalefiltered} that it is equivalent to prove that the diagram
\begin{equation*}
\begin{tikzcd}
\RHom_{\MF^\phi_\adm(\Q_p)}(D_1, D_2)\ar[r]\ar[d] & \RHom_{\Mod^\phi(\Q_p)}(T_\crys(D_1), T_\crys(D_2))\ar[d] \\
\RHom_{\MF(\Q_p)}(T_{\dR, +}(D_1), T_{\dR, +}(D_2))\ar[r] & \RHom_{\Q_p}(T_\dR(D_1), T_\dR(D_2))
\end{tikzcd}
\end{equation*}
is cartesian. However, this follows from \cref{prop:beilfibsq-crys} using the tensor-hom adjunction and the fact that $T_\crys, T_{\dR, +}$ and $T_\dR$ commute with internal hom.
\end{proof}

\section{Comparing syntomic cohomology and étale cohomology}
\label{sect:syntomicetale}

In this section, we prove a comparison theorem between syntomic cohomology and étale cohomology with coefficients in sufficiently negative Hodge--Tate weights. This generalises previous results by Abhinandan, Colmez--Nizio{\l} and Tsuji, see \EDIT{\cite[Cor.\ 1.11]{Abhinandan}}, \cite[Cor.\ 5.21]{ColmezNiziol} and \cite[Thm.\ 3.3.4]{Tsuji}. Most notably, while the aforementioned results all concern the rational case, i.e.\ one needs to invert $p$, this is not necessary in our setting and our results hold integrally.

\begin{thm}
\label{thm:syntomicetale-maincoarse}
Let $X$ be a smooth qcqs $p$-adic formal scheme of relative dimension $d$ over $\Spf\Z_p$. For any $F$-gauge $E\in\Perf(X^\Syn)$ with Hodge--Tate weights all at most $-d-2$, there is a natural isomorphism
\begin{equation*}
R\Gamma_\Syn(X, E)\cong R\Gamma_\proet(X_\eta, T_\et(E))\;.
\end{equation*}
\end{thm}

If we restrict to vector bundle $F$-gauges, we can obtain an even finer comparison. However, note that the range of our result differs slightly from those of the results of Colmez--Nizio{\l} and Tsuji mentioned above, which is due to the fact that they are working with log structures; see also the discussion in \cite[Rem.s 1.3, 1.12.(i)]{Abhinandan}.

\begin{thm}
\label{thm:syntomicetale-mainfine}
Let $X$ be a smooth qcqs $p$-adic formal scheme. For any vector bundle $F$-gauge $E\in\Vect(X^\Syn)$ with Hodge--Tate weights all at most $-i-1$ for some $i\geq 0$, the natural morphism
\begin{equation*}
R\Gamma_\Syn(X, E)\rightarrow R\Gamma_\proet(X_\eta, T_\et(E))
\end{equation*}
induces an isomorphism
\begin{equation*}
\tau^{\leq i} R\Gamma_\Syn(X, E)\cong\tau^{\leq i} R\Gamma_\proet(X_\eta, T_\et(E))
\end{equation*}
and an injection on $H^{i+1}$.
\end{thm}

\comment{
In the geometric case, i.e.\ if $X$ admits a morphism to $\Spf\O_C$, we can obtain the comparison above in an even larger range. In particular, using $p\geq 2$, this recovers the result from \cite[Thm.\ 10.1]{THHandPAdicHodgeTheory}.

\begin{thm}
\label{thm:syntomicetale-mainfineoc}
Let $X$ be a smooth qcqs $p$-adic formal scheme over $\Spf\O_C$. For any vector bundle $F$-gauge $E\in\Vect(X^\Syn)$ with Hodge--Tate weights all at most $-i+p-2$ for some $i\geq 0$, the natural morphism
\begin{equation*}
R\Gamma_\Syn(X, E)\rightarrow R\Gamma_\proet(X_\eta, T_\et(E))
\end{equation*}
induces an isomorphism
\begin{equation*}
\tau^{\leq i} R\Gamma_\Syn(X, E)\cong\tau^{\leq i} R\Gamma_\proet(X_\eta, T_\et(E))
\end{equation*}
and an injection on $H^{i+1}$.
\end{thm}

As in the arithmetic case, one can also formulate a version that works for arbitrary perfect complexes on $X^\Syn$ instead of just vector bundles:

\begin{thm}
\label{thm:syntomicetale-maincoarseoc}
Let $X$ be a smooth qcqs $p$-adic formal scheme of relative dimension $d$ over $\Spf\O_C$. For any $F$-gauge $E\in\Perf(X^\Syn)$ with Hodge--Tate weights all at most $-d+p-2$ for some $i\geq 0$, there is a natural isomorphism
\begin{equation*}
R\Gamma_\Syn(X, E)\cong R\Gamma_\proet(X_\eta, T_\et(E))\;.
\end{equation*}
\end{thm}
}

\subsection{The reduced locus of $X^\Syn$}
\label{subsect:syntomicetale-red}

To prove \cref{thm:syntomicetale-maincoarse} and \cref{thm:syntomicetale-mainfine}, we will define and describe a certain locus inside $X^\N$ and $X^\Syn$, similarly to what we have done for the reduced locus of $\Z_p^\N$ and $\Z_p^\Syn$. For the rest of this section, let $X$ be a bounded $p$-adic formal scheme.

\begin{defi}
The \emph{reduced locus} $X_\red^\Syn$ of $X^\Syn$ is defined as the pullback
\begin{equation*}
\begin{tikzcd}
X^\Syn_\red\ar[d]\ar[r] & X^\Syn\ar[d] \\
\Z_{p, \red}^\Syn\ar[r] & \Z_p^\Syn\nospacepunct{\;.}
\end{tikzcd}
\end{equation*}
Similarly, we define the reduced locus $X_\red^\N$ of $X^\N$.
\end{defi}

Alternatively, one may describe the reduced locus of $X^\Syn$ as the locus cut out by the equations $p=v_{1, X}=0$, where $v_{1, X}\in H^0(X^\Syn, \O\{p-1\}/p)$ denotes the pullback of the section $v_1\in H^0(\Z_p^\Syn, \O\{p-1\}/p)$ to $X^\Syn$. Then $X_\red^\N$ can also be obtained as the pullback of $X_\red^\Syn$ to $X^\N$.

To pass between cohomology on $X_\red^\Syn$ and $X^\Syn$, we again introduce a syntomic filtration similarly to the case $X=\Spf\Z_p$ \EEDIT{treated in \cref{defi:syntomification-fil}:}

\begin{defi}
Let \EEDIT{$E\in\D((X^\Syn)_{p=0})$}. The filtration
\begin{equation*}
\begin{tikzcd}[column sep=scriptsize]
\dots\ar[r, "v_{1, X}"] & R\Gamma(X^\Syn, E\{-(p-1)\})\ar[r, "v_{1, X}"] & R\Gamma(X^\Syn, E)\ar[r, "v_{1, X}"] & R\Gamma(X^\Syn, E\{p-1\})\ar[r, "v_{1, X}"] & \dots  
\end{tikzcd}
\end{equation*}
is called the \emph{syntomic filtration} and denoted $\Fil_\bullet^\Syn R\Gamma(X^\Syn, E)[\frac{1}{v_1}]$. We denote the underlying unfiltered object by $R\Gamma(X^\Syn, E)[\frac{1}{v_1}]$.
\end{defi}

\EDIT{Note that, similarly to (\ref{eq:syntomification-grsyn}), for any \EEDIT{$E\in\D((X^\Syn)_{p=0})$}, there is a canonical isomorphism 
\begin{equation}
\label{eq:syntomicetale-grsyn}
\gr^\Syn_\bullet R\Gamma(X^\Syn, E)[\tfrac{1}{v_1}]\cong R\Gamma(X_\red^\Syn, E/v_{1, X}\{\bullet(p-1)\})\;,
\end{equation}
where $E/v_{1, X}\coloneqq\cofib(E\{-(p-1)\}\xrightarrow{v_{1, X}} E)$. Moreover, as in the case $X=\Spf\Z_p$, \EEDIT{see (\ref{eq:syntomification-synet})}, we have the following result relating the syntomic filtration to étale cohomology:
}

\begin{prop}
\label{prop:syntomicetale-filtrationetale}
Let $X$ be a bounded \EEDIT{qcqs} $p$-adic formal scheme. For any \EEDIT{$E\in\D((X^\Syn)_{p=0})$}, \EDIT{the syntomic filtration is complete}. \EEDIT{If $X$ is moreover $p$-quasisyntomic and $E\in\Perf((X^\Syn)_{p=0})$, then} there is a natural isomorphism
\begin{equation*}
R\Gamma(X^\Syn, E)[\tfrac{1}{v_1}]\cong R\Gamma_\proet(X_\eta, T_\et(E))\;.
\end{equation*}
\end{prop}
\begin{proof}
We work over $\F_p$ throughout and thus mostly suppress the subscript $(-)_{p=0}$ from the notation. We first claim that the section $v_{1, X}$ is topologically nilpotent, i.e.\ its pullback to any scheme $\Spec S\rightarrow X^\Syn$ is nilpotent; as $v_{1, X}$ arises from $v_1$ via pullback, it suffices to prove this for $X=\Spf\Z_p$. However, recalling that we have defined $v_1$ as the pullback of the section $u^pt$ of the line bundle $\O(-1)\boxtimes\O(p)$ on $\A^1_-/\G_m\times (\A^1_+/\G_m)^\dR$, it suffices to show that $(ut)^p$ is topologically nilpotent on $\Z_p^\N$. As this is the pullback of the section $t^p: \O\rightarrow\O(-p)$ along the composite map 
\begin{equation*}
\Z_p^\N\xrightarrow{(t, u)}\A^1_-/\G_m\times (\A^1_+/\G_m)^\dR\xrightarrow{\mathrm{mult}} (\A^1_-/\G_m)^\dR\;,
\end{equation*}
it suffices to show that the pullback of $t^p$ along $\Z_p^\prism\xrightarrow{\widetilde{\mu}} (\A^1_-/\G_m)^\dR$ is topologically nilpotent and we may do this after a further pullback along the fpqc cover $F: \Z_p^\prism\rightarrow\Z_p^\prism$. Using the commutative diagram (it is in fact a pullback square)
\begin{equation*}
\begin{tikzcd}
\Z_p^\prism\ar[r, "F"]\ar[d, "\mu", swap] & \Z_p^\prism\ar[d, "\widetilde{\mu}"] \\
\A^1_-/\G_m\ar[r] & (\A^1_-/\G_m)^\dR
\end{tikzcd}
\end{equation*}
from \cref{ex:filprism-jdrjht}, it is then enough to check that the map $\mu: \Z_p^\prism\rightarrow\A^1_-/\G_m$ factors through $\widehat{\A}^1_-/\G_m\subseteq \A^1_-/\G_m$, but this is true by \cref{ex:prismatisation-cwdivtodiv}.

To prove the second statement, first observe that there is a natural map
\begin{equation*}
R\Gamma(X^\Syn, E)[\tfrac{1}{v_1}]\rightarrow R\Gamma_\proet(X_\eta, T_\et(E))
\end{equation*}
as the étale realisation carries $v_{1, X}$ to an isomorphism by \cite[Ex.\ 6.3.3]{FGauges}. To show that this map is an isomorphism, \cref{lem:syntomicetale-colimtot} allows us to reduce to the case where $X=\Spf R$ is quasiregular semiperfectoid using quasisyntomic descent and \cref{lem:filprism-quasisyntomiccover}. However, in this case, the claim follows from the description of the étale realisation functor from \cite[Constr.\ 6.3.1]{FGauges}. 
\end{proof}

\subsection{The components of the reduced locus of $X^\Syn$}
\label{subsect:syntomicetale-compred}

We now explain how to obtain the reduced locus $X^\Syn_\red$ by gluing several components which have a fairly explicit description.  \EDIT{Before proceeding, we encourage the reader to revisit the description of the reduced locus $\Z_{p, \red}^\Syn$ of $\Z_p^\Syn$ given towards the end of Section \ref{sect:syntomification} as the case $X=\Spf\Z_p$ will be prototypical.} We let $Y\coloneqq X_{p=0}$ denote the special fibre of $X$.

\begin{defi}
The \emph{de Rham component} $X^\N_\dR$ of $X^\N$ is defined as the pullback
\begin{equation*}
\begin{tikzcd}
X^\N_\dR\ar[r]\ar[d] & X^\N\ar[d] \\
\Z_{p, \dR}^\N\ar[r] & \Z_p^\N\nospacepunct{\;.}
\end{tikzcd}
\end{equation*}
\end{defi}

Note that, by definition of $\Z_{p, \dR}^\N$, the map $i_\dR$ induces an fpqc cover $\Spec\F_p\rightarrow\Z_{p, \dR}^\N$ which classifies the $\F_p$-point $(\F_p\xrightarrow{\id}\F_p, \G_a^\dR(\F_p)\xrightarrow{0} \G_a^\dR(\F_p), W(\F_p)\xrightarrow{p} W(\F_p))$ of $\Z_p^\N$ and it is straightforward to check that the pullback of $\G_a^\N$ along this cover identifies with $\G_a^\dR$. Thus, we see that this induces a pullback square
\begin{equation}
\label{eq:syntomicetale-coverdr}
\begin{tikzcd}
(X^\dR)_{p=0}\ar[r]\ar[d] & X^\N_\dR\ar[d] \\
\Spec\F_p\ar[r] & \Z_{p, \dR}^\N\nospacepunct{\;.}
\end{tikzcd}
\end{equation}

\begin{rem}
If $X$ is smooth and qcqs, using (\ref{eq:syntomicetale-coverdr}) \EEDIT{and \cref{rem:drstack-vect}}, one can give an explicit description of vector bundles on $X^\N_\dR$ in terms of vector bundles on $Y$ with a flat connection \EEDIT{having locally nilpotent $p$-curvature} equipped with an additional Sen operator $\Theta$ satisfying certain compatibilities.
\end{rem}

Although we will not need it in the sequel, we also shortly want to discuss the Hodge--Tate component $X^\N_\HT$ defined as the pullback
\begin{equation*}
\begin{tikzcd}
X^\N_\HT\ar[r]\ar[d] & X^\N\ar[d] \\
\Z_{p, \HT}^\N\ar[r] & \Z_p^\N\nospacepunct{\;.}
\end{tikzcd}
\end{equation*}
To this end, recall that $\Z_{p, \HT}^\N$ can be described as \EEDIT{$j_\HT(\Z_p^\prism)\cap\Z_{p, \HT, c}^\N=j_\HT(\Z_p^\prism)\cap (\Z_p^\N)_{p=t=0}$} and hence $X^\N_\HT$ identifies with $(X^\prism)_{p=\mu=0}$ via the map $j_\HT$. In other words, we have $X^\N_\HT\cong (X^\HT)_{p=0}$, where $X^\HT$ is the \emph{Hodge--Tate stack} of $X$, see \cite[Constr.\ 3.7]{PFS}, and this means that there is a pullback square
\begin{equation}
\label{eq:syntomicetale-coverht}
\begin{tikzcd}
(X^\dHod)_{p=0}\ar[r]\ar[d] & X^\N_\HT\ar[d] \\
\Spec\F_p\ar[r] & \Z_{p, \HT}^\N\nospacepunct{\;,}
\end{tikzcd}
\end{equation}
where $X^\dHod$ denotes the \emph{diffracted Hodge stack} of $X$ as introduced in \cite[Constr.\ 3.8]{PFS}, see also \cite[Def.\ 3.2.1]{NygaardHodge}, and the lower map classifies the $\F_p$-point $(\F_p\xrightarrow{0}\F_p, \G_a^\dR(\F_p)\xrightarrow{\id} \G_a^\dR(\F_p), W(\F_p)\xrightarrow{p} W(\F_p))$ of $\Z_p^\N$.

\begin{defi}
The \emph{Hodge-filtered de Rham component} $X^\N_{\dR, +}$ of $X^\N$ is defined as the pullback
\begin{equation*}
\begin{tikzcd}
X^\N_{\dR, +}\ar[r]\ar[d] & X^\N\ar[d] \\
\Z_{p, \dR, +}^\N\ar[r] & \Z_p^\N\nospacepunct{\;.}
\end{tikzcd}
\end{equation*}
\end{defi}

Note that, by definition of \EDIT{$\Z_{p, \dR, +}^\N$}, there is an fpqc cover \EDIT{$\A^1_-/\G_m\rightarrow\Z_{p, \dR, +}^\N$} induced by the Hodge-filtered de Rham map and it is straightforward to check that the pullback of $\G_a^\N$ along this cover identifies with $\G_a^{\dR, +}$. In other words, this cover induces a pullback square
\begin{equation}
\label{eq:syntomicetale-coverdr+}
\begin{tikzcd}
(X^{\dR, +})_{p=0}\ar[r]\ar[d] & X^\N_{\dR, +}\ar[d] \\
\A^1_-/\G_m\ar[r] & \Z_{p, \dR, +}^\N\nospacepunct{\;.}
\end{tikzcd}
\end{equation}

\begin{rem}
If $X$ is smooth and qcqs, using (\ref{eq:syntomicetale-coverdr+}) \EEDIT{and \cref{rem:fildrstack-vect}}, one can describe vector bundles on $X_{\dR, +}^\N$ explicitly in terms of filtered vector bundles on $Y$ equipped with a flat connection satisfying Griffiths transversality and a Sen operator.
\end{rem}

\begin{defi}
The \emph{conjugate-filtered Hodge--Tate component} $X^\N_{\HT, c}$ of $X^\N$ is defined as the pullback
\begin{equation*}
\begin{tikzcd}
X^\N_{\HT, c}\ar[r]\ar[d] & X^\N\ar[d] \\
\Z_{p, \HT, c}^\N\ar[r] & \Z_p^\N\nospacepunct{\;.}
\end{tikzcd}
\end{equation*}
\end{defi}

Recalling that $\Z_{p, \HT, c}^\N\cong \G_{a, +}^\dR/\G_m$, we see that there is an fpqc cover $\A^1_+/\G_m\rightarrow\Z_{p, \HT, c}^\N$ and, by definition of the \emph{conjugate-filtered diffracted Hodge stack} $X^{\dHod, c}$ of $X$ as introduced in \cite[Def.\ 3.3.1]{NygaardHodge}, this cover induces a pullback square
\begin{equation*}
\begin{tikzcd}
(X^{\dHod, c})_{p=0}\ar[r]\ar[d] & X^\N_{\HT, c}\ar[d] \\
\A^1_+/\G_m\ar[r] & \Z_{p, \HT, c}^\N\nospacepunct{\;.}
\end{tikzcd}
\end{equation*}

\begin{defi}
The \emph{Hodge component} $X^\N_\Hod$ of $X^\N$ is defined as the pullback
\begin{equation*}
\begin{tikzcd}
X^\N_\Hod\ar[r]\ar[d] & X^\N\ar[d] \\
\Z_{p, \Hod}^\N\ar[r] & \Z_p^\N\nospacepunct{\;.}
\end{tikzcd}
\end{equation*}
\end{defi}

Note that base changing the fpqc cover $\A^1_-/\G_m\rightarrow\Z_{p, \dR, +}^\N$ from above to $\Z_{p, \Hod}^\N\subseteq\Z_{p, \dR, +}^\N$ yields an fpqc cover $B\G_m\rightarrow\Z_{p, \Hod}^\N$. Moreover, from the pullback square (\ref{eq:syntomicetale-coverdr+}), we infer that there is a pullback square
\begin{equation}
\label{eq:syntomicetale-coverhod}
\begin{tikzcd}
(X^\Hod)_{p=0}\ar[r]\ar[d] & X^\N_\Hod\ar[d] \\
B\G_m\ar[r] & \Z_{p, \Hod}^\N\nospacepunct{\;.}
\end{tikzcd}
\end{equation}

For our analysis of the syntomic cohomology of $X$, it will be vital to have a good understanding of vector bundles on $X_\Hod^\N$ and their cohomology. For this, we will first need to recall some facts about Higgs bundles. Thus, let $Y$ be a scheme or formal scheme which is smooth over a base \EDIT{(formal) scheme} $S$. For us, the case $S=\Spec\F_p$ will be the most relevant as we will want to apply the following definitions in our context of $Y$ being the the special fibre of some smooth $p$-adic formal scheme $X$.

\begin{defi}
Let $V\in\Vect(Y)$ be a vector bundle on $Y$. A \emph{Higgs field} on \EEDIT{$V$} is an $\O_Y$-linear morphism $\Phi: V\rightarrow V\tensor\Omega_{Y/S}^1$ such that $\Phi\wedge\Phi=0$. Here, $\Phi\wedge\Phi: V\rightarrow V\tensor\Omega_{Y/S}^2$ denotes the composition
\begin{equation*}
V\xrightarrow{\mathrlap{\hspace{0.25cm}\Phi}\hphantom{\Phi\tensor\id}} V\tensor\Omega^1_{Y/S}\xrightarrow{\Phi\tensor\id} V\tensor\Omega^1_{Y/S}\tensor\Omega^1_{Y/S}\xrightarrow{\id\tensor\wedge}V\tensor\Omega^2_{Y/S}\;,
\end{equation*}
\EDIT{where $\wedge: \Omega^1_{Y/S}\tensor\Omega^1_{Y/S}\rightarrow\Omega^1_{Y/S}\wedge\Omega^1_{Y/S}\cong \Omega^2_{Y/S}$ is the canonical map.}
\end{defi}

From a Higgs field $\Phi$ on some $V\in\Vect(Y)$, we obtain induced morphisms $\Phi: V\tensor\Omega_{Y/S}^i\rightarrow V\tensor\Omega_{Y/S}^{i+1}$ given by
\begin{equation*}
V\tensor\Omega_{Y/S}^i\xrightarrow{\Phi\tensor\id}V\tensor\Omega_{Y/S}^1\tensor\Omega_{Y/S}^i\xrightarrow{\id\tensor\wedge} V\tensor\Omega_{Y/S}^{i+1}
\end{equation*}
and thus a variant of the Hodge complex 
\begin{equation*}
V\xrightarrow{\Phi} V\tensor\Omega^1_{Y/S}\xrightarrow{\Phi} V\tensor\Omega^2_{Y/S}\xrightarrow{\Phi}\dots\;.
\end{equation*}

By the tensor-hom adjunction, for any vector bundle $V$ on $Y$, an $\O_Y$-linear morphism $\Phi: V\rightarrow V\tensor\Omega_{Y/S}^1$ is equivalent to the data of an $\O_Y$-linear morphism
\begin{equation*}
\Phi_{(-)}: \T_{Y/S}\rightarrow \sEnd_{\O_Y}(V)\;,
\end{equation*}
which can explicitly be described as sending any local section $\theta\in\T_{Y/S}(U)$ to the endomorphism
\begin{equation*}
\Phi_\theta: V|_U\xrightarrow{\mathrlap{\hspace{0.3cm}\Phi}\hphantom{\id\tensor\theta^\vee}} (V\tensor\Omega^1_{Y/S})|_U\xrightarrow{\id\tensor\theta^\vee} V|_U\;;
\end{equation*}
\EDIT{here, $\T_{Y/S}$ denotes the tangent bundle of $Y$ over $S$, as usual.} One can then check that the condition $\Phi\wedge\Phi=0$ is exactly equivalent to requiring that $\Phi_\theta$ and $\Phi_{\theta'}$ commute for any local sections $\theta, \theta'\in\T_{Y/S}(U)$. Thus, we conclude that equipping $V$ with a Higgs field is equivalent to equipping $V$ with the structure of a $\Sym^\bullet_{\O_Y}\T_{Y/S}$-module, \EDIT{where we remind the reader that $\Sym^\bullet_{\O_Y}\T_{Y/S}$ denotes the symmetric algebra of the $\O_Y$-module $\T_{Y/S}$}. Finally, we say that a Higgs field $\Theta$ on $V\in\Vect(Y)$ is \emph{locally nilpotent} if for any local sections $\theta\in\T_{Y/S}(U), s\in V(U)$, we have $\Phi_\theta^n(s)=0$ for some $n>0$.

Finally, we note that the above definitions immediately generalise to the case of arbitrary quasi-coherent complexes $V\in\D(Y)$ instead of vector bundles: A Higgs bundle on $Y$ is an $\O_Y$-linear morphism $\Phi: V\tensor\Omega_{Y/S}^1\rightarrow\Omega_{Y/S}^1$ and this is equivalent to the datum of equipping $V$ with the structure of a $\Sym^\bullet_{\O_Y}\T_{Y/S}$-complex; it is locally nilpotent if for any local section $\theta\in\T_{Y/S}(U)$, the operator $\Phi_\theta$ acts locally nilpotently on the sections of the cohomology sheaves of $V$ over $U$. Again, we obtain an analogue of the Hodge complex given by the total complex
\begin{equation*}
\Tot(V\xrightarrow{\Phi} V\tensor\Omega^1_{Y/S}\xrightarrow{\Phi} V\tensor\Omega^2_{Y/S}\xrightarrow{\Phi}\dots)\;.
\end{equation*}

We now return to our particular case of $Y=X_{p=0}$ for a smooth $p$-adic formal scheme $X$. Using the pullback square (\ref{eq:syntomicetale-coverhod}), it turns out that the main part of the work lies in analysing vector bundles on $(X^\Hod)_{p=0}$ and their cohomology.

\begin{lem}
\label{lem:syntomicetale-identifyhod}
Let $X$ be a smooth qcqs $p$-adic formal scheme \EDIT{and put $Y=X_{p=0}$}. Then there is an isomorphism
\begin{equation*}
\EEDIT{
(X^\Hod)_{p=0}\cong B_{Y\times B\G_m}\V(\T_{Y/\F_p}(1))^\sharp
}
\end{equation*}
over $Y\times B\G_m$.
\end{lem}
\begin{proof}
Note that, over the stack $B\G_m$, the map $\G_a^\Hod\cong \G_a\oplus B\V(\O(1))^\sharp\rightarrow \G_a$ is a square-zero extension of the target by $B\V(\O(1))^\sharp$ and thus the induced map
\begin{equation*}
(X^\Hod)_{p=0}\rightarrow (B\G_m)_{p=0}
\end{equation*}
of stacks over $B\G_m$ is a \EEDIT{$\V(\T_{Y/\F_p}(1))^\sharp$-gerbe}, \EDIT{where $\T_{Y/\F_p}$ denotes the tangent sheaf of $Y$ over $\F_p$}, by derived deformation theory.\footnote{More precisely, we are using the following assertion in derived algebraic geometry: Let $Y$ be a finite type $\F_p$-scheme and $A'\rightarrow A$ a square-zero extension of animated $\F_p$-algebras with fibre $N\in\D^{\leq 0}(A)$. Then the fibre of the map $Y(A')\rightarrow Y(A)$ over a point $\eta: \Spec A\rightarrow X$ is a torsor for $\operatorname{Der}_{\F_p}(\O_Y, \eta_*N)\cong\Map_A(\eta^*L_{Y/\F_p}, N)$; \EDIT{the proof is similar to \cite[Thm.\ 5.1.13]{FlatPurity}.} Note that, if furthermore $N=L[1]\in\D^{\leq -1}(A)$ and $Y$ is smooth, we have $\Map_A(\eta^*L_{Y/\F_p}, N)\cong B(\eta^*\T_{Y/\F_p}\tensor_A L)$.} However, as the square-zero extension above splits via the canonical injection, the gerbe is also split and we are done.
\end{proof}

Due to the description of $(X^\Hod)_{p=0}$ provided by the lemma above, we will make use of the following general statement about $\V(E)^\sharp$-representations for a vector bundle $E$:

\begin{lem}
\label{lem:syntomicetale-vesharpreps}
Let $X$ be a scheme and $E\in\Vect(X)$ a vector bundle on $X$. Then there is a natural equivalence
\begin{equation*}
\D(B\V(E)^\sharp)\cong\D(\widehat{\V(E^\vee)})
\end{equation*}
\EEDIT{of stable $\infty$-categories} compatible with the forgetful functors on both sides; here, $\widehat{\V(E^\vee)}$ denotes the completion of $\V(E^\vee)$ at the zero section. Moreover, if $V$ is a quasi-coherent complex on $B\V(E)^\sharp$ and $\pi: B\V(E)^\sharp\rightarrow X$ is the projection, then
\begin{equation*}
\pi_*V\cong \Tot(V\rightarrow V\tensor E^\vee\rightarrow V\tensor \wedge^2 E^\vee\rightarrow\dots)\;,
\end{equation*}
where the maps are induced by the natural $\Sym^\bullet_{\O_X} E$-module structure on $V$ via the equivalence above. Note that we abuse notation and also use $V$ to denote the pullback of $V$ to $X$.
\end{lem}
\begin{proof}
The first assertion is \cite[Prop.\ 2.4.5]{FGauges}. For the second one, we reduce to the case where $X=\Spec R$ is affine. Then we have
\begin{equation*}
\pi_*V=R\Gamma(B\V(E)^\sharp, V)\cong\RHom_{B\V(E)^\sharp}(\O, V)\cong \RHom_{\Sym^\bullet_R E}(R, V)\;,
\end{equation*}
where the action of $E$ on $R$ is trivial in the rightmost term. Using the Koszul resolution
\begin{equation*}
(\dots\rightarrow\wedge^2 E\tensor_R \Sym^\bullet_R E\rightarrow E\tensor_R \Sym^\bullet_R E\rightarrow\Sym^\bullet_R E)\rightarrow R
\end{equation*}
to compute the RHom above, we obtain the result.
\end{proof}

\begin{prop}
\label{prop:syntomicetale-vecthod}
Let $X$ be a smooth qcqs $p$-adic formal scheme \EDIT{and put $Y=X_{p=0}$}. Then giving a quasi-coherent complex on $(X^\Hod)_{p=0}$ amounts to specifying a graded quasi-coherent complex $V=\bigoplus_i V_i$ on $Y$ equipped with a Higgs field $\Phi: V\rightarrow V\tensor\Omega_{Y/\F_p}^1$ which is locally nilpotent and decreases degree by $1$, i.e.\ the restriction of $\Phi$ to $V_i$ comes with a factorisation through $V_{i-1}\tensor\Omega_{Y/\F_p}^1$.
\end{prop}
\begin{proof}
This now follows by combining \cref{lem:syntomicetale-identifyhod} \EDIT{with the $\G_m$-equivariant version of} \cref{lem:syntomicetale-vesharpreps}.
\end{proof}

\begin{prop}
\label{prop:syntomicetale-cohomologyhod}
Let $X$ be a smooth qcqs $p$-adic formal scheme \EDIT{and put $Y=X_{p=0}$}. If $E\in\D((X^\Hod)_{p=0})$ corresponds to the data $(V_\bullet, \Phi)$ via \cref{prop:syntomicetale-vecthod}, then the pushforward of $E$ to $B\G_m$ identifies with the graded complex whose degree $i$ term is given by
\begin{equation*}
R\Gamma(Y, \Tot(V_i\xrightarrow{\Phi}V_{i-1}\tensor\Omega_{Y/\F_p}^1\xrightarrow{\Phi} V_{i-2}\tensor\Omega_{Y/\F_p}^2\xrightarrow{\Phi}\dots))\;.
\end{equation*}
\end{prop}
\begin{proof}
Immediate from \cref{lem:syntomicetale-identifyhod} and \EDIT{the $\G_m$-equivariant version of} \cref{lem:syntomicetale-vesharpreps}.
\end{proof}

Finally, let us explain how to obtain $X^\N_\red$ and $X^\Syn_\red$ from the components above by gluing: From the analogous statements for $X=\Spf\Z_p$, we see that gluing $X_{\dR, +}^\N$ and $X_{\HT, c}^\N$ along $X_\Hod^\N$ yields the stack $X^\N_\red$. To obtain $X^\Syn_\red$, we then have to further glue $X_\dR^\N$ and $X_\HT^\N$ along a Frobenius twist.

\subsection{Proof of the main theorems}

We now move on to proving our main statements concerning the comparison between syntomic cohomology and étale cohomology announced in the beginning. As in the case $X=\Spf\Z_p$, we use the notation $F_\dR(-), F_\Hod(-), F_{\dR, +}(-)$ and $F_{\HT, c}(-)$ to denote cohomology on the components of $X^\Syn_\red$.

\begin{proof}[Proof of \cref{thm:syntomicetale-mainfine}]
Let $d$ be the relative dimension of $X$ over $\Spf\Z_p$. We are going to show that $\cofib(R\Gamma_\Syn(X, E)\rightarrow R\Gamma_\proet(X, T_\et(E)))$ is concentrated in degrees at least $i+1$, \EEDIT{which suffices to establish the statement}. By $p$-completeness, it suffices to check this after reducing mod $p$, i.e.\ we may assume that $E$ is a vector bundle on $(X^\Syn)_{p=0}$. Subsequently, \cref{prop:syntomicetale-filtrationetale} reduces us to proving that the complexes \EDIT{$\gr^\Syn_k R\Gamma(X^\Syn, E)[\tfrac{1}{v_1}]$} are concentrated in \EDIT{cohomological} degrees at least $i+1$ for all $k\geq 1$ \EDIT{as this will imply that the cofibre of the map
\begin{equation*}
R\Gamma(X^\Syn, E)\rightarrow R\Gamma(X^\Syn, E)[\tfrac{1}{v_1}]
\end{equation*}
is concentrated in cohomological degrees at least $i+1$, from which our result follows.} \EDIT{Using the isomorphism
\begin{equation*}
\gr^\Syn_k R\Gamma(X^\Syn, E)[\tfrac{1}{v_1}]\cong R\Gamma(X^\Syn_\red, E/v_{1, X}\{k(p-1)\})
\end{equation*}
from (\ref{eq:syntomicetale-grsyn}), we now replace $E$ by $E/v_{1, X}\{k(p-1)\}$; then our assumption becomes that $E$ is a vector bundle on $X^\Syn_\red$ with Hodge--Tate weights all at most $-i-p$ and we have to} prove that $R\Gamma(X^\Syn_\red, E)$ is concentrated in \EDIT{cohomological} degrees at least $i+1$. For this, we identify the pushforward $\pi_{X^\Syn, *} E$ of $E$ to $\Z_{p, \red}^\Syn$ with the data $(V, \Fil^\bullet V, \Fil_\bullet V)$ with $\gr^\bullet V=\gr_\bullet V$ equipped with operators $D$ and $\Theta$ according to the description of the reduced locus of $\Z_p^\Syn$. Now we make the following observation: by \cref{prop:syntomicetale-cohomologyhod}, the Hodge--Tate weights of $\pi_{X^\Syn, *} E$ are all at most $-i-p+d$; moreover, for $0\leq k\leq d$, the complex $\gr^{-i-p+d-k} V$ is concentrated in degrees at least $d-k$. In particular, we deduce that
\begin{equation*}
F_\Hod(E)=\fib(\gr^0 V\xrightarrow{\Theta} \gr^{-p} V)
\end{equation*}
is concentrated in degrees at least $i+1$. As the filtration $\Fil^\bullet V$ is \EDIT{complete by \cref{lem:finiteness-pushforwardcomplete} and base change for the cartesian square (\ref{eq:syntomicetale-coverdr+})}, we may also conclude that $\Fil^0 V$ is concentrated in degrees at least $i+p$ while $\Fil^{-p} V$ is concentrated in degrees at least $i$. Thus, the complex
\begin{equation*}
F_{\dR, +}(E)=\fib(\Fil^0 V\xrightarrow{\Theta} \Fil^{-p} V)
\end{equation*}
is concentrated in degrees at least $i+1$. Finally, we may also infer that the natural map $\Fil_{-1} V\rightarrow V$ is an isomorphism in cohomological degrees at most \EEDIT{$i+p-1$ while} $\Fil_0 V\rightarrow V$ \EEDIT{is an isomorphism in cohomological degrees at most $i+p$}. Thus, the natural map
\begin{equation*}
\EEDIT{
F_{\HT, c}(E)=\fib(\Fil_0 V\xrightarrow{D} \Fil_{-1} V)\rightarrow\fib(V\xrightarrow{\Theta} V)=F_\dR(E)
}
\end{equation*}
is also an isomorphism in cohomological degrees at most $i+p$. Overall, using the formula
\begin{equation*}
R\Gamma(X^\Syn_\red, E)=\fib(F_{\dR, +}(E)\oplus F_{\HT, c}(E)\xrightarrow{a_E-b_E} F_\dR(E)\oplus F_\Hod(E))
\end{equation*} 
from (\ref{eq:syntomification-cohomologyreduced}), we see that $R\Gamma(X^\Syn_\red, E)$ is indeed concentrated in degrees at least $i+1$, as wanted.
\end{proof}

\begin{proof}[Proof of \cref{thm:syntomicetale-maincoarse}]
As in the \EDIT{proof of \cref{thm:syntomicetale-mainfine}, we may assume that $E$ is a vector bundle on $(X^\Syn)_{p=0}$ by $p$-completeness of both sides and then \cref{prop:syntomicetale-filtrationetale} reduces us to proving that
\begin{equation*}
\gr^\Syn_k R\Gamma(X^\Syn, E)[\tfrac{1}{v_1}]\cong R\Gamma(X^\Syn_\red, E/v_{1, X}\{k(p-1)\})
\end{equation*}
vanishes for all $k\geq 1$, where we have again used (\ref{eq:syntomicetale-grsyn}). Replacing $E$ by $E/v_{1, X}\{k(p-1)\}$, we thus have to show that,} for $E\in\Perf(X^\Syn_\red)$ with Hodge--Tate weights all at most $-p-d-1$, the cohomology $R\Gamma(X^\Syn_\red, E)$ vanishes. As before, we identify $\pi_{X^\Syn, *} E$ with the data $(V, \Fil^\bullet V, \Fil_\bullet V)$ with $\gr^\bullet V=\gr_\bullet V$ equipped with operators $D$ and $\Theta$. From \cref{prop:syntomicetale-cohomologyhod}, we see that $\pi_{X^\Syn, *} E$ has Hodge--Tate weights all at most $-p-1$. Thus, we conclude
\begin{equation*}
F_\Hod(E)=\fib(\gr^0 V\xrightarrow{\Theta} \gr^{-p} V)=0\;.
\end{equation*}
Moreover, as the filtration $\Fil^\bullet V$ is again \EDIT{complete by \cref{lem:finiteness-pushforwardcomplete}}, we also have $\Fil^0 V=\Fil^{-p} V=0$ and hence also
\begin{equation*}
F_{\dR, +}(E)=\fib(\Fil^0 V\xrightarrow{\Theta} \Fil^{-p} V)=0\;.
\end{equation*}
Finally, we see that $\Fil_{-1} V=\Fil_0 V=V$ and thus
\begin{equation*}
F_{\HT, c}(E)=\fib(\Fil_{-1} V\xrightarrow{D}\Fil_0 V)=\fib(V\xrightarrow{\Theta} V)=F_\dR(E)\;.
\end{equation*}
Overall, using (\ref{eq:syntomification-cohomologyreduced}), we obtain
\begin{equation*}
R\Gamma(X^\Syn_\red, E)=\fib(F_{\dR, +}(E)\oplus F_{\HT, c}(E)\xrightarrow{a_E-b_E} F_\dR(E)\oplus F_\Hod(E))=0\;,
\end{equation*} 
as wanted.
\end{proof}

\comment{
Finally, we show how to obtain the claimed improvement upon the result from \cref{thm:syntomicetale-mainfine} in the case where $X$ admits a morphism to $\Spf\O_C$:

\begin{proof}[Proof of \cref{thm:syntomicetale-mainfineoc}]
Let $d$ be the relative dimension of $X$ over $\Spf\O_C$. As in the proof of \cref{thm:syntomicetale-mainfine}, \EDIT{we may first assume that $E\in\Vect((X^\Syn)_{p=0})$ and are then reduced to showing that
\begin{equation*}
\gr^\Syn_k R\Gamma(X^\Syn, E)[\tfrac{1}{v_1}]\cong R\Gamma(X^\Syn_\red, E/v_{1, X}\{k(p-1)\})
\end{equation*}
is concentrated in cohomological degrees at least $i+1$ for all $k\geq 1$. Replacing $E$ by $E/v_{1, X}\{k(p-1)\}$, our assumption thus becomes that $E$ is a vector bundle on $X^\Syn_\red$ with Hodge--Tate weights all at most $-i-1$ and we have to show that the complex $R\Gamma(X^\Syn_\red, E)$ is concentrated in cohomological degrees at least $i+1$. Denoting the map $X^\Syn\rightarrow\O_C^\Syn$ by $\pi'_{X^\Syn}$ and} recalling that 
\begin{equation*}
\O_C^\N\cong \Spf (A_\inf(\O_C)\langle u, t\rangle/(ut-\phi^{-1}(\xi)))/\G_m
\end{equation*}
from \cref{ex:filprism-perfd} and $v_{1, \O_C}=u^pt\in A_\inf(\O_C)[u, t]/(ut-\phi^{-1}(\xi), p)$, we may identify the pushforward \EDIT{$\pi'_{X^\Syn, *}E$} of $E$ to $\O_{C, \red}^\Syn$ with a diagram
\begin{equation*}
\begin{tikzcd}
\dots \ar[r,shift left=.5ex,"t"]
  & \ar[l,shift left=.5ex, "u"] M^{i+1} \ar[r,shift left=.5ex,"t"] & \ar[l,shift left=.5ex, "u"] M^i \ar[r,shift left=.5ex,"t"] & \ar[l,shift left=.5ex, "u"] M^{i-1} \ar[r,shift left=.5ex,"t"] & \ar[l,shift left=.5ex, "u"] \dots
\end{tikzcd}
\end{equation*}
of $\xi$-complete $A_\inf(\O_C)/p$-complexes such that $ut=tu=\phi^{-1}(\xi)$ and $u^pt=0$ together with an identification $\tau: \phi^*M^\infty\cong M^{-\infty}$ similarly to \cref{ex:syntomification-fgaugesperfd}. Consider the filtration $\Fil^\bullet V\coloneqq M^\bullet/M^{\bullet-1}$ induced by the $t$-maps with associated graded $\gr^\bullet V$ \EDIT{and note that the filtered complex $\Fil^\bullet V$ identifies with the pullback of $\pi'_{X^\Syn, *}E$ to} 
\begin{equation*}
\EDIT{
(\O_{C, \red}^\Syn)_{u=0}=\O_{C, \dR, +}^\N\cong \Spec(\O_C/p)[t]/\G_m\;.
}
\end{equation*}
From the analogue of \cref{prop:syntomicetale-cohomologyhod} over $\Spf\O_C$, we may conclude that $\gr^\bullet V=0$ for $\bullet\geq -i+d$ and that, moreover, for $0\leq k\leq d$, the complex $\gr^{-i+d-k-1} V$ is concentrated in degrees at least $d-k$. In particular, as the filtration $\Fil^\bullet V$ is again \EDIT{complete by \cref{lem:finiteness-pushforwardcomplete}}, we conclude that $\Fil^\bullet V=0$ for $\bullet\geq -i+d$ and that $\Fil^{-i+d-k-1} V$ is concentrated in degrees at least $d-k$ for $0\leq k\leq d$. \EDIT{In other words,} the map $u: M^\bullet\rightarrow M^{\bullet+1}$ is an isomorphism for $\bullet\geq -i+d-1$ and, for $0\leq k\leq d$, the map $u: M^{-i+d-k-2}\rightarrow M^{-i+d-k-1}$ induces an isomorphism in cohomological degrees at most $d-k-1$ and an injection in degree $d-k$ for $0\leq k\leq d$. In particular, we conclude that the map $u^\infty: M^0\rightarrow M^\infty$ induces an isomorphism in cohomological degrees at most $i+1$ and an injection in degree $i+2$. Moreover, the map $u^p: M^{-1}\rightarrow M^{p-1}$ induces an isomorphism in cohomological degrees at most $i$ and an injection in degree $i+1$, so, due to $u^pt=0$, we conclude that the map $t: M^0\rightarrow M^{-1}$ induces the zero map in cohomological degrees at most $i+1$ and hence the same is true for $t^\infty: M^0\rightarrow M^{-\infty}$. By virtue of
\begin{equation*}
R\Gamma(\O_{C, \red}^\Syn, E)=\fib(M^0\xrightarrow{t^\infty-\tau u^\infty} M^{-\infty})\;,
\end{equation*} 
see \cref{ex:syntomification-fgaugesperfd}, we conclude that $R\Gamma(\O_{C, \red}^\Syn, E)$ is concentrated in degrees at least $i+1$ and this finishes the proof.
\end{proof}

\begin{proof}[Proof of \cref{thm:syntomicetale-maincoarseoc}]
We use the notations from the \EDIT{proof of \cref{thm:syntomicetale-mainfineoc} and can again reduce to showing that
\begin{equation*}
\gr^\Syn_k R\Gamma(X^\Syn, E)[\tfrac{1}{v_1}]\cong R\Gamma(X^\Syn_\red, E/v_{1, X}\{k(p-1)\})
\end{equation*}
vanishes in the case $E\in\Vect((X^\Syn)_{p=0})$. As before, we replace $E$ by $E/v_{1, X}\{k(p-1)\}$; then our assumption becomes that $E$ is a vector bundle} on $X^\Syn_\red$ with Hodge--Tate weights all at most $-d-1$ \EDIT{and we have to prove that} the complex $R\Gamma(X^\Syn_\red, E)$ vanishes. Again, from \cref{prop:syntomicetale-cohomologyhod}, we see that $\gr^\bullet V=0$ for $\bullet\geq 0$. As before, we now successively conclude that $\Fil^\bullet V=0$ for $\bullet\geq 0$ and hence the maps $u: M^\bullet\rightarrow M^{\bullet+1}$ are isomorphisms for $\bullet\geq -1$. In particular, the maps $u^\infty: M^0\rightarrow M^\infty$ and $u^p: M^{-1}\rightarrow M^{p-1}$ are isomorphisms and thus $t: M^0\rightarrow M^{-1}$ must be zero due to $u^pt=0$. By virtue of
\begin{equation*}
R\Gamma(\O_{C, \red}^\Syn, E)=\fib(M^0\xrightarrow{t^\infty-\tau u^\infty} M^{-\infty})\;,
\end{equation*} 
see \cref{ex:syntomification-fgaugesperfd}, we obtain $R\Gamma(\O_{C, \red}^\Syn, E)=0$, as required.
\end{proof}

\begin{rem}
\EDIT{
Note that the proofs of \cref{thm:syntomicetale-mainfineoc} and \cref{thm:syntomicetale-maincoarseoc} only make use of the fact that $\O_C$ is perfectoid, but do not rely on $C$ being algebraically closed.
}
\end{rem}
}

\section{A comparison of crystalline cohomology and étale cohomology}
\label{sect:cryset}

In this final section, we establish a generalisation of a comparison theorem of Colmez--Nizio{\l} between rational crystalline cohomology and étale cohomology, see \cite[Cor.\ 1.4]{ColmezNiziol}, which allows for coefficients. 

\begin{thm}
\label{thm:cryset-main}
Let $X$ be a $p$-adic formal scheme \EEDIT{which is smooth and proper over $\Spf\Z_p$}. For any crystalline local system $L$ on $X_\eta$ with Hodge--Tate weights all at most $-i-1$ for some $i\geq 0$, let $\cal{E}$ be its associated $F$-isocrystal. Then there is a natural morphism
\begin{equation*}
R\Gamma_\crys(X_{p=0}, \cal{E})[\tfrac{1}{p}][-1]\rightarrow R\Gamma_\proet(X_\eta, L)[\tfrac{1}{p}]\;,
\end{equation*}
which induces an isomorphism
\begin{equation*}
\tau^{\leq i} R\Gamma_\crys(X_{p=0}, \cal{E})[\tfrac{1}{p}][-1]\cong\tau^{\leq i} R\Gamma_\proet(X_\eta, L)[\tfrac{1}{p}]
\end{equation*}
and an injection on $H^{i+1}$.
\end{thm}

\EDIT{In order to clear up all the terms appearing in the statement, we will first recall the basic definitions surrounding crystalline local systems. Afterwards, we will go on to proving \cref{thm:cryset-main}, for which we will roughly proceed as follows: First, we show how to associate to any crystalline local system $L$ on $X_\eta$ an $F$-gauge $E$ on $X$ with the property that $T_\et(E)\cong L$ and $T_\dR(E)\cong \cal{E}$ under the identification $X^\dR\cong (X_{p=0})^\prism$ from \cite[Constr.\ 3.1.1]{FGauges}. Then we prove that
\begin{equation*}
\tau^{\leq i} R\Gamma_\dR(X, T_\dR(E))[\tfrac{1}{p}][-1]\cong\tau^{\leq i} R\Gamma_\proet(X_\eta, T_\et(E))[\tfrac{1}{p}]
\end{equation*}
for this $F$-gauge $E$ using the results from Section \ref{sect:beilfibsq} and Section \ref{sect:syntomicetale}.}

\subsection{Recollections on crystalline local systems}
\label{subsect:review-cryslocsys}

\EDIT{We briefly recap the theory of crystalline local systems along the lines of \cite[Sect. 2]{GuoReinecke} and start by recalling the definition of an $F$-isocrystal. For this, let $Y$ be an $\F_p$-scheme.

\begin{defi}
Consider the standard divided power structure on $\Z_p$ (i.e.\ the one given by $\gamma_n(p)=\tfrac{p^n}{n!}$ for $n\geq 0$). A \emph{crystal} (in coherent sheaves) on $Y$ is a sheaf $\cal{E}$ of $\O_{Y/\Z_p}$-modules on the big crystalline site $(Y/\Z_p)_\crys$ of $Y$ satisfying the following two properties:
\begin{enumerate}[label=(\roman*)]
\item For each object $(U, T, \gamma)$ of $(Y/\Z_p)_\crys$, the restriction $\cal{E}|_T$ of $\cal{E}$ to the Zariski-site of $T$ is a coherent sheaf.

\item For each map $\alpha: (U', T', \gamma')\rightarrow (U, T, \gamma)$ in $(Y/\Z_p)_\crys$, the induced map $\alpha^*(\cal{E}|_{T'})\rightarrow \cal{E}|_T$ is an $\O_T$-linear isomorphism.
\end{enumerate}
We denote the category of crystals on $Y$ by $\Crys(Y/\Z_p)$ and call the isogeny category
\begin{equation*}
\Isoc(Y/\Z_p)\coloneqq \Crys(Y/\Z_p)[\tfrac{1}{p}]
\end{equation*}
the category of \emph{isocrystals} on $Y$.
\end{defi}

The absolute Frobenius on $Y$ and the Witt vector Frobenius on $\Z_p$ (which is the identity in our case) are compatible and hence we obtain a morphism of sites
\begin{equation*}
F: (Y/\Z_p)_\crys\rightarrow (Y/\Z_p)_\crys
\end{equation*}
inducing a morphism of the corresponding topoi.

\begin{defi}
Let $Y$ be an $\F_p$-scheme. An \emph{$F$-isocrystal} on $Y$ is an isocrystal $\cal{E}$ on $Y$ equipped with an isomorphism $\phi: F^*\cal{E}\rightarrow\cal{E}$. We use $\Isoc^\phi(Y/\Z_p)$ to denote the category of $F$-isocrystals on $Y$. 
\end{defi}

To proceed further, we need to introduce certain period rings on the pro-étale site of any locally noetherian adic space $W$ over $\Spa\Q_p$. As affinoid perfectoid objects of $W_\proet$ form a basis of the topology by \cite[Lem.\ 4.6, Prop.\ 4.8]{PAdicHodgeTheory}, \EEDIT{a sheaf on $W_\proet$ can be defined by giving its values on affinoid perfectoid objects; since $\Q_p\rightarrow C$ is pro-étale, we may furthermore restrict to affinoid perfectoids over $\Spa C$, where we recall that $C$ denotes a fixed completed algebraic closure of $\Q_p$.}

\begin{defi}
Let $W$ be a locally noetherian adic space \EEDIT{over $\Spa\Q_p$}. The sheaves $\mathbb{A}_\crys, \mathbb{B}_\crys^+$ and $\mathbb{B}_\crys$ on $W_\proet$ are defined as follows: For any affinoid perfectoid $\Spa(R, R^+)\in W_\proet$ \EEDIT{over $\Spa C$}, we define
\begin{enumerate}[label=(\roman*)]
\item $\mathbb{A}_\crys(R, R^+)\coloneqq A_\crys(R^+)$, where we recall that the right-hand side was defined in \cref{ex:drstack-perfd},
\item $\mathbb{B}_\crys^+(R, R^+)\coloneqq \mathbb{A}_\crys(R, R^+)[\tfrac{1}{p}]$,
\item $\mathbb{B}_\crys(R, R^+)\coloneqq \mathbb{B}_\crys^+(R, R^+)[\tfrac{1}{\mu}]$, where $\mu\coloneqq [\epsilon^{1/p}]-1\in A_\inf(\O_C)$ is defined as in the proof of \cref{lem:beilfibsq-etalephimod}.
\end{enumerate}
\end{defi}

Now assume that $W=X_\eta$ is the generic fibre of a smooth $p$-adic formal scheme $X$. In this case, we can define a variant of the sheaf $\mathbb{B}_\crys$ with coefficients in any $F$-isocrystal on the special fibre $X_{p=0}$:

\begin{defi}
Let $X$ be a smooth $p$-adic formal scheme and $\cal{E}\in\Isoc^\phi(X_{p=0}/\Z_p)$ an $F$-isocrystal on the special fibre of $X$. Then we define sheaves $\mathbb{B}_\crys^+(\cal{E})$ and $\mathbb{B}_\crys(\cal{E})$ on $X_{\eta, \proet}$ as follows: For an affinoid perfectoid $\Spa(R, R^+)\in X_{\eta, \proet}$ \EEDIT{over $\Spa C$}, we define
\begin{enumerate}[label=(\roman*)]
\item $\mathbb{B}_\crys^+(\cal{E})(R, R^+)\coloneqq \cal{E}(A_\crys(R^+))[\tfrac{1}{p}]$, where
\begin{equation*}
\cal{E}(A_\crys(R^+))\coloneqq \lim_n \cal{E}(A_\crys(R^+)/p, A_\crys(R^+)/p^n, \gamma)
\end{equation*}
and $\gamma$ denotes the canonical divided power structure,

\item $\mathbb{B}_\crys(\cal{E})(R, R^+)\coloneqq \mathbb{B}_\crys^+(\cal{E})(R, R^+)[\tfrac{1}{\mu}]$.
\end{enumerate}
\end{defi}

Note that the Frobenius on $A_\crys(R^+)$ for any affinoid perfectoid $\Spa(R, R^+)$ over $\Spa C$ induces a Frobenius $F: \mathbb{B}_\crys\rightarrow\mathbb{B}_\crys$ on the sheaf $\mathbb{B}_\crys$ and we have an isomorphism $F^*\mathbb{B}_\crys(\cal{E})\cong\mathbb{B}_\crys(F^*\cal{E})$ for any $F$-isocrystal $\cal{E}$ on $X_{p=0}$. Thus, we obtain a Frobenius
\begin{equation*}
\phi: F^*\mathbb{B}_\crys(\cal{E})\cong\mathbb{B}_\crys(F^*\cal{E})\xrightarrow{\cong}\mathbb{B}_\crys(\cal{E})
\end{equation*}
on $\mathbb{B}_\crys(\cal{E})$ by functoriality of the assignment $\cal{E}\mapsto \mathbb{B}_\crys(\cal{E})$.

By \cite[Prop.\ 2.38]{GuoReinecke}, we may now define a crystalline local system as follows:

\begin{defi}
Let $X$ be a smooth $p$-adic formal scheme. A $\widehat{\Z}_p$-local system $L$ on $X_{\eta, \proet}$ is called \emph{crystalline} if there exists an $F$-isocrystal $\cal{E}$ on $X_{p=0}$ and an isomorphism
\begin{equation*}
\mathbb{B}_\crys(\cal{E})\cong \mathbb{B}_\crys\tensor_{\widehat{\Z}_p} L
\end{equation*}
compatible with the Frobenius on both sides. In this case, we call $\cal{E}$ the \emph{associated $F$-isocrystal} of $L$. Here, $\widehat{\Z}_p$ denotes the sheaf $\lim_n \Z/p^n\Z$ on $X_{\eta, \proet}$, as usual.
\end{defi}

\begin{rem}
Intuitively, one should think of a crystalline local system on $X_\eta$ as a continuous family of crystalline representations. Indeed, in yet unpublished work, Guo--Yang show that a local system $L$ on $X_\eta$ is crystalline if and only if its restrictions to all closed points of $X_\eta$ are crystalline representations, see \cite[Rem.\ 2.32]{GuoReinecke}. In particular, this also implies that the notion of a local system on $X_\eta$ being crystalline is independent of the choice of the smooth integral model $X$.
\end{rem}

Finally, we need to establish the notion of \emph{Hodge--Tate weights} of a crystalline local system $L$ on $X_\eta$. For this, we first need to define even more period sheaves.

\begin{defi}
\EEDIT{Let $W$ be a locally noetherian adic space over $\Spa\Q_p$. The sheaves $\mathbb{B}_\dR^+$ and $\mathbb{B}_\dR$ on $W_\proet$ are defined as follows: For any affinoid perfectoid $\Spa(R, R^+)\in W_\proet$ over $\Spa C$, we define}
\begin{enumerate}[label=(\roman*)]
\item \EEDIT{$\mathbb{B}_\dR^+(R, R^+)\coloneqq A_\inf(R^+)[\tfrac{1}{p}]^\wedge_{\Ker\theta}$ equipped with the $(\Ker\theta)$-adic filtration, where $\theta: A_\inf(R^+)[\tfrac{1}{p}]\rightarrow R$ is induced by the usual Fontaine map,}
\item \EEDIT{$\mathbb{B}_\dR(R, R^+)\coloneqq \mathbb{B}_\dR^+(R, R^+)[\tfrac{1}{t}]$ with the filtration
\begin{equation*}
\Fil^n \mathbb{B}_\dR\coloneqq \sum_{m\geq -n} t^{-m}\Fil^{m+n}\mathbb{B}_\dR^+(R, R^+)\;,
\end{equation*}
where $t\coloneqq \log [\epsilon]\in A_\crys(R, R^+)$ is defined as in (\ref{eq:beilfibsq-t}).}
\end{enumerate}
\end{defi}

\begin{defi}
Let $W$ be a locally noetherian adic space \EEDIT{over $\Spa\Q_p$}. The sheaf $\O\mathbb{B}_\dR^+$ on $W_\proet$ is defined as follows: For any affinoid perfectoid $\Spa(R, R^+)=\lim_i\Spa(R_i, R_i^+)\in W_\proet$ \EEDIT{over $\Spa C$}, we define
\begin{equation*}
\EEDIT{
\O\mathbb{B}_\dR^+(R, R^+)\coloneqq \colim_i (R_i^+\widehat{\tensor}_{\Z_p} A_\inf(R^+))[\tfrac{1}{p}]^\wedge_{\Ker\theta}\;,
}
\end{equation*}
where we abuse notation and denote by $\theta: (R_i^+\widehat{\tensor}_{\Z_p} A_\inf(R^+))[\tfrac{1}{p}]\rightarrow R$ the map induced by the two maps $R_i\rightarrow R$ and $A_\inf(R^+)\rightarrow R$. We equip $\O\mathbb{B}_\dR^+(R, R^+)$ with the $(\Ker\theta)$-adic filtration and hence obtain a decreasing filtration on $\O\mathbb{B}_\dR^+(R, R^+)[\tfrac{1}{t}]$ given by
\begin{equation*}
\Fil^n \O\mathbb{B}_\dR^+(R, R^+)[\tfrac{1}{t}]\coloneqq \sum_{m\geq -n} t^{-m}\Fil^{m+n}\O\mathbb{B}_\dR^+(R, R^+)\;.
\end{equation*}
This defines a filtered sheaf on \EEDIT{$W_\proet$} whose completion with respect to the filtration we denote by $\O\mathbb{B}_\dR$ and call the \emph{structural de Rham sheaf}. Its associated graded is the \emph{structural Hodge--Tate sheaf} denoted by $\O\mathbb{B}_\HT$.
\end{defi}

Now recall that being crystalline implies being de Rham in the sense of \cite[Def. 8.3]{PAdicHodgeTheory} by \cite[Cor.\ 2.37]{GuoReinecke} and thus, \emph{a fortiori}, also being Hodge--Tate in the sense of \cite[Def. 5.6]{HTWeights}, see \cite[Cor.\ 3.12]{LiuZhu}. In particular, this means that the sheaves
\begin{equation}
\label{eq:cryset-ddr}
D_\dR(L)\coloneqq \nu_*(L\tensor_{\widehat{\Z}_p} \O\mathbb{B}_\dR)
\end{equation}
and 
\begin{equation}
\label{eq:cryset-dht}
D_\HT(L)\coloneqq \nu_*(L\tensor_{\widehat{\Z}_p} \O\mathbb{B}_\HT)
\end{equation}
are vector bundles on $X_{\eta, \et}$. Here, the map $\nu: X_{\eta, \proet}\rightarrow X_{\eta, \et}$ denotes the canonical projection; however, we warn the reader that, contrary to our usual convention that all pushforwards are automatically derived, the pushforward $\nu_*$ appearing in (\ref{eq:cryset-dht}) and (\ref{eq:cryset-ddr}) is \emph{underived}.}

\begin{defi}
\EDIT{
Let $X$ be a smooth $p$-adic formal scheme and $L$ a crystalline local system on the generic fibre $X_\eta$. The \emph{Hodge--Tate weights} of $L$ are those integers $i\in\Z$ such that the graded vector bundle $D_\HT(L)$ is non-trivial in grading degree $i$.
}
\end{defi}

\begin{rem}
\label{rem:cryset-htweightsddr}
\EDIT{
Note that, by definition, the sheaf $D_\HT(L)$ is the associated graded of the filtered vector bundle $D_\dR(L)$. Thus, one may alternatively define the Hodge--Tate weights of $L$ as the integers $i\in\Z$ for which the filtration on $D_\dR(L)$ jumps.
}
\end{rem}

\begin{rem}
\EDIT{
If $X_\eta$ is geometrically connected, by \cite[Thm.\ 1.1]{HTWeights}, we may alternatively define the Hodge--Tate weights of $L$ as the Hodge--Tate weights of the $p$-adic Galois representation $L_{\ol{x}}$ of $k(x)$, where $x$ denotes any classical point of $X_\eta$ (i.e.\ any point of $X_\eta$ seen as a rigid-analytic variety) and $k(x)$ is its residue field.
}
\end{rem}

\subsection{$F$-gauges associated to crystalline local systems}
\label{subsect:locsystofgauge}

We first show how to associate an $F$-gauge to any crystalline local system on the generic fibre of a smooth formal scheme $X$ over $\Z_p$. To this end, recall that, \EDIT{in independent work, both} Guo--Reinecke \EDIT{and Du--Liu--Moon--Shimizu} have recently given a description of crystalline local systems on $X_\eta$ in terms of the prismatic site of $X$, which we review here briefly, \EDIT{see \cite{GuoReinecke} and \cite{DuLiuMoonShimizu}.}

\begin{defi}
For any prism $(A, I)$, let $\Vect^\phi(\Spec(A)\setminus V(p, I))$ denote the category of vector bundles $M$ on $\Spec(A)\setminus V(p, I)$ together with an $A$-linear isomorphism $\phi^*M[1/I]\cong M[1/I]$, where $\phi$ is the Frobenius of $A$. Then the category of \emph{analytic prismatic $F$-crystals} on $X$ is defined as
\begin{equation*}
\Vect^{\an, \phi}(X_\prism)\coloneqq \lim_{(A, I)\in X_\prism} \Vect^\phi(\Spec(A)\setminus V(p, I))\;.
\end{equation*}
\end{defi}

In our setting, it turns out that crystalline local systems and analytic prismatic $F$-crystals are the same:

\begin{prop}
\label{prop:cryset-locsysvectan}
Let $X$ be a smooth $p$-adic formal scheme. Then there is a natural equivalence of categories
\begin{equation*}
\Vect^{\an, \phi}(X_\prism)\cong \Loc^\crys_{\Z_p}(X_\eta)\;.
\end{equation*}
\end{prop}
\begin{proof}
See \cite[Thm.\ A]{GuoReinecke} \EDIT{or \cite[Thm.\ 1.3, Rem.\ 1.5]{DuLiuMoonShimizu}.}
\end{proof}

Next, given an analytic prismatic $F$-crystal on $X$, we extend it to a prismatic $F$-crystal in perfect complexes on $X$ by (underived) pushforward along the open immersion 
\begin{equation*}
\EEDIT{
\Spec(A)\setminus V(p, I)\rightarrow\Spec A
}
\end{equation*}
for all $(A, I)\in X_\prism$.

\begin{prop}
\label{prop:cryset-vectanpushforward}
Let $X$ be a smooth $p$-adic formal scheme. Then \EEDIT{(underived)} pushforward along the open immersions $\Spec(A)\setminus V(p, I)\rightarrow \Spec A$ for all $(A, I)\in X_\prism$ induces a fully faithful \EDIT{embedding}
\begin{equation*}
\Vect^{\an, \phi}(X_\prism)\hookrightarrow \Perf^\phi(X_\prism)
\end{equation*}
from the category of analytic prismatic $F$-crystals to the category of prismatic $F$-crystals in perfect complexes.
\end{prop}
\begin{proof}
See \cite[Thm.\ 5.10]{GuoReinecke}.
\end{proof}

To finish the construction, we have to show how to extend a prismatic $F$-crystal in perfect complexes $M$, i.e.\ a perfect complex on $X^\prism$ together with a Frobenius structure, to a perfect complex on $X^\N$. The Frobenius structure will ensure that the result descends to $X^\Syn$, see \cite[Ex.\ 6.1.7]{FGauges}. By quasisyntomic descent, see \cref{lem:filprism-quasisyntomiccover}, it suffices to describe the procedure for $X=\Spf R$ with a quasiregular semiperfectoid ring $R$. Here, we have
\begin{equation*}
R^\N\cong \Spf\Rees(\Fil^\bullet_\N \Prism_R)/\G_m
\end{equation*}
by \cref{ex:filprism-qrsp} and thus we need to equip $M$ with a decreasing filtration $\Fil^\bullet M$ compatible with the Nygaard filtration on $\Prism_R$. This construction is supplied by the following definition from \cite[Def. 6.6.6]{FGauges}:

\begin{defi}
Let $R$ be a quasiregular semiperfectoid ring and $(M, \tau)\in\Mod^\phi(\Prism_R)$ be a prismatic $F$-crystal in finitely presented modules on $\Spf R$. A \emph{Nygaardian filtration} on $M$ is a filtration $\Fil^\bullet M$ of $M$ in $(p, I)$-complete $\Prism_R$-modules such that the map
\begin{equation*}
M\xrightarrow{\can} \phi^*M\hookrightarrow \phi^*M[1/I]\overset{\tau}{\cong} M[1/I]
\end{equation*}
carries $\Fil^i M$ into $I^i M$ for all $i\in\Z$; here, $I$ is the effective Cartier divisor on $\Prism_R$ defining the initial prism of $R_\prism$. A Nygaardian filtration is called \emph{saturated} if it is maximal, i.e.\ if $\Fil^\bullet M$ is the preimage of $I^\bullet M$ under the map above. 
\end{defi}

Indeed, this construction gives what we want provided we impose a small restriction on the class of prismatic $F$-crystals in perfect complexes we are considering.

\begin{defi}
Let $X$ be a \EDIT{$p$-quasisyntomic} $p$-adic formal scheme. We define the category $\Perf^\phi_{I-\mathrm{tf}}(X_\prism)$ to be the full subcategory of $\Perf^\phi(X_\prism)$ consisting of prismatic $F$-crystals $\cal{E}$ satisfying the following property: For any quasiregular semiperfectoid ring $R$ such that $\Spf R\rightarrow X$ is \EDIT{$p$-quasisyntomic}, the complex $\cal{E}(\Prism_R)$ is concentrated in cohomological degree zero and has no $I$-torsion, where again $(\Prism_R, I)$ is the inital object of $R_\prism$.
\end{defi}

\begin{prop}
\label{prop:cryset-fcrysfgauge}
Let $X$ be a smooth $p$-adic formal scheme. Then there is a fully faithful embedding
\begin{equation*}
\Perf^\phi_{I-\mathrm{tf}}(X_\prism)\hookrightarrow \Perf(X^\Syn)
\end{equation*}
characterised by the requirement that its composition with pullback along $R^\Syn\rightarrow X^\Syn$ for any quasiregular semiperfectoid ring $R$ such that $\Spf R\rightarrow X$ is $p$-completely flat sends a prismatic $F$-crystal $\cal{E}$ to the module $\cal{E}(\Prism_R)$ equipped with its saturated Nygaardian filtration.
\end{prop}
\begin{proof}
See \cite[Thm.\ 2.31]{GuoLi}.
\end{proof}

To be able to apply the above proposition in our situation, we only need to check that the essential image of the functor from \cref{prop:cryset-vectanpushforward} is contained in $\Perf^\phi_{I-\mathrm{tf}}(X_\prism)$. However, for any quasiregular semiperfectoid ring $R$, any object in the essential image of 
\begin{equation*}
\Vect^{\an, \phi}(R_\prism)\rightarrow \Perf^\phi(R_\prism)\cong \Perf^\phi(\Prism_R)
\end{equation*}
identifies with the global sections of a vector bundle on $\Spec(\Prism_R)\setminus V(p, I)$ and this immediately implies the claim as $I$ is an effective Cartier divisor on $\Prism_R$ and hence $\Prism_R$ has no $I$-torsion.

We summarise the results we have collected so far:

\begin{thm}
\label{thm:cryset-locsysfgauges}
Let $X$ be a smooth $p$-adic formal scheme. Then there is a fully faithful embedding
\begin{equation*}
\Loc^\crys_{\Z_p}(X_\eta)\hookrightarrow \Perf(X^\Syn)
\end{equation*}
which is a weak right inverse to the étale realisation. The image of a crystalline local system $L$ under this embedding is called the \emph{associated $F$-gauge} $E$ of $L$ and has the property that $T_\dR(E)$ identifies with the associated $F$-isocrystal of $L$.
\end{thm}
\begin{proof}
The existence of the embedding follows by combining \cref{prop:cryset-locsysvectan}, \cref{prop:cryset-vectanpushforward} and \cref{prop:cryset-fcrysfgauge}; moreover, \EDIT{by the description of the étale realisation from \cref{rem:syntomification-etalerealisationglobal},} the relation with the étale realisation is clear from the construction: \EDIT{indeed, the functor from \cref{prop:cryset-fcrysfgauge} is a weak right inverse to the functor $\Perf(X^\Syn)\rightarrow\Perf^\phi(X_\prism)$ from (\ref{eq:syntomification-etalerealisationglobal}) and this implies the claim.} Finally, the last \EDIT{statement about compatibility with the de Rham realisation} follows from the commutativity of the diagram
\begin{equation*}
\begin{tikzcd}[row sep=tiny]
X^\dR\ar[dr, "i_\dR"]\ar[dd, "\cong", swap] \\
& X^\prism \\
(X_{p=0})^\prism\ar[ur]
\end{tikzcd}
\end{equation*}
and the fact that the associated $F$-isocrystal of $L$ can be obtained from the associated analytic prismatic $F$-crystal via \EDIT{the composite functor
\begin{equation*}
\Vect^{\an, \phi}(X_\prism)\rightarrow\Vect^{\an, \phi}((X_{p=0})_\prism)\cong \Isoc^\phi(X_{p=0}/\Z_p)
\end{equation*}
from} \cite[Constr.\ 3.9.(ii)]{GuoReinecke}, see the proof of \cite[Thm.\ 3.12]{GuoReinecke}. 
\end{proof}

\begin{defi}
The essential image of the embedding 
\begin{equation*}
\Loc^\crys_{\Z_p}(X_\eta)\hookrightarrow \Perf(X^\Syn)
\end{equation*}
is called the category of \emph{reflexive $F$-gauges} on $X$ and denoted by $\Coh^\refl(X^\Syn)$.
\end{defi}

\begin{rem}
We shortly discuss why the notions of Hodge--Tate weights of a crystalline local system $L$ and its associated $F$-gauge $E$ are compatible. By \cref{rem:cryset-htweightsddr}, this would follow from the fact that $D_\dR(L)$ can be identified with (the pullback to $X_\eta$ of) $T_{\dR, +}(E)$. Using \cite[Cor.\ 6.19]{PAdicHodgeTheory}, it then suffices to prove that there is a filtered isomorphism
\begin{equation*}
L\tensor_{\widehat{\Z}_p} \O\mathbb{B}_\dR\cong T_{\dR, +}(E)\tensor_{\O_{X_\eta}} \O\mathbb{B}_\dR\;.
\end{equation*}
\EEDIT{However, the unfiltered isomorphism exists by \cite[Prop.\ 2.36]{GuoReinecke} and the fact that $T_\dR(E)$ identifies with the associated $F$-isocrystal of $L$. To check that this isomorphism respects the filtrations, we may now proceed locally on $X_{\eta, \proet}$ and hence assume $\O\mathbb{B}_\dR^+\cong \mathbb{B}_\dR^+[[X_1, \dots, X_n]]$ by \cite[Prop.\ 6.10]{PAdicHodgeTheory}. Then we are reduced to showing that 
\begin{equation*}
L\tensor_{\widehat{\Z}_p} \mathbb{B}_\dR\cong T_{\dR, +}(E)\tensor_{\O_{X_\eta}} \mathbb{B}_\dR
\end{equation*}
compatibly with the filtrations on either side, which may be checked on affinoid perfectoid objects $\Spa(R, R^+)\in X_{\eta, \proet}$ over $\Spa C$ such that $L|_{\Spa(R, R^+)}$ is trivial since such objects form a basis of $X_{\eta, \proet}$ -- however, in this case, the claim follows by virtually the same proof as \cref{lem:beilfibsq-etalefiltered} with the use of \cite[Lem.\ 4.26]{BMS} replaced by \cite[Lem.\ 3.13]{GuoReinecke}.}
\end{rem}

\subsection{Proof of the main theorem}

We now move on to prove the necessary comparison results for reflexive $F$-gauges. However, for simplicity, we start with the case of vector bundles:

\begin{prop}
\label{prop:cryset-fine}
Let $X$ be a $p$-adic formal scheme \EEDIT{which is smooth and proper over $\Spf\Z_p$}. For any vector bundle $F$-gauge $E\in\Vect(X^\Syn)$ with Hodge--Tate weights all at most $-i-1$ for some $i\geq 0$, there is a natural morphism
\begin{equation*}
R\Gamma_\dR(X, T_\dR(E))[\tfrac{1}{p}][-1]\rightarrow R\Gamma_\proet(X_\eta, T_\et(E))[\tfrac{1}{p}]\;,
\end{equation*}
which induces an isomorphism
\begin{equation*}
\tau^{\leq i} R\Gamma_\dR(X, T_\dR(E))[\tfrac{1}{p}][-1]\cong\tau^{\leq i} R\Gamma_\proet(X_\eta, T_\et(E))[\tfrac{1}{p}]
\end{equation*}
and an injection on $H^{i+1}$.
\end{prop}
\begin{proof}
Combining \cref{thm:beilfibsq-fmrat} and \cref{prop:finiteness-main}, we obtain a fibre sequence
\begin{equation*}
R\Gamma_\dR(X, T_\dR(E))[\tfrac{1}{p}][-1]\rightarrow R\Gamma_\Syn(X, E)[\tfrac{1}{p}]\rightarrow \EDIT{\Fil^0_\Hod R\Gamma_\dR(X, T_{\dR, +}(E))[\tfrac{1}{p}]}\;.
\end{equation*}
In view of \cref{thm:syntomicetale-mainfine}, it thus suffices to show that \EDIT{$\Fil^0_\Hod R\Gamma_\dR(X, T_{\dR, +}(E))$} is concentrated in degrees at least $i+1$. To this end, identify $T_{\dR, +}(\pi_{X^\Syn, *}E)$ with a filtered object $\Fil^\bullet V$ \EEDIT{and let $d$ be the relative dimension of $X$ over $\Spf\Z_p$}. Then \cref{prop:syntomicetale-cohomologyhod} implies that $\gr^\bullet V=0$ for $\bullet\geq d-i$ and that, moreover, $\gr^{d-i-1-k} V$ is concentrated in degrees at least $d-k$ for $0\leq k\leq d$. As the filtration $\Fil^\bullet V$ is \EDIT{complete by \cref{lem:finiteness-pushforwardcomplete}}, we may \EDIT{thus} conclude that $\Fil^0 V$ is concentrated in degrees at least $i+1$, which yields the claim.
\end{proof}

We can also formulate an analogous result for arbitrary perfect $F$-gauges, albeit in a slightly weaker range:

\begin{prop}
\label{prop:cryset-coarse}
Let $X$ be a $p$-adic formal scheme \EEDIT{which is proper and smooth} of relative dimension $d$ over $\Spf\Z_p$. For any perfect $F$-gauge $E\in\Perf(X^\Syn)$ with Hodge--Tate weights all at most $-d-2$ for some $i\geq 0$, there is a natural isomorphism
\begin{equation*}
R\Gamma_\dR(X, T_\dR(E))[\tfrac{1}{p}][-1]\cong R\Gamma_\proet(X_\eta, T_\et(E))[\tfrac{1}{p}]\;.
\end{equation*}
\end{prop}
\begin{proof}
As in the proof of \cref{prop:cryset-fine}, now using \cref{thm:syntomicetale-maincoarse} in place of \cref{thm:syntomicetale-mainfine}, we are reduced to showing that \EDIT{$\Fil^0_\Hod R\Gamma_\dR(X, T_{\dR, +}(E))=0$}. However, \EEDIT{by \cref{lem:finiteness-pushforwardcomplete}}, the filtered object $T_{\dR, +}(\pi_{X^\Syn, *}E)\in\D(\A^1_-/\G_m)$ is \EDIT{complete} and \EEDIT{\cref{prop:syntomicetale-cohomologyhod} shows that} $\pi_{X^\Syn, *}E$ has Hodge--Tate weights all at most $-2$, hence $T_{\dR, +}(\pi_{X^\Syn, *}E)$ is zero in filtration degree zero and this implies the claim.
\end{proof}

Now observe that, by \cref{thm:cryset-locsysfgauges}, having the conclusion of \cref{prop:cryset-fine} for reflexive $F$-gauges instead of just vector bundle $F$-gauges would immediately imply \cref{thm:cryset-main}. \EDIT{This will be a consequence of the following lemma:}

\begin{lem}
\label{lem:cryset-degreezero}
Let $X$ be a smooth $p$-adic formal scheme and $E\in\Perf(X^\Syn)$ a reflexive $F$-gauge. Then the pullback of $E$ to $X^\N_\Hod$ is concentrated in cohomological degree zero. 
\end{lem}
\begin{proof}
By quasisyntomic descent, see \cref{lem:filprism-quasisyntomiccover}, we may check this in the case where $X=\Spf R$ for a $p$-torsionfree quasiregular semiperfectoid ring $R$. \EEDIT{For this, choose a perfectoid ring $R_0$ mapping to $R$; then also the initial object $(\Prism_{R_0}, (\xi))$ of the absolute prismatic site of $R_0$ maps to the initial object $(\Prism_R, I)$ of the absolute prismatic site of $R$ and, in particular, we conclude that $I=\xi\Prism_R$ is principal by rigidity. Now recall that $R^\N\cong \Spf\Rees(\Fil^\bullet_\N\Prism_R)/\G_m$ by \cref{ex:filprism-qrsp} and that} the $F$-gauge $E$ arises by endowing the module $M$ of global sections of a vector bundle on $\Spec(\Prism_R)\setminus V(p, I)$ equipped with a Frobenius structure $\tau: \phi^*M[1/I]\cong M[1/I]$ with its saturated Nygaardian filtration $\Fil^\bullet M$ \EEDIT{by \cref{prop:cryset-fcrysfgauge}.} We begin by noting that $M$ is $p$-torsionfree since $\Prism_R$ is and therefore the pullback $\Fil^\bullet M/p$ of $E$ to $(R^\N)_{p=0}$ is concentrated in degree zero. Now further observe that $M/I$ is $p$-torsionfree by virtue of the exact sequence
\begin{equation*}
\begin{tikzcd}
0\ar[r] & M\ar[r] & M[1/I]\oplus M[1/p]\ar[r] & M[1/pI]\nospacepunct{\;;}
\end{tikzcd}
\end{equation*}
\EEDIT{indeed, if $pm=am'$ for some $m, m'\in M$ and $a\in I$, then we would have $m/a=m'/p$, i.e.\ $m/a\in M$, which would in turn imply that $m$ is already zero in $M/I$. This now} implies that, \EEDIT{for all $i\in\Z$, the map} $\Fil^i M/p\rightarrow\Fil^{i-1} M/p$ is injective: If $m=pm'$ for some $m\in\Fil^i M$ and $m'\in\Fil^{i-1} M$, then $p\tau(m')=\tau(m)\in I^i M$; \EEDIT{as $\tau(m')\in I^{i-1}$ and $I$ is generated by $\xi$, we obtain the equation $p\tfrac{\tau(m')}{\xi^{i-1}}=\tfrac{\tau(m)}{\xi^{i-1}}\in IM$ and and hence $\tfrac{\tau(m')}{\xi^{i-1}}\in IM$, i.e.\ $\tau(m')\in I^iM$,} as $M/I$ is $p$-torsionfree -- however, this implies $m'\in\Fil^i M$ by definition of the saturated Nygaardian filtration and hence already $m=0\in\Fil^i M/p$. Therefore, also the pullback $\gr^\bullet M/p=(\Fil^\bullet M/p)/(\Fil^{\bullet+1} M/p)$ of $E$ to $(R^\N)_{p=t=0}=R^\N_{\HT, c}$ is concentrated in degree zero. 

\EEDIT{Finally, to compute the further pullback of $E$ to $R^\N_\Hod$, observe that there is a commutative diagram
\begin{equation*}
\begin{tikzcd}[ampersand replacement=\&]
R^\N_{\dR, +}\ar[r]\ar[d] \& R^\N\ar[d] \\
R_{0, \dR, +}^\N\ar[r]\ar[d] \& R_0^\N\ar[d] \\
\Z_{p, \dR, +}^\N\ar[r] \& \Z_p^\N\nospacepunct{\;,}
\end{tikzcd}
\end{equation*}
in which the large rectangle and the bottom square are pullbacks by definition. Thus, also the top square is a pullback and recalling that $R_{0, \dR, +}=(R_0^\N)_{p=u^p=0}$, where we use the isomorphism
\begin{equation*}
R_0^\N\cong \Spf(A_\inf(R_0)\langle u, t\rangle/(ut-\phi^{-1}(\xi)))/\G_m
\end{equation*}
from \cref{ex:filprism-qrsp}, we learn that $R^\N_{\dR, +}$ identifies with the locus 
\begin{equation*}
\{p=\phi^{-1}(\xi)^pt^{-p}=0\}=\{p=\xi t^{-p}=0\}\subseteq R^\N\;.
\end{equation*}
Thus, we conclude that the pullback of $\gr^\bullet M/p$ from $R^\N_{\HT, c}$ to $R^\N_\Hod=R^\N_{\HT, c}\cap R^\N_{\dR, +}$ identifies with the graded object given by} the cofibre of the map
\begin{equation*}
(\Fil^{\bullet-p} M/p)/(\Fil^{\bullet-p+1} M/p)\tensor_{\Prism_R/p} \xi\Prism_R/p\rightarrow (\Fil^\bullet M/p)/(\Fil^{\bullet+1} M/p)\;.
\end{equation*}
However, \EEDIT{recalling that $I=\xi\Prism_R$,} by definition of the saturated Nygaardian filtration, this is clearly an injection since multiplication by $\xi$ corresponds to multiplication by $\xi^p$ after applying $\tau$ by $\phi$-linearity \EEDIT{and $\xi$ is a non-zerodivisor.}
\end{proof}

\begin{proof}[Proof of \cref{thm:cryset-main}]
\EDIT{As observed previously, \cref{thm:cryset-locsysfgauges} reduces us to having the conclusion of \cref{prop:cryset-fine} for reflexive $F$-gauges instead of just vector bundle $F$-gauges. However, note} that all we really need \EDIT{in order} to copy the proofs of \cref{prop:cryset-fine} and \cref{thm:syntomicetale-mainfine} verbatim for a reflexive $F$-gauge $E$ is that the pullback of $E$ to $(X^\Hod)_{p=0}$ is concentrated in cohomological degree zero. \EDIT{As this follows from \cref{lem:cryset-degreezero} and the flatness of the cover $(X^\Hod)_{p=0}\rightarrow X^\N_\Hod$ from (\ref{eq:syntomicetale-coverhod}), we are done.}
\end{proof}

\appendix

\section{Some base change results}
\label{sect:basechange}

Throughout this paper, we often use base change results for cartesian squares of formal stacks like
\begin{equation*}
\begin{tikzcd}
X^\prism\ar[r, "j_\dR"]\ar[d] & X^\N\ar[d] \\
\Z_p^\prism\ar[r, "j_\dR"] & \Z_p^\N\nospacepunct{\;.}
\end{tikzcd}
\end{equation*}
However, as we are working with formal stacks instead of algebraic (or derived) stacks, such base change compatibilities are usually not automatic. Thus, in this appendix, we indicate how to prove these statements -- in the main body of the paper, we will then often use them without further mention.

\begin{defi}
We say that a cartesian diagram
\begin{equation*}
\begin{tikzcd}
\cal{X}'\ar[r, "g'"]\ar[d, "f'", swap] & \cal{X}\ar[d, "f"] \\
\cal{Y}'\ar[r, "g"] & \cal{Y}
\end{tikzcd}
\end{equation*}
of stacks satisfies \emph{base change for bounded below complexes} if, for any \EEDIT{$F\in\D^+(X)$, see \cref{defi:beilfibsq-tstructbounded},} there is an isomorphism
\begin{equation*}
g^*f_*F\cong f'_*{g'}^*F
\end{equation*}
given by the natural map.
\end{defi}

As it turns out, \EEDIT{many} of the base change results we need are an application of the following general statement:

\begin{prop}
\label{prop:basechange-locclosed}
Let $f: \cal{X}\rightarrow\cal{Y}$ be a morphism of stacks over $\Spf\Z_p$ such that $\cal{X}$ can be written as the colimit of a simplicial flat affine formal $\cal{Y}$-scheme $U^\bullet$ and assume that $\cal{Y}$ admits a flat cover from a regular affine formal scheme. \EEDIT{Then,} for any locally closed immersion $i: \cal{Z}\rightarrow\cal{Y}$, the cartesian diagram
\begin{equation*}
\begin{tikzcd}
\cal{X}_{\cal{Z}}\ar[r, "i'"]\ar[d, "f'", swap] & \cal{X}\ar[d, "f"] \\
\cal{Z}\ar[r, "i"] & \cal{Y}
\end{tikzcd}
\end{equation*}
satisfies base change for bounded below complexes. 
\end{prop}

\EDIT{
Before we can prove the above proposition, we need to establish the following technical statement about the commutation of filtered colimits with cosimplicial totalisations:}

\begin{lem}
\EDIT{
\label{lem:syntomicetale-colimtot}
Let $\cal{X}$ be a stack admitting a flat cover from an affine scheme or a noetherian affine formal scheme. Assume we are given a filtered diagram of cosimplicial objects 
\begin{equation*}
F: I\times\Delta^\op\rightarrow \D^{\geq 0}(\cal{X})\subseteq \D(\cal{X})
\end{equation*}
in the category $\D(\cal{X})$ whose terms are all cosimplicial objects of the full subcategory $\D^{\geq 0}(\cal{X})$, i.e.\ $I$ is a filtered category and $\Delta$ denotes the simplex category, as usual. Then taking colimits over $I$ in $\D(\cal{X})$ commutes with cosimplicial totalisation, i.e.\ there is a natural isomorphism}
\begin{equation*}
\EDIT{
\colim_{i\in I} \Tot(F(i, \bullet))\cong \Tot(\colim_{i\in I} F(i, \bullet))
}
\end{equation*}
\EDIT{in the category $\D(\cal{X})$.}
\end{lem}
\begin{proof}
\EDIT{
We first claim that filtered colimits in $\D(\cal{X})$ have finite cohomological dimension, i.e.\ that there is some $r\geq 0$ with the property that, if $G, G': K\rightarrow \D(\cal{X})$ are two filtered diagrams of the same shape together with a natural transformation $G\rightarrow G'$ with the property that $\tau^{\leq -1}G(k)\cong \tau^{\leq -1}G'(k)$ for all $k\in K$, then there is a natural isomorphism
\begin{equation}
\label{eq:syntomicetale-colimfindim}
\tau^{\leq -r-1}\colim_{k\in K} G(k)\cong\tau^{\leq -r-1}\colim_{k\in K} G'(k)\;.
\end{equation} 
Indeed, as colimits commute with pullbacks, this may be checked on a flat cover. Now there are two cases to consider: If we are given a flat cover $\Spec A\rightarrow\cal{X}$, the claim is clear as filtered colimits of $A$-modules are exact. If, however, we are given a flat cover $\Spf A\rightarrow\cal{X}$ with $A$ noetherian, then colimits in $\D(\Spf A)\cong \widehat{\D}(A)$ are computed by first taking the colimit in $\D(A)$ and then applying derived completion, so in this case the claim follows from the fact that derived completion has finite cohomological dimension in our setting due to the ideal of definition of $A$ being finitely generated, see \cite[Tag 0AAJ]{Stacks}. Thus, the claim is proved. In particular, note that, by picking $G'(k)=0$ for all $k$, we have also shown that, if $G(k)\in\D^{\geq 0}(\cal{X})$ for all $k\in K$, then
\begin{equation*}
\colim_{k\in K} G(k)\in \D^{\geq -r}(\cal{X})\;.
\end{equation*}

We now prove the claim from the statement. To check that the natural map 
\begin{equation*}
\colim_{i\in I} \Tot(F(i, \bullet))\rightarrow \Tot(\colim_{i\in I} F(i, \bullet))
\end{equation*}
is an isomorphism, it suffices to check that it induces an isomorphism on $\tau^{\leq n}$ for all $n$. As $\colim_{i\in I} F(i, \bullet)\in\D^{\geq -r}(\cal{X})$ by the first paragraph, the totalisation on the right-hand side may be replaced by a partial totalisation if we are only interested in the result after $n$-truncation; more precisely, we have
\begin{equation*}
\tau^{\leq n} \Tot(\colim_{i\in I} F(i, \bullet))\cong \tau^{\leq n}\lim_{j\in\Delta_{\leq n+r+1}} \colim_{i\in I} F(i, j)\;,
\end{equation*}
where $\Delta_{\leq n+r+1}$ denotes the full subcategory of $\Delta$ spanned by the objects $j$ with $j\leq n+r+1$. Now using that finite limits commute with filtered colimits and the finite cohomological dimension of filtered colimits, we obtain
\begin{equation*}
\begin{split}
\tau^{\leq n}\lim_{j\in\Delta_{\leq n+r+1}} \colim_{i\in I} F(i, j)&\cong \tau^{\leq n}\colim_{i\in I} \lim_{j\in\Delta_{\leq n+r+1}} F(i, j) \\
&\cong \tau^{\leq n} \colim_{i\in I} \tau^{\leq n+r}\lim_{j\in\Delta_{\leq n+r+1}} F(i, j)\;.
\end{split}
\end{equation*}
Finally, we use that $F(i, j)\in\D^{\geq 0}(\cal{X})$ to conclude that the $(n+r)$-truncation of the partial totalisation above actually agrees with the $(n+r)$-truncation of the full totalisation and make use of the finite cohomological dimension of filtered colimits once again to conclude
\begin{equation*}
\begin{split}
\tau^{\leq n} \colim_{i\in I} \tau^{\leq n+r}\lim_{j\in\Delta_{\leq n+r+1}} F(i, j)&\cong \tau^{\leq n} \colim_{i\in I} \tau^{\leq n+r}\Tot(F(i, \bullet)) \\
&\cong \tau^{\leq n} \colim_{i\in I} \Tot(F(i, \bullet))\;,
\end{split}
\end{equation*}
so we are done.
}
\end{proof}

\begin{proof}[Proof of \cref{prop:basechange-locclosed}]
We follow the argument given in the proof of \cite[Thm.\ 3.3.5]{FGauges}. Namely, using the presentations $U^\bullet\rightarrow\cal{X}$ and $U^\bullet_{\cal{Z}}\rightarrow\cal{X}_{\cal{Z}}$, we are reduced to checking that the formation of cosimplicial totalisations in $\D^{\geq n}(\cal{Y})$ commutes with pullback along $i$. However, as $i$ is a locally closed immersion, pullback along $i$ is merely tensoring with $i_*\O_{\cal{Z}}$ and this is a perfect complex over some localisation of $\cal{Y}$ by regularity of $\cal{Y}$. Thus, the claim now follows from the fact that cosimplicial totalisations in $\D^{\geq n}(\cal{Y})$ commute with filtered colimits, \EEDIT{see \cref{lem:syntomicetale-colimtot},} as well as tensoring with perfect complexes.
\end{proof}

By way of example, we will now indicate how to use \cref{prop:basechange-locclosed} to prove many of the base change results we need.

\begin{cor}
\label{cor:basechange-nygaardlocclosed}
Let $X$ be a bounded $p$-adic formal scheme which is $p$-quasisyntomic and qcqs. Then the cartesian diagram
\begin{equation*}
\begin{tikzcd}
X^\prism\ar[r, "j_\dR"]\ar[d] & X^\N\ar[d] \\
\Z_p^\prism\ar[r, "j_\dR"] & \Z_p^\N
\end{tikzcd}
\end{equation*}
satisfies base change for bounded below complexes. 
\end{cor}
\begin{proof}
To apply \cref{prop:basechange-locclosed}, first \EDIT{recall} that $j_\dR: \Z_p^\prism\rightarrow\Z_p^\N$ is an open immersion and that \EDIT{there is a flat cover} $\Spf\Z_p\langle u, t\rangle\rightarrow\Z_p^\N$, \EDIT{see \cref{rem:filprism-regular}}. To provide the necessary simplicial presentation of $X^\N$, first observe that, by quasisyntomic descent, see \cref{lem:filprism-quasisyntomiccover}, we may assume that $X=\Spf R$ for a quasiregular semiperfectoid ring $R$. However, then the claim is clear by the explicit description of $R^\N$ from \cref{ex:filprism-qrsp}.
\end{proof}

\begin{cor}
\label{cor:basechange-dr}
Let $X$ be a bounded $p$-adic formal scheme which is smooth and qcqs. Then the cartesian diagram
\begin{equation*}
\begin{tikzcd}
X^\Hod\ar[r]\ar[d] & X^{\dR, +}\ar[d] \\
\Z_p^\Hod\ar[r] & \Z_p^{\dR, +}
\end{tikzcd}
\end{equation*}
satisfies base change for bounded below complexes.
\end{cor}
\begin{proof}
As $\Z_p^{\dR, +}=\A^1_-/\G_m$ \EDIT{admits a flat cover from the regular affine formal scheme $\A^1$} and $\Z_p^\Hod=B\G_m\rightarrow\A^1_-/\G_m=\Z_p^{\dR, +}$ is a closed immersion, we again want to apply \cref{prop:basechange-locclosed}. To obtain the necessary simplicial presentation of $X^{\dR, +}$, first observe that the functor $X\mapsto X^{\dR, +}$ preserves (affine) étale maps by the proof of \cite[Thm.\ 2.3.6]{FGauges}. As $X$ Zariski-locally admits an affine étale map $X\rightarrow\A^n$, we may thus assume that we are given an affine étale map \EEDIT{$X^{\dR, +}\rightarrow (\A^n)^{\dR, +}=(\G_a/\V(\O(1))^\sharp)^n$} of stacks over $\A^1_-/\G_m$. \EEDIT{Then the pullback $\widetilde{X}$ of the ${\V(\O(1))^\sharp}^n$-torsor $\G_a^n\rightarrow (\G_a/\V(\O(1))^\sharp)^n$ to $X^{\dR, +}$ is a flat affine formal scheme over $\A^1_-/\G_m$ and we have $X^{\dR, +}\cong \widetilde{X}/{\V(\O(1))^\sharp}^n$, which yields the desired presentation.}
\end{proof}

The second strategy we can employ to prove base change results is to explicitly identify descent data. Let us again explain what we mean by this using an example:

\begin{lem}
\label{lem:basechange-nygaardsyn}
Let $X$ be a bounded $p$-adic formal scheme which is $p$-quasisyntomic and qcqs. Then the cartesian diagram
\begin{equation*}
\begin{tikzcd}
X^\N\ar[r, "j_\N"]\ar[d, "\pi_{X^\N}", swap] & X^\Syn\ar[d, "\pi_{X^\Syn}"] \\
\Z_p^\N\ar[r, "j_\N"] & \Z_p^\Syn
\end{tikzcd}
\end{equation*}
satisfies base change for bounded below complexes. 
\end{lem}
\begin{proof}
\EDIT{Given any $F\in\D^+(X^\Syn)$,} by base change for the cartesian diagrams
\begin{equation*}
\begin{tikzcd}
X^\prism\ar[r, "j_\dR"]\ar[d] & X^\N\ar[d] \\
\Z_p^\prism\ar[r, "j_\dR"] & \Z_p^\N
\end{tikzcd}\;,\hspace{0.5cm}
\begin{tikzcd}
X^\prism\ar[r, "j_\HT"]\ar[d] & X^\N\ar[d] \\
\Z_p^\prism\ar[r, "j_\HT"] & \Z_p^\N
\end{tikzcd}\;,
\end{equation*}
see \cref{cor:basechange-nygaardlocclosed}, \EDIT{we have
\begin{equation*}
j_\dR^*\pi_{X^\N, *}j_\N^*F\cong \pi_{X^\prism, *}j_\dR^*j_\N^*F\cong \pi_{X^\prism, *}j_\HT^*j_\N^*F\cong j_\HT^*\pi_{X^\N, *}j_\N^*F\;,
\end{equation*}
where $j_\dR^*j_\N^*F\cong j_\HT^*j_\N^*F$ is due to the fact that $j_\N^*F$ is pulled back from $X^\Syn$. Thus,} we see that $\pi_{X^\N, *}j_\N^* F$ is equipped with a canonical descent datum to $\Z_p^\Syn$ and this descent datum corresponds to the complex $\pi_{X^\Syn, *} F$.
\end{proof}

The only base change claim which cannot be readily proved with either of the two techniques above is the following:

\begin{lem}
\label{lem:basechange-nygaarddr}
Let $X$ be a bounded $p$-adic formal scheme which is smooth and qcqs. Then the cartesian diagram
\begin{equation*}
\begin{tikzcd}
X^{\dR, +}\ar[r, "i_{\dR, +}"]\ar[d, "\pi_{X^{\dR, +}}", swap] & X^\N\ar[d, "\pi_{X^\N}"] \\
\Z_p^{\dR, +}\ar[r, "i_{\dR, +}"] & \Z_p^\N
\end{tikzcd}
\end{equation*}
satisfies base change for bounded below complexes.
\end{lem}
\begin{proof}
To check that the base change map $i_{\dR, +}^*\pi_{X^\N, *} F\rightarrow \pi_{X^{\dR, +}, *}i_{\dR, +}^* F$ is an isomorphism for any $F\in\D^+(X^\N)$, we may reduce mod $p$ by $p$-completeness, i.e.\ we may pull back to the stack $(\A^1_-/\G_m)_{p=0}$. Moreover, as the square
\begin{equation*}
\begin{tikzcd}
(X^{\dR, +})_{p=0}\ar[r]\ar[d] & X^{\dR, +}\ar[d] \\
(\A^1_-/\G_m)_{p=0}\ar[r] & \A^1_-/\G_m
\end{tikzcd}
\end{equation*}
satisfies base change for bounded below complexes by the proof of \cref{cor:basechange-dr}, we are reduced to proving that the square
\begin{equation*}
\begin{tikzcd}
(X^{\dR, +})_{p=0}\ar[r]\ar[d] & X^\N\ar[d] \\
\A^1_-/\G_m\ar[r] & \Z_p^\N
\end{tikzcd}
\end{equation*}
satisfies base change for bounded below complexes; note that we have now simply written $\A^1_-/\G_m$ instead of $(\A^1_-/\G_m)_{p=0}$ as any instance of the stack $\A^1_-/\G_m$ will be over $\F_p$ from now on. However, note that the above square admits a factorisation as 
\begin{equation*}
\begin{tikzcd}
(X^{\dR, +})_{p=0}\ar[r]\ar[d] & X^\N_{\dR, +}\ar[r]\ar[d] & X^\N\ar[d] \\
\A^1_-/\G_m\ar[r] & \Z_{p, \dR, +}^\N\ar[r] & \Z_p^\N\nospacepunct{\;,}
\end{tikzcd}
\end{equation*}
so it suffices to show that both the left and the right square satisfy base change for bounded below complexes. For the right square, as $\Z_{p, \dR, +}^\N\rightarrow \Z_p^\N$ is a closed immersion \EDIT{by the description of the reduced locus}, this follows from an application of \cref{prop:basechange-locclosed}; for the left square, we recall that $\A^1_-/\G_m\rightarrow\Z_{p, \dR, +}$ and hence also $(X^{\dR, +})_{p=0}\rightarrow X^\N_{\dR, +}$ is a $\cal{G}$-torsor, where $\cal{G}$ is the group scheme described at the end of Section \ref{sect:syntomification}, so here the claim follows by identifying descent data.
\end{proof}

\section{Finiteness of pushforwards along $X^\Syn\rightarrow\Z_p^\Syn$}
\label{sect:finiteness}

In this appendix, we will prove a finiteness result for pushforwards of perfect complexes along $X^\Syn\rightarrow\Z_p^\Syn$. Namely, the result we want to \EEDIT{establish} goes as follows:

\begin{prop}
\label{prop:finiteness-main}
Let $X$ be a $p$-adic formal scheme which is smooth and proper \EEDIT{over $\Spf\Z_p$}. For any perfect $F$-gauge $E\in\Perf(X^\Syn)$, the pushforward $\pi_{X^\Syn, *} E$ is perfect.
\end{prop}

\EDIT{
Before we can proceed to the proof, we first need to establish the following two lemmas:
}

\begin{lem}
\label{lem:finiteness-perfa1gm}
\EEDIT{We work over $\F_p$.}
\begin{enumerate}[label=(\roman*)]
\EEDIT{
\item A complex $E\in\D(B\G_m)$ is perfect if and only if it corresponds to a graded complex consisting of only finitely many nonzero graded pieces, which are all bounded complexes of finite-dimensional $\F_p$-vector spaces.
\item For a complex $E\in\D(\A^1_-/\G_m)$, the following are equivalent:}
\begin{enumerate}[label=(\arabic*)]
\EEDIT{
\item $E$ is perfect.
\item $E$ is $t$-complete and its pullback to $B\G_m$ is perfect (recall that $t$ denotes the coordinate on $\A^1$).
\item Under the Rees equivalence, $E$ corresponds to a finite filtration $\Fil^\bullet V$ such that all graded pieces $\gr^\bullet V$ are bounded complexes of finite-dimensional $\F_p$-vector spaces; here, the filtration $\Fil^\bullet V$ being finite means that $\Fil^\bullet V=0$ for $\bullet\gg 0$ and that the transition maps $\Fil^\bullet V\rightarrow \Fil^{\bullet-1} V$ are isomorphisms for $\bullet\ll 0$.
}
\end{enumerate}
\end{enumerate}
\end{lem}
\begin{proof}
\EEDIT{
Part (i) is clear as perfectness may be checked after pullback along $\Spec\F_p\rightarrow B\G_m$. For part (ii), first observe that, if $E\in\D(\A^1_-/\G_m)$ is perfect, then its pullback to $B\G_m$ must be perfect as well and \cite[Rem.\ 2.2.7]{FGauges} shows that $E$ is $t$-complete; hence, (1) $\Rightarrow$ (2). The implication (2) $\Rightarrow$ (3) is clear by (i) as the Rees equivalence matches $t$-completeness with completeness as a filtered object. Finally, to prove (3) $\Rightarrow$ (1), note that (3) ensures that the pullback of $E$ to $\A^1$ is a bounded complex of finitely generated $\F_p[t]$-modules -- however, as $\F_p[t]$ is regular, this is the same as being perfect.
}
\end{proof}

\begin{lem}
\label{lem:finiteness-pushforwardcomplete}
\EDIT{
Let $X$ be a \EEDIT{smooth qcqs} $p$-adic formal scheme. Then for any perfect complex $E\in\Perf((X^{\dR, +})_{p=0})$ on $(X^{\dR, +})_{p=0}$, the pushforward $\pi_{X^{\dR, +}, *}E\in\D(\A^1_-/\G_m)$ is $t$-complete (where we recall that $t$ denotes the coordinate on $\A^1$).
}
\end{lem}
\begin{proof}
\EDIT{We employ a similar strategy to the proof of \cite[Thm.\ 2.3.6]{FGauges}: One first observes that if $X$ is written as a colimit of a finite diagram $U^\bullet$ of affine open subschemes, then $X^{\dR, +}$ is the colimit of the diagram $U^{\bullet, \dR, +}$ and thus, by smoothness and the qcqs assumption, we may assume that \EEDIT{$X=\Spf R$ is affine and that} there is an étale map \EEDIT{$\Spf R\rightarrow \A^n$}. Then the diagram
\begin{equation*}
\EEDIT{
\begin{tikzcd}[ampersand replacement=\&]
\Spf R\times\A^1_-/\G_m\ar[r]\ar[d] \& R^{\dR, +}\ar[d] \\
\A^n\times\A^1_-/\G_m\ar[r] \& (\A^n)^{\dR, +}
\end{tikzcd}
}
\end{equation*}
is cartesian and the bottom map is a torsor for the group scheme \EEDIT{${\V(\O(1))^\sharp}^n$} over $\A^1_-/\G_m$, hence the same is true for the top row. This means that the map \EEDIT{$(R^{\dR, +})_{p=0}\rightarrow\A^1_-/\G_m$} admits a factorisation
\begin{equation}
\label{eq:finiteness-factorpidr+}
\EEDIT{
(R^{\dR, +})_{p=0}\cong (\Spec R/p\times\A^1_-/\G_m)/{\V(\O(1))^\sharp}^n\rightarrow B_{\A^1_-/\G_m}{\V(\O(1))^\sharp}^n\rightarrow \A^1_-/\G_m
}
\end{equation}
and hence the pushforward of a quasi-coherent complex along this map admits a similar description to the one in \cref{lem:syntomicetale-vesharpreps}: Namely, \EEDIT{we can identify} a quasi-coherent complex $E$ on \EEDIT{$(R^{\dR, +})_{p=0}$} with a complex $V$ \EEDIT{of $R/p$-modules} equipped with a decreasing filtration $\Fil^\bullet V$ and a flat connection \EEDIT{$\nabla: V\rightarrow V\tensor\Omega_{(R/p)/\F_p}^1$} having locally nilpotent $p$-curvature \EEDIT{such that} the restriction of $\nabla$ to $\Fil^i V$ comes with a factorisation through \EEDIT{$\Fil^{i-1} V\tensor\Omega_{(R/p)/\F_p}^1$} along the lines of \cref{rem:fildrstack-vect}. \EEDIT{Now observing that $\Omega_{(R/p)/\F_p}^1\cong (R/p)^n$ since $\Spf R$ is étale over $\A^n$ and using the $\G_m$-equivariant version of \cref{lem:syntomicetale-vesharpreps}, we find that $\pi_{R^{\dR, +}, *}E$} identifies with the filtered complex whose degree $i$ term is given by
\begin{equation}
\label{eq:finiteness-filcohdr+}
\Tot(\Fil^i V\xrightarrow{\nabla} \Fil^{i-1} V\tensor\Omega_{(R/p)/\F_p}^1\xrightarrow{\nabla} \Fil^{i-2} V\tensor\Omega_{(R/p)/\F_p}^2\xrightarrow{\nabla}\dots)
\end{equation}
\EEDIT{by virtue of the factorisation (\ref{eq:finiteness-factorpidr+}).}
However, if $E$ is perfect, then the filtration $\Fil^\bullet V$ is \EEDIT{finite by \cref{lem:finiteness-perfa1gm}} and hence also the filtration given by (\ref{eq:finiteness-filcohdr+}) is \EEDIT{finite and, in particular, complete}, so we are done as the Rees equivalence matches completeness as a filtered object with $t$-completeness.
}
\end{proof}

\begin{proof}[Proof of \cref{prop:finiteness-main}]
\EDIT{Recalling from the proof of \cref{prop:syntomicetale-filtrationetale} that $v_1$ is topologically nilpotent on $\Z_p^\Syn$}, we may check perfectness after pulling back to the reduced locus $\Z_{p, \red}^\Syn$ by $(p, v_1)$-completeness; as the natural map $\Z_{p, \red}^\N\rightarrow\Z_{p, \red}^\Syn$ is a cover, we may even pull back to $\Z_{p, \red}^\N$. By the gluing description of $\Z_{p, \red}^\N$, we first observe that the reasoning of \cite[Fn. 44]{FGauges} shows that it suffices to check perfectness after pullback to $\Z_{p, \HT, c}^\N$ and $\Z_{p, \dR, +}^\N$. Finally, recall that there are covers $\A^1_-/\G_m\rightarrow\Z_{p, \dR, +}$ and $\A^1_+/\G_m\rightarrow\Z_{p, \HT, c}$, so we may check perfectness after pullback along those.
\EEDIT{Using \cref{lem:finiteness-perfa1gm} and} some base change results analogous to the ones from Appendix \ref{sect:basechange}, we are thus reduced to showing the following assertions:
\begin{enumerate}[label=(\alph*)]
\item Pushforward along $(X^\Hod)_{p=0}\rightarrow B\G_m$ preserves perfect complexes;
\item pushforward along $(X^{\dHod, c})_{p=0}\rightarrow \A^1_+/\G_m$ carries perfect complexes to $t$-complete objects;
\item pushforward along $(X^{\dR, +})_{p=0}\rightarrow \A^1_-/\G_m$ carries perfect complexes to $t$-complete objects.
\end{enumerate}
For (a), we observe that this follows immediately from \EEDIT{\cref{lem:finiteness-perfa1gm}} together with \cref{prop:syntomicetale-cohomologyhod}; here, we make use of the properness assumption in order to conclude that cohomology on $X_{p=0}$ preserves perfect complexes. Moreover, (b) follows immediately by the same reasoning as in the proof of \cite[Lem.\ 3.3.4]{NygaardHodge}. Finally, (c) is exactly \cref{lem:finiteness-pushforwardcomplete}.
\end{proof}

\bibliographystyle{amsalpha}
\bibliography{References}
\end{document}